\documentclass[10pt, a4paper, leqno]{article}
\usepackage[utf8]{inputenc} 
\usepackage[english]{babel} 
\usepackage[usenames, dvipsnames]{xcolor}
\usepackage[a4paper, tmargin=1.0in,bmargin=1.0in, lmargin=0.7in, rmargin=0.7in]{geometry} 
\usepackage{bm} 
\usepackage{amsmath} 
\usepackage{amsthm}
\usepackage{dsfont}
\usepackage{amssymb}
\usepackage{cases} 
\usepackage{tikz} 
\usepackage{stmaryrd} 
\usepackage{multirow} 
\usepackage{lipsum} 
\usepackage{wrapfig} 

\newenvironment{proofWithoutProof}
{
}
{\null\hfill\qedsymbol
}

\usepackage{empheq}

\usepackage{titlesec}
\titleformat{\section}
  {\normalfont\sffamily\Large\bfseries\color{black}}
  {\thesection}{1em}{}
\titleformat{\subsection}
  {\normalfont\sffamily\large\bfseries\color{black}}
  {\thesubsection}{1em}{}
\titleformat{\subsubsection}
  {\normalfont\sffamily\bfseries\color{black}}
  {\thesubsubsection}{1em}{}

\usepackage{placeins} 

\usepackage{environ}

\usepackage[colorlinks=true,linkcolor=OrangeRed,citecolor=RoyalBlue]{hyperref} 
\usepackage{cleveref} 
\Crefname{proposition}{Prop.}{Props.} 
\Crefname{theorem}{Thm.}{Thms.}
\Crefname{appendix}{App.}{Apps.}
\Crefname{definition}{Def.}{Defs.}
\Crefname{remark}{Rem.}{Rems.}
\Crefname{figure}{Fig.}{Figs.}

\theoremstyle{plain}
\newtheorem{proposition}{Proposition}%
\newtheorem{conjecture}{Conjecture}%
\newtheorem{lemma}{Lemma}%
\newtheorem{theorem}{Theorem}%
\newtheorem*{theorem*}{Theorem}
\theoremstyle{definition}
\newtheorem{definition}{Definition}%
\theoremstyle{remark}
\newtheorem{remark}{Remark}%
\newcommand{\lbm}{LBM}
\newcommand{\fd}{FD}
\newcommand{\adHoc}{\emph{ad hoc}}
\newcommand{\confer}{\emph{cf.}}
\newcommand{\strong}[1]{\emph{#1}}

\newcommand{\vectorial}[1]{\bm{#1}}
\newcommand{\matricial}[1]{\bm{#1}}
\newcommand{\solutionCauchyProblem}{u}
\newcommand{\timeVariable}{t}
\newcommand{\spaceVariable}{x}
\newcommand{\initial}{\circ}
\newcommand{\solutionCauchyProblemInitial}{\solutionCauchyProblem^{\initial}}
\newcommand{\finalTime}{T}
\newcommand{\reals}{\mathbb{R}}
\newcommand{\numberConservationLaws}{M}
\newcommand{\naturals}{\mathbb{N}}
\newcommand{\naturalsWithoutZero}{\naturals^{*}}
\newcommand{\domainLength}{L}
\newcommand{\flux}{\varphi}
\newcommand{\cKClass}[1]{C^{#1}}
\newcommand{\advectionVelocity}{V}
\newcommand{\boundaryDatumInflow}{g}
\newcommand{\spaceGridPoint}[1]{\spaceVariable_{#1}}
\newcommand{\timeGridPoint}[1]{\timeVariable^{#1}}
\newcommand{\indexSpace}{j}
\newcommand{\numberSpacePoints}{J}
\newcommand{\indexTime}{n}
\newcommand{\spaceStep}{\Delta \spaceVariable}
\newcommand{\timeStep}{\Delta \timeVariable}
\newcommand{\definitionEquality}{:=}
\newcommand{\integerInterval}[2]{\llbracket #1, #2 \rrbracket}
\newcommand{\relatives}{\mathbb{Z}}
\newcommand{\latticeVelocity}{\lambda}
\newcommand{\distributionFunctionLetter}{f}
\newcommand{\discrete}[1]{\mathsf{#1}}
\newcommand{\distributionFunctionDiscrete}{\discrete{\distributionFunctionLetter}}
\newcommand{\collided}{\star}
\newcommand{\conservedMomentDiscrete}{\discrete{u}}
\newcommand{\relaxationParameter}{\omega}
\newcommand{\coefficientExtrapolation}{c}
\newcommand{\orderExtrapolation}{\sigma}
\newcommand{\orderExtrapolationBiased}{\overline{\orderExtrapolation}}
\newcommand{\sourceTermBoundaryDiscrete}{\discrete{S}}
\newcommand{\outflowExtrapolationTag}{O-BC-E}
\newcommand{\nonConservedMomentDiscrete}{\discrete{v}}
\newcommand{\atEquilibrium}{\textnormal{eq}}

\newcommand{\coefficientCombinationElimination}{\gamma}
\newcommand{\catalanTwoIndices}[2]{T(#1, #2)}
\newcommand{\testFunction}{\psi}
\newcommand{\bigO}[1]{O(#1)}
\newcommand{\laplaceTransformed}[1]{\tilde{#1}}

\newcommand{\complex}{\mathbb{C}}
\newcommand{\realPartLetter}{\text{Re}}
\newcommand{\realPart}[1]{\realPartLetter(#1)}
\newcommand{\imagPartLetter}{\text{Im}}
\newcommand{\imagPart}[1]{\imagPartLetter(#1)}
\newcommand{\timeShiftOperator}{z}
\newcommand{\courantNumber}{C}

\newcommand{\frequency}{\xi}

\newcommand{\fourierShift}{\kappa}
\newcommand{\solutionCharStable}{\fourierShift_-}
\newcommand{\solutionCharUnstable}{\fourierShift_+}
\newcommand{\coefficientStable}{A_-}
\newcommand{\coefficientUnstable}{A_+}

\newcommand{\reflectionCoefficientLetter}{{R}}
\newcommand{\reflectionCoefficientLetterOut}{\reflectionCoefficientLetter_{\text{out}}}
\newcommand{\reflectionCoefficientLetterIn}{\reflectionCoefficientLetter_{\text{in}}}
\newcommand{\commonTerm}{\Pi}
\newcommand{\schemeMatrixFD}{\matricial{E}_{2\numberSpacePoints}}
\newcommand{\schemeMatrixFDSize}[1]{\matricial{E}_{#1}}

\newcommand{\schemeMatrixFDToeplitz}{\matricial{E}^{\times}_{2\numberSpacePoints}}
\newcommand{\schemeMatrixFDCirculant}{\matricial{E}^{\circ}_{2\numberSpacePoints}}

\newcommand{\identityMatrix}[1]{\matricial{Id}_{#1}}
\newcommand{\nullMatrix}[1]{\matricial{0}_{#1}}
\newcommand{\matrixSpace}[2]{\mathcal{M}_{#1}(#2)}
\newcommand{\transpose}[1]{#1^{\intercal}}
\newcommand{\twoStepSolutionVector}{\vectorial{\discrete{U}}}
\newcommand{\indexRow}{i}
\newcommand{\indexColumn}{j}
\newcommand{\determinant}{\textnormal{det}}
\newcommand{\adjugate}{\textbf{\text{adj}}}
\newcommand{\trace}{\textnormal{tr}}
\newcommand{\twoVariablesSolutionVector}{\vectorial{\discrete{F}}}
\newcommand{\schemeMatrixLBM}{\matricial{E}_{\textnormal{LBM}}}
\newcommand{\spectrum}{\textnormal{sp}}
\newcommand{\canonicalBasisVector}[1]{\vectorial{e}_{#1}}
\newcommand{\differential}{\text{d}}
\newcommand{\sign}{\text{sgn}}
\newcommand{\targetEigenvalue}{\timeShiftOperator^{\odot}}
\newcommand{\targetFourier}{\fourierShift^{\odot}}

\newcommand{\ball}[2]{{B}_{#1}(#2)}
\newcommand{\chebyshevPolynomialSecondKind}[1]{U_{#1}}
\newcommand{\schemeMatrixLBMBlock}[1]{\schemeMatrixLBM^{#1}}
\newcommand{\coefficientOutflowFDZero}{\alpha}
\newcommand{\coefficientOutflowFDMinusOne}{\beta}
\newcommand{\numberCoefficientsOutFlowFDZero}{\numberSpacePoints_{\text{out}}^{\coefficientOutflowFDZero}}
\newcommand{\numberCoefficientsOutFlowFDMinusOne}{\numberSpacePoints_{\text{out}}^{\coefficientOutflowFDMinusOne}}
\newcommand{\schemeMatrixFDBlockZero}{\matricial{A}_{\numberSpacePoints}}
\newcommand{\schemeMatrixFDBlockMinusOne}{\matricial{B}_{\numberSpacePoints}}
\newcommand{\schemeMatrixFDToeplitzBlockZero}{\matricial{A}^{{\times}}_{\numberSpacePoints}}
\newcommand{\schemeMatrixFDToeplitzBlockMinusOne}{\matricial{B}^{{\times}}_{\numberSpacePoints}}
\newcommand{\schemeMatrixFDBlockZeroEntry}{a}
\newcommand{\schemeMatrixFDBlockMinusOneEntry}{b}
\newcommand{\parturbationBoundaryOutflow}{\vectorial{b}_{\text{out}}}
\newcommand{\parturbationBoundaryInflow}{\vectorial{b}_{\text{in}}}

\newcommand{\modifiedTimeShiftOperator}{\eta}
\newcommand{\curlyMapsTo}{\mapstochar\mathrel{\mspace{0.45mu}}\leadsto}
\newcommand{\asymptoticSpectrumLetter}{\textnormal{asp}}
\newcommand{\asymptoticSpectrum}{\underset{\numberSpacePoints\to+\infty}{\asymptoticSpectrumLetter}}
\newcommand{\groupVelocity}{V_{\text{g}}}

\usepackage{pifont}

\providecommand{\keywords}[1]{\textbf{\textit{Keywords: }} \emph{#1}}
\providecommand{\amsCat}[1]{\textbf{\textit{MSC: }} #1}

\numberwithin{equation}{section}

\title{{Consistency and stability of boundary conditions for a two-velocities lattice Boltzmann scheme}}

\author{\textsc{Thomas Bellotti} \\ Université Paris-Saclay, CNRS, CentraleSupélec, Laboratoire EM2C, 91190, Gif-sur-Yvette, France\footnote{Past affiliation: IRMA, Université de Strasbourg, 67000 Strasbourg, France.}}

\ifpdf
\hypersetup{
  pdftitle={Consistency and stability of boundary conditions for a two-velocities lattice Boltzmann scheme},
  pdfauthor={T. Bellotti}
}
\fi

\begin{document}

\maketitle


\begin{abstract}
    We theoretically explore boundary conditions for lattice Boltzmann methods, focusing on a toy two-velocities scheme to tackle a linear one-dimensional advection equation. By mapping lattice Boltzmann schemes to Finite Difference schemes, we facilitate rigorous consistency and stability analyses. We develop kinetic boundary conditions for inflows and outflows, highlighting the trade-off between accuracy and stability, which we successfully overcome. 
    Consistency analysis relies on modified equations, whereas stability is assessed using GKS (Gustafsson, Kreiss, and Sundström) theory and---when this approach fails on coarse meshes---spectral and pseudo-spectral analyses of the scheme's matrix that explain effects germane to low resolutions.
\end{abstract}

\keywords{boundary conditions, lattice Boltzmann, Finite Difference, GKS, reflection coefficient, quasi-Toeplitz}\\
\amsCat{65M06, 65M12, 65N12}

\section*{Introduction}

So far, like numerous topics concerning lattice Boltzmann (in short, LBM) schemes, the study of boundary conditions has been strongly polarized towards applications, focusing on multidimensional problems, curved boundaries \cite{bouzidi2001momentum}, and elaborated schemes.
For example, theoretical stability has only been studied for specific systems and boundary conditions \cite{junk2009convergence, junk2008convergence} using the well-known ``stability structure''. 
Moreover, difficulties in using this technique to analyze the problem we address in this paper have been reported \cite{rheinlander2010stability}.
Believing that it is now time to understand things theoretically, which is increasingly difficult with the complexity of the schemes, we focus on the simplest \lbm{} method available: a two-velocities scheme.
On this numerical method, we aim at assessing consistency and stability of boundary conditions.
We find that the answer can pass through \strong{corresponding Finite Difference} (in short, FD) schemes on the variable of interest, turning unknowns in \lbm{} schemes into time-steps for \fd{}s.
In our context, the \lbm{} scheme features two unknowns $\conservedMomentDiscrete$ and $\nonConservedMomentDiscrete$, where only $\conservedMomentDiscrete \approx \solutionCauchyProblem$, with $\solutionCauchyProblem$ the solution of the partial differential equation, and $\nonConservedMomentDiscrete$ is a merely numerical unknown.
The scheme acts as $(\conservedMomentDiscrete^{\indexTime}, \nonConservedMomentDiscrete^{\indexTime}) \curlyMapsTo (\conservedMomentDiscrete^{\indexTime + 1}, \nonConservedMomentDiscrete^{\indexTime + 1})$, whereas its corresponding \fd{} scheme proceeds as $(\conservedMomentDiscrete^{\indexTime-1}, \conservedMomentDiscrete^{\indexTime}) \curlyMapsTo \conservedMomentDiscrete^{\indexTime + 1}$.
The latter schemes are called ``corresponding'' because if $\conservedMomentDiscrete^{\indexTime}$ is the solution of the \lbm{}, it also exactly fulfills the \fd{} scheme.

Our previous works \cite{bellotti2022finite, bellotti2023truncation} have shown a systematic way of turning \lbm{} schemes into \fd{} ones on unbounded domains or bounded domains with periodic boundary conditions. 
This secures a rigorous framework for the numerical analysis of \lbm{} methods.
With non-trivial boundary conditions, we lack this methodical path, due to the loss of space invariance stemming from the boundaries. 
The development of such a transformation in a general setting is beyond the scope of the present paper, which however shows that---whenever it can be explicitly constructed by \adHoc{} computations---\fd{} counterparts are indeed a powerful tool to understand boundary conditions imposed at the \lbm{} level.

Conversely, the main points of the present work are the following.
\begin{itemize}
    \item We develop \strong{kinetic boundary conditions} to handle inflows and outflows.
    The word ``kinetic'' must be understood as ``without modifying the \lbm{} algorithm'', \emph{i.e.} enforced during the transport phase using distribution functions entering the domain.
    \item We show that a \strong{compromise} exists between \strong{accuracy} and \strong{stability}.
    Still, we propose corrections based on \strong{boundary source terms} to recover the needed accuracy while retaining stability.

    \item We study stability blending several approaches.
    \begin{itemize}
        \item First, through the \strong{GKS (Gustafsson, Kreiss, and Sundstr\"om) analysis} \cite{gustafsson1972stability, trefethen1984instability}. However, it loses its predictive power on coarse meshes, \confer{} \cite{beam1982stability}, and stable boundary conditions may turn into unstable ones and \strong{viceversa}, due to interactions between boundaries.
        \item In these cases, we inspect the spectrum of the \strong{matrix associated with the scheme} \cite{vilar2015development}. 
        Furthermore, we link the order of poles in the \strong{reflection coefficient} \cite{trefethen1984instability}---a GKS notion---with the number of eigenvalues of the scheme matrix  tending to isolated points as the number of degrees of freedom increases.
    We also perform analyses plotting pseudo-spectra \cite{doi:10.2514/6.1997-1939, trefethen2005spectra}.
    \end{itemize}
    
\end{itemize}

We focus on the approximation of  $\vectorial{\solutionCauchyProblem} = \vectorial{\solutionCauchyProblem}(\timeVariable, \spaceVariable) \in \reals^{\numberConservationLaws}$, the solution of the 1D system of $\numberConservationLaws \in \naturalsWithoutZero$ \strong{conservation laws}
\begin{subnumcases}{}
    \partial_{\timeVariable}\vectorial{\solutionCauchyProblem}(\timeVariable, \spaceVariable) + \partial_{\spaceVariable}(\vectorial{\flux}(\vectorial{\solutionCauchyProblem}(\timeVariable, \spaceVariable))) = 0 & $\timeVariable \in (0, \finalTime ],  \quad \spaceVariable \in (0, \domainLength)$, \label{eq:conservationLaw}
    \\
    \vectorial{\solutionCauchyProblem}(\timeVariable = 0, \spaceVariable) = \vectorial{\solutionCauchyProblem}^{\initial}(\spaceVariable) & $\spaceVariable \in (0, \domainLength)$, \label{eq:initialCondition} \\
    \textnormal{suitable boundary conditions on } \{0, \domainLength\} & $\timeVariable \in (0, \finalTime ]$.
\end{subnumcases}
Here, $\finalTime > 0$ is the final time, $\domainLength > 0$ the domain length, $\vectorial{\flux} : \reals^{\numberConservationLaws} \to \reals^{\numberConservationLaws}$ a flux of class $\cKClass{1}$, and $\vectorial{\solutionCauchyProblem}^{\initial} : (0, \domainLength) \to \reals^{\numberConservationLaws}$ the initial datum.
We mostly concentrate on the scalar case, thus $\numberConservationLaws = 1$, with linear flux $\flux(\solutionCauchyProblem) = \advectionVelocity \solutionCauchyProblem$.
When $\advectionVelocity < 0$, the problem under scrutiny reads
\begin{subnumcases}{}
    \partial_{\timeVariable}{\solutionCauchyProblem}(\timeVariable, \spaceVariable) + \advectionVelocity \partial_{\spaceVariable}{\solutionCauchyProblem}(\timeVariable, \spaceVariable) = 0 & $\timeVariable \in (0, \finalTime ],  \quad \spaceVariable \in (0, \domainLength)$, \label{eq:advection}
    \\
    {\solutionCauchyProblem}(\timeVariable = 0, \spaceVariable) = {\solutionCauchyProblem}^{\initial}(\spaceVariable) & $\spaceVariable \in (0, \domainLength)$, \label{eq:advectionInitial} \\
    \solutionCauchyProblem(\timeVariable, \spaceVariable = \domainLength) = \boundaryDatumInflow (\timeVariable) & $\timeVariable \in (0, \finalTime ]$,\label{eq:DirichletInflowCondition}
\end{subnumcases}
where $\boundaryDatumInflow : (0, \finalTime] \to \reals$ is the trace on the inflow point $\spaceVariable = \domainLength$.
As $\spaceVariable = 0$ is an outflow, no ``physical'' boundary condition is needed here.

To ``give a flavour'' of the results proved in the paper, the compromise between accuracy and stability for a scalar linear problem is informally stated as follows.
\begin{theorem*}
    Consider a two-velocities \lbm{} scheme tackling \eqref{eq:advection}, \eqref{eq:advectionInitial}, and \eqref{eq:DirichletInflowCondition} while $\advectionVelocity<0$, second-order accurate under periodic boundary conditions.
    At the boundary, utilize an exact anti-bounce-back condition on the inflow and an extrapolation of order $\orderExtrapolation\in \naturalsWithoutZero$ on the missing distribution function on the outflow.
    Consider smooth initial data at equilibrium.
    Then
    \begin{center}\setlength{\tabcolsep}{3pt}
        {\renewcommand{\arraystretch}{1.1}
        \begin{tabular}{|c|ccc|}
            \hline
            & \textsc{Consistency} & \textsc{Stability} & \textsc{Order of converg.} ($L^2$)\\
            & \Cref{prop:outflowExtrapolationModifiedEquation} & \Cref{prop:stabInstGKS} & \cite[Theorem 1]{coulombel2015leray} \\
            \hline
            \hline
            \multirow{4}{*}{$\orderExtrapolation = 1$} & Trunc. error initially at the outflow: $\bigO{\spaceStep}$ & \multirow{4}{*}{GKS-stable} & \multirow{4}{*}{$\bigO{\spaceStep^{3/2}}$} \\
            & Trunc. error initially in the bulk: $\bigO{\spaceStep^2}$ & & \\
            & Trunc. error eventually in time at the outflow: $\bigO{\spaceStep^2}$ & &\\
            & Trunc. error eventually in time in the bulk: $\bigO{\spaceStep^3}$ & & \\
            \hline
            \multirow{4}{*}{$\orderExtrapolation \geq 2$} & Trunc. error initially at the outflow: $\bigO{\spaceStep^2}$ & \multirow{4}{*}{GKS-unstable} & \multirow{4}{*}{$\bigO{\spaceStep^{2}}$} \\
            & Trunc. error initially in the bulk: $\bigO{\spaceStep^2}$ & & \\
            & Trunc. error eventually in time at the outflow: $\bigO{\spaceStep^2}$ & &\\
            & Trunc. error eventually in time in the bulk: $\bigO{\spaceStep^3}$ & & \\
            \hline
        \end{tabular}}
    \end{center}
\end{theorem*}

We also informally state \Cref{thm:countingThroughReflection}, linking the order of poles of the reflection coefficient to the number of eigenvalues tending to isolated points.
\begin{theorem*}
    Under suitable assumptions, the number $\orderExtrapolation$ of eigenvalues clustering around a limit isolated point $\targetEigenvalue \in \complex$ associated to the outflow boundary condition equals the order of the pole of the reflection coefficient $\reflectionCoefficientLetterOut$ of the outflow boundary condition at $\targetEigenvalue$ (\emph{i.e.} $\reflectionCoefficientLetterOut(\timeShiftOperator)\sim (\timeShiftOperator-\targetEigenvalue)^{-\orderExtrapolation}$ as $\timeShiftOperator\to\targetEigenvalue$).
\end{theorem*}
In the case $|\targetEigenvalue|=1$, the previous theorem allows to understand instabilities linked to the outflow boundary condition, which grow polynomially with order $\orderExtrapolation - 1$. Indeed, the order of the pole of the reflection coefficient is very easy to compute, compared to a full study of the spectrum.

The paper is structured as follows.
The numerical scheme and numerical boundary conditions are presented in \Cref{sec:scheme}.
The \lbm{} scheme is then turned into its corresponding \fd{} scheme in \Cref{sec:correspondingFD}.
This allows to rigorously study consistency, as presented in \Cref{sec:consistency}, and stability, \confer{} \Cref{sec:stability}.
Conclusions and perspectives are proposed in \Cref{sec:conclusions}.

\section{Two-velocities lattice Boltzmann scheme}\label{sec:scheme}

This section describes the \lbm{} scheme.
We introduce the space-time discretization in \Cref{sec:spaceTimeDiscretization}.
\Cref{sec:bulkLBM} treats the scheme used in the bulk of the domain, while \Cref{sec:boundaryLBM} present strategies adopted on the boundary points.
Finally, \Cref{sec:initializationLBM} is devoted to the initalisation of the numerical scheme.

\subsection{Space and time discretization}\label{sec:spaceTimeDiscretization}

\begin{figure} 
    \begin{center}
        \begin{tikzpicture}[scale=0.82]
            \fill[red,opacity=0.5] (0, 0) rectangle (8, 3);
            \shade[bottom color=red, top color=white, opacity=0.5] (0,3) rectangle (8,3.8);

            \draw[->] (-2.5,0) -- (10,0) node[right] {$\spaceVariable$};
            \draw[->] (0,0) -- (0,4) node[above] {$\timeVariable$};
            \foreach \x in {-1, 0,...,3, 5, 6,..., 9}
                \foreach \y in {0,0.5,...,3}
                    \fill (\x,\y) circle (2pt);

            \node at (4, 1.5) {$\cdots$};
            \node[below] at (4, 0) {$\cdots$};
            \node[below] at (3, 0) {$\cdots$};
            \node[below] at (5, 0) {$\cdots$};

            \foreach \x in {-1, 0, 1, 2}
                \node[below] at (\x,0) {$\spaceGridPoint{\x}$};
            \node[below] at (0, -1em) {$0$};
            \node[below] at (8, -1em) {$\domainLength$};
            \foreach \x in {-3, -2, -1}
                \node[below] at (9+\x,0) {$\spaceGridPoint{\numberSpacePoints \x}$};
            \node[below] at (9,0) {$\spaceGridPoint{\numberSpacePoints}$};
            \foreach \y in {1, 2, ..., 6}
                \node[left] at (0, .5*\y) {$\timeGridPoint{\y}$};

            \draw[|<->|] (9.2,1) -- (9.2,1.5) node[midway, right] {$\timeStep$};
            \draw[|<->|] (6,3.2) -- (7,3.2) node[midway, above] {$\spaceStep$};
        \end{tikzpicture}
    \end{center}\caption{\label{fig:mesh}Computational mesh (dots) on $(0, \finalTime]\times (0, \domainLength)$ (in red).}
\end{figure}

The  spatial domain $(0, \domainLength)$ is discretized with $\numberSpacePoints \in \naturalsWithoutZero$ grid-points $\spaceGridPoint{\indexSpace} \definitionEquality \indexSpace \spaceStep$, with $\indexSpace \in \integerInterval{0}{\numberSpacePoints - 1}$ and $\spaceStep \definitionEquality \domainLength/(\numberSpacePoints - 1)$.
We include the boundary points $\spaceVariable = 0$ and $\domainLength$ in the computational domain \emph{via} $\spaceGridPoint{0} = 0$ and $\spaceGridPoint{\numberSpacePoints - 1} = \domainLength$.
Moreover, the definition of $\spaceGridPoint{\indexSpace}$ works for $\indexSpace \in \relatives$, and thus we can define ghost points outside $[0, \domainLength]$.

Concerning time discretization, we utilize grid-points $\timeGridPoint{\indexTime} \definitionEquality \indexTime\timeStep$ with $\indexTime \in \naturals$.
The time step $\timeStep$ is linked to the space step $\spaceStep$ by $\timeStep \definitionEquality \spaceStep/\latticeVelocity$ with $\latticeVelocity > 0$ called ``lattice velocity''.
This scaling between space and time discretization, known as ``acoustic scaling'', is relevant in this context where information---see \eqref{eq:conservationLaw}---travels at \strong{finite speed} and the numerical method is \strong{time-explicit}.
Therefore, CFL (Courant-Friedrichs-Lewy) restrictions are needed on $\latticeVelocity$ to ensure stability.
This whole spatio-temporal setting is depicted in \Cref{fig:mesh}.

\subsection{Bulk scheme}\label{sec:bulkLBM}

Starting from inside the domain, we utilize the so-called $\textnormal{D}_1\textnormal{Q}_2^{\numberConservationLaws}$ scheme \cite{graille2014approximation}, based on two distribution functions $\vectorial{\distributionFunctionDiscrete}^{+} \in \reals^{\numberConservationLaws}$ and $\vectorial{\distributionFunctionDiscrete}^{-} \in \reals^{\numberConservationLaws}$, each associated with a positively and negatively-moving unknown.
At each stage of the algorithm, we define $\vectorial{\conservedMomentDiscrete} \definitionEquality \vectorial{\distributionFunctionDiscrete}^{+} + \vectorial{\distributionFunctionDiscrete}^{-}$, and it reads:
\begin{align}
    \textnormal{Collision} \qquad \qquad &\vectorial{\distributionFunctionDiscrete}_{\indexSpace}^{\pm, \indexTime \collided} = (1-\relaxationParameter) \vectorial{\distributionFunctionDiscrete}_{\indexSpace}^{\pm, \indexTime} + \relaxationParameter \bigl ( \overbrace{\tfrac{1}{2} \vectorial{\conservedMomentDiscrete}_{\indexSpace}^{\indexTime} \pm  \tfrac{1}{2\latticeVelocity} \vectorial{\flux}(\vectorial{\conservedMomentDiscrete}_{\indexSpace}^{\indexTime})}^{\vectorial{\distributionFunctionDiscrete}^{\pm, \atEquilibrium}(\vectorial{\conservedMomentDiscrete}_{\indexSpace}^{\indexTime}) \definitionEquality}\bigr ), \quad \text{for}\quad\indexSpace \in \integerInterval{0}{\numberSpacePoints - 1}.\tag{C}\label{eq:collision}\\
    \textnormal{Transport} \qquad \qquad &
    \vectorial{\distributionFunctionDiscrete}_{\indexSpace}^{+, \indexTime + 1} = \vectorial{\distributionFunctionDiscrete}_{\indexSpace- 1}^{+, \indexTime \collided}, \quad \text{for} \quad \indexSpace \in \integerInterval{1}{\numberSpacePoints - 1}, \qquad \text{and}\qquad \vectorial{\distributionFunctionDiscrete}_{\indexSpace}^{-, \indexTime + 1} = \vectorial{\distributionFunctionDiscrete}_{\indexSpace+ 1}^{-, \indexTime \collided}, \quad \text{for}\quad\indexSpace \in \integerInterval{0}{\numberSpacePoints - 2}.
    \tag{T}\label{eq:transport}
\end{align}
The relaxation parameter $\relaxationParameter \in (0, 2]$.
Notice that $\vectorial{\conservedMomentDiscrete}_{\indexSpace}^{\indexTime}$ is conserved throughout the collision \eqref{eq:collision}, and has to be interpreted as an approximation of $\vectorial{\solutionCauchyProblem}(\timeGridPoint{\indexTime}, \spaceGridPoint{\indexSpace})$, when the latter is smooth.
The approximation is---besides boundary conditions---\strong{first-order accurate} in $\spaceStep$ when $\relaxationParameter \in (0, 2)$ and \strong{second-order accurate} for $\relaxationParameter = 2$, see \cite{graille2014approximation}.

\subsection{Boundary schemes}\label{sec:boundaryLBM}

From \eqref{eq:transport}, we see that the scheme is not yet defined on the boundary grid-points $\spaceGridPoint{0}$ (for $\vectorial{\distributionFunctionDiscrete}^{+}$) and $\spaceGridPoint{\numberSpacePoints-1}$ (for $\vectorial{\distributionFunctionDiscrete}^{-}$).
As stated in the Introduction, we do not change schemes at boundaries, boiling down to devise ``\strong{prepared}'' values $\vectorial{\distributionFunctionDiscrete}_{-1}^{+, \indexTime\collided}$ and $\vectorial{\distributionFunctionDiscrete}_{\numberSpacePoints}^{-, \indexTime\collided}$ in ghost cells after collision, \confer{} \Cref{fig:mesh}.
For the sake of presentation, we consider the scalar case $\numberConservationLaws = 1$ and assume that we need to set an inflow condition at $\spaceVariable = \domainLength$ and no boundary condition (outflow) at $\spaceVariable = 0$.

\subsubsection{Inflow boundary condition}

The inflow boundary condition is of Dirichlet type \eqref{eq:DirichletInflowCondition}.
Yet, it cannot be directly enforced on $\conservedMomentDiscrete$, but---according to out policy---we make good use of the ghost value $\distributionFunctionDiscrete_{\numberSpacePoints}^{-, \indexTime\collided}$ to achieve the desired result.
We follow the approach introduced in \cite[Chapter 6]{helie2023schema} and consider
\begin{equation}\tag{I-BC}\label{eq:inflowConditionScheme}
    \distributionFunctionDiscrete_{\numberSpacePoints - 1}^{-, \indexTime + 1} = \distributionFunctionDiscrete_{\numberSpacePoints}^{-, \indexTime \collided} = - \distributionFunctionDiscrete_{\numberSpacePoints - 2}^{+, \indexTime \collided} + \boundaryDatumInflow(\timeGridPoint{\indexTime + 1}).
\end{equation}
The aim of \eqref{eq:inflowConditionScheme} being to preserve overall second-order accuracy when $\relaxationParameter = 2$, it is \strong{different} from the standard ``anti-bounce-back'', see \cite{dubois2020anti}, which would be $\distributionFunctionDiscrete_{\numberSpacePoints - 1}^{-, \indexTime + 1} = \distributionFunctionDiscrete_{\numberSpacePoints}^{-, \indexTime \collided} = - \distributionFunctionDiscrete_{\numberSpacePoints - 1}^{+, \indexTime \collided} + \boundaryDatumInflow(\timeGridPoint{\indexTime})$ and only first-order accurate.
Two essential ideas make \eqref{eq:inflowConditionScheme} suitable in our case.
First, information travels \strong{two cells at once}, since the ghost point $\spaceGridPoint{\numberSpacePoints}$ is populated using values defined at $\spaceGridPoint{\numberSpacePoints - 2}$. Second, the value of the trace is enforced ``\strong{in the future}'', using $\boundaryDatumInflow(\timeGridPoint{\indexTime + 1})$ instead of the time-marching approach, which would rather employ $\boundaryDatumInflow(\timeGridPoint{\indexTime})$.

\subsubsection{Outflow boundary conditions}

On the outflow, no physical boundary condition has to be enforced.
However, \lbm{} schemes call for \strong{numerical boundary conditions}.
The aim is twofold. 
On the one hand, the boundary condition should allow information to quit the domain.
On the other hand, it ought not produce spurious waves counter-propagating against physical ones that can lower the order of the method or---even worse---foster instabilities \cite{trefethen1984instability}.

The rationale is to make the scheme behave as much as \strong{no boundary were present}.
We therefore \strong{extrapolate} the missing ghost value $\distributionFunctionDiscrete_{-1}^{+, \indexTime\collided}$ from those inside the computational domain following \cite{goldberg1977boundary}, and employ 
\begin{equation}\label{eq:extrapolationBoundaryCondition}
    \distributionFunctionDiscrete_{0}^{+, \indexTime + 1} = \distributionFunctionDiscrete_{-1}^{+, \indexTime \collided} = \sum_{\indexSpace = 0}^{\orderExtrapolation - 1} \coefficientExtrapolation_{\indexSpace} \distributionFunctionDiscrete_{\indexSpace}^{+, \indexTime \collided} + \sourceTermBoundaryDiscrete_0^{\indexTime + 1} 
    \qquad \textnormal{with} \qquad \coefficientExtrapolation_{\indexSpace} \definitionEquality (-1)^{\indexSpace} \binom{\orderExtrapolation}{\indexSpace + 1}.\tag{\outflowExtrapolationTag$\orderExtrapolation$}
\end{equation}
In \eqref{eq:extrapolationBoundaryCondition}, $\orderExtrapolation \in \naturalsWithoutZero$ is the order of the extrapolation, and $\sourceTermBoundaryDiscrete_0^{\indexTime + 1}$ is a source term depending only on $\conservedMomentDiscrete^{\indexTime}, \conservedMomentDiscrete^{\indexTime-1}$, \emph{etc}.
Again, observe that---contrarily to \cite{coulombel2020neumann} and respecting our policy---extrapolations are not enforced on $\conservedMomentDiscrete$.
The low-order extrapolations read in this case:
\begin{align}
    \distributionFunctionDiscrete_{0}^{+, \indexTime + 1} &= \distributionFunctionDiscrete_{-1}^{+, \indexTime \collided} = \distributionFunctionDiscrete_{0}^{+, \indexTime \collided} + \sourceTermBoundaryDiscrete_0^{\indexTime + 1} \tag{\outflowExtrapolationTag1},\label{eq:papillon1} \\
    \distributionFunctionDiscrete_{0}^{+, \indexTime + 1} &= \distributionFunctionDiscrete_{-1}^{+, \indexTime \collided} = 2\distributionFunctionDiscrete_{0}^{+, \indexTime \collided} - \distributionFunctionDiscrete_{1}^{+, \indexTime \collided} + \sourceTermBoundaryDiscrete_0^{\indexTime + 1}, \tag{\outflowExtrapolationTag2}\label{eq:papillon2}\\
    \distributionFunctionDiscrete_{0}^{+, \indexTime + 1} &= \distributionFunctionDiscrete_{-1}^{+, \indexTime \collided} = 3\distributionFunctionDiscrete_{0}^{+, \indexTime \collided} - 3\distributionFunctionDiscrete_{1}^{+, \indexTime \collided} + \distributionFunctionDiscrete_{2}^{+, \indexTime \collided} + \sourceTermBoundaryDiscrete_0^{\indexTime + 1}. \tag{\outflowExtrapolationTag3}
\end{align}
Condition \eqref{eq:papillon2} looks similar but is different from the ones proposed by \cite{guo2002extrapolation, junk2008outflow}. In particular, the first paper uses extrapolations for the conserved quantities ($\conservedMomentDiscrete$ in our case) and then reconstructs equilibrium and non-equilibrium distribution functions from this. The second work extrapolates distribution functions from those at time $\timeGridPoint{\indexTime+1}$, thus after collision and transport, rather than at $\timeGridPoint{\indexTime}$ after collision only as we do.

Another possible boundary condition relies on the \lbm{} scheme recast on the conserved and non-conserved moments $\conservedMomentDiscrete$ and $\nonConservedMomentDiscrete \definitionEquality \latticeVelocity (\distributionFunctionDiscrete^+ - \distributionFunctionDiscrete^-)$.
This latter unknown has equilibrium  $\nonConservedMomentDiscrete^{\atEquilibrium}(\conservedMomentDiscrete) = \latticeVelocity (\distributionFunctionDiscrete^{+, \atEquilibrium} (\conservedMomentDiscrete)- \distributionFunctionDiscrete^{-, \atEquilibrium}(\conservedMomentDiscrete)) = \flux(\conservedMomentDiscrete)$.
We imagine to enforce---at the future time $\timeGridPoint{\indexTime + 1}$---a Neumann boundary condition on $\conservedMomentDiscrete$ using a first-order extrapolation, and set $\nonConservedMomentDiscrete$ at its equilibrium.
This gives $\tilde{\conservedMomentDiscrete}_0^{\indexTime + 1} = \conservedMomentDiscrete_1^{\indexTime + 1} = \distributionFunctionDiscrete_0^{+, \indexTime\collided} + \distributionFunctionDiscrete_2^{-, \indexTime\collided}$ and $\tilde{\nonConservedMomentDiscrete}_0^{\indexTime + 1} = \nonConservedMomentDiscrete^{\atEquilibrium} (\tilde{\conservedMomentDiscrete}_0^{\indexTime + 1}) = \nonConservedMomentDiscrete^{\atEquilibrium} (\conservedMomentDiscrete_1^{\indexTime + 1}) = \nonConservedMomentDiscrete^{\atEquilibrium} (\distributionFunctionDiscrete_0^{+, \indexTime\collided} + \distributionFunctionDiscrete_2^{-, \indexTime\collided})$, where the tilde is added to stress that at the end $\conservedMomentDiscrete_0^{\indexTime + 1}\neq \tilde{\conservedMomentDiscrete}_0^{\indexTime + 1}$ in general.
Otherwise said, these identities  are solely employed to devise a value for $\distributionFunctionDiscrete_{-1}^{+, \indexTime\collided}$, and do not change those of the distribution functions inside the computational domain: we have---rewriting on the positive distribution function and adding a source term
\begin{equation}\label{eq:conditionKinRod}
    \distributionFunctionDiscrete_{0}^{+, \indexTime + 1} = \distributionFunctionDiscrete_{-1}^{+, \indexTime \collided} = \tfrac{1}{2}\tilde{\conservedMomentDiscrete}_0^{\indexTime + 1} +\tfrac{1}{2\latticeVelocity} \tilde{\nonConservedMomentDiscrete}_0^{\indexTime + 1}  + \sourceTermBoundaryDiscrete_0^{\indexTime + 1} = \tfrac{1}{2}(\distributionFunctionDiscrete_0^{+, \indexTime\collided} + \distributionFunctionDiscrete_2^{-, \indexTime\collided}) + \tfrac{1}{2\latticeVelocity} \nonConservedMomentDiscrete^{\atEquilibrium} (\distributionFunctionDiscrete_0^{+, \indexTime\collided} + \distributionFunctionDiscrete_2^{-, \indexTime\collided}) + \sourceTermBoundaryDiscrete_0^{\indexTime + 1}.\tag{O-BC-F}
\end{equation}

\subsection{Initialization}\label{sec:initializationLBM}

Finally, to wholly define the numerical scheme, we link the initial distribution functions to the initial datum of the Cauchy problem \eqref{eq:initialCondition} given by $\vectorial{\solutionCauchyProblem}^{\initial}$.
Initializing \strong{at equilibrium} is the longstanding ``no-brainer'' choice:
\begin{equation}\tag{IE}\label{eq:initializationEquilibrium}
    \vectorial{\distributionFunctionDiscrete}_{\indexSpace}^{\pm, 0} = \tfrac{1}{2} \vectorial{\solutionCauchyProblem}^{\initial}(\spaceGridPoint{\indexSpace}) \pm \tfrac{1}{2\latticeVelocity}\vectorial{\flux} (\vectorial{\solutionCauchyProblem}^{\initial}(\spaceGridPoint{\indexSpace})), \qquad \indexSpace \in \integerInterval{0}{\numberSpacePoints - 1}.
\end{equation}
Though this choice has been proved compatible with second-order accuracy in the bulk \cite{bellotti2024initialisation}, see also \cite{van2009smooth, junk2015l2} for discussions on initial conditions for \lbm{} schemes, it is hitherto not clear  that it does not destroy second-order together with boundary conditions.

\section{Corresponding Finite Difference scheme}\label{sec:correspondingFD}

The context being now set, we transform the \lbm{} scheme into its \strong{corresponding \fd{} scheme} solely on $\conservedMomentDiscrete$.
From now on, only the scalar case $\numberConservationLaws = 1$ is considered, even though some results easily extend to $\numberConservationLaws > 1$.
The link between the discretized initial datum and the continuous one is $\conservedMomentDiscrete_{\indexSpace}^{0} = \solutionCauchyProblemInitial(\spaceGridPoint{\indexSpace})$.

\subsection{Bulk and inflow}

\begin{proposition}\label{prop:bulkAndInflow}
    Consider the \lbm{} scheme given by \eqref{eq:initializationEquilibrium}, \eqref{eq:collision}, \eqref{eq:transport}, the inflow boundary condition \eqref{eq:inflowConditionScheme}, and the outflow boundary condition \eqref{eq:extrapolationBoundaryCondition}.
    Then, its corresponding \fd{} scheme on $\conservedMomentDiscrete$ reads 
    \begin{align}
        \conservedMomentDiscrete_{\indexSpace}^1 = \tfrac{1}{2}(\conservedMomentDiscrete_{\indexSpace - 1}^0 + \conservedMomentDiscrete_{\indexSpace + 1}^0) + \tfrac{1}{2\latticeVelocity} (\flux(\conservedMomentDiscrete_{\indexSpace - 1}^0) - \flux(\conservedMomentDiscrete_{\indexSpace + 1}^0)), \qquad \conservedMomentDiscrete_{\numberSpacePoints - 1}^1 &= \boundaryDatumInflow(\timeGridPoint{1}), \label{eq:bulkFDInitial}\\
        \conservedMomentDiscrete_{\indexSpace}^{\indexTime + 1} = \tfrac{2-\relaxationParameter}{2}(\conservedMomentDiscrete_{\indexSpace - 1}^{\indexTime} + \conservedMomentDiscrete_{\indexSpace + 1}^{\indexTime}) + (\relaxationParameter - 1) \conservedMomentDiscrete_{\indexSpace}^{\indexTime - 1} + \tfrac{\relaxationParameter}{2\latticeVelocity} (\flux(\conservedMomentDiscrete_{\indexSpace - 1}^{\indexTime}) - \flux(\conservedMomentDiscrete_{\indexSpace + 1}^{\indexTime})), \qquad \conservedMomentDiscrete_{\numberSpacePoints - 1}^{\indexTime + 1} &= \boundaryDatumInflow(\timeGridPoint{\indexTime + 1}), \qquad \indexTime \in \naturalsWithoutZero, \label{eq:bulkFDScheme}
    \end{align}
    with $\indexSpace \in \integerInterval{1}{\numberSpacePoints - 2}$.
\end{proposition}
The proof of \Cref{prop:bulkAndInflow} is provided in \Cref{app:proof:bulkAndInflow}.
The bulk scheme at the first time step \eqref{eq:bulkFDInitial} is a Lax-Friedrichs scheme for \eqref{eq:conservationLaw}.
This comes from the initial datum at equilibrium \eqref{eq:initializationEquilibrium}.
The bulk scheme eventually in time \eqref{eq:bulkFDScheme} is a convex combination, when $\relaxationParameter \in [0, 1]$, between a leap-frog scheme for the wave equation $\partial_{\timeVariable\timeVariable} \solutionCauchyProblem - \latticeVelocity^2\partial_{\spaceVariable\spaceVariable}\solutionCauchyProblem = 0$ and a Lax-Friedrichs scheme for \eqref{eq:conservationLaw}, and---when $\relaxationParameter \in [1, 2]$---between a Lax-Friedrichs scheme for \eqref{eq:conservationLaw} and a leap-frog scheme for \eqref{eq:conservationLaw}.
The scheme at $\spaceGridPoint{\numberSpacePoints - 1}$ is exact for the inflow condition \eqref{eq:DirichletInflowCondition}.
\begin{remark}[Assumptions of \Cref{prop:bulkAndInflow}]
    Assumptions on the outflow boundary conditions appear in \Cref{prop:bulkAndInflow}.
    Moreover, the claim does not hold for \eqref{eq:conditionKinRod}.
    Both are due to the fact that the choice of outflow boundary condition can impact the numerical scheme in the bulk, that is, for $\indexSpace \in \integerInterval{1}{\numberSpacePoints - 2}$.
\end{remark}

\subsection{Outflow}

We first provide the corresponding \fd{} scheme for \eqref{eq:extrapolationBoundaryCondition}.
The proof is in \Cref{app:proof:outflowExtrapolation}.
\begin{proposition}\label{prop:outflowExtrapolation}
    Consider the \lbm{} scheme given by \eqref{eq:initializationEquilibrium}, \eqref{eq:collision}, \eqref{eq:transport}, and the outflow boundary condition \eqref{eq:extrapolationBoundaryCondition}.
    Then, its corresponding \fd{} scheme on $\conservedMomentDiscrete$ reads as follows.
    For the initial time-step:
    \begin{equation}\label{eq:initialSchemeBoundaryExtrapolation}
        \conservedMomentDiscrete_{0}^{1} = \tfrac{1}{2} \Bigl ( \sum_{\substack{\indexSpace = 0 \\ \indexSpace \neq 1}}^{\orderExtrapolation - 1} \coefficientExtrapolation_{\indexSpace} \conservedMomentDiscrete_{\indexSpace}^{0} + (\coefficientExtrapolation_1 + 1) \conservedMomentDiscrete_{1}^{0} \Bigr ) + \tfrac{1}{2\latticeVelocity} \Bigl ( \sum_{\substack{\indexSpace = 0 \\ \indexSpace \neq 1}}^{\orderExtrapolation - 1} \coefficientExtrapolation_{\indexSpace} \flux(\conservedMomentDiscrete_{\indexSpace}^{0}) + (\coefficientExtrapolation_1 - 1) \flux(\conservedMomentDiscrete_{1}^{0}) \Bigr )  + \sourceTermBoundaryDiscrete_0^{ 1}.
    \end{equation}
    Eventually in time, for $\indexTime \in \naturalsWithoutZero$.
    \begin{itemize}
        \item For $\orderExtrapolation = 1$, we have:
        \begin{equation}
            \conservedMomentDiscrete_0^{\indexTime + 1} = \tfrac{\relaxationParameter}{2} \conservedMomentDiscrete_0^{\indexTime} + \tfrac{2-\relaxationParameter}{2} \conservedMomentDiscrete_1^{\indexTime} + \tfrac{\relaxationParameter}{2\latticeVelocity} (\flux(\conservedMomentDiscrete_0^{\indexTime}) - \flux(\conservedMomentDiscrete_1^{\indexTime})) + \sourceTermBoundaryDiscrete_0^{\indexTime + 1} + (1-\relaxationParameter) \sourceTermBoundaryDiscrete_0^{\indexTime}. \label{eq:bulkBoundarySigma1}
        \end{equation}
        \item For $\orderExtrapolation \geq 2$, we have:
        \begin{multline}\label{eq:bulkBoundarySigmaGenericTwo}
            \conservedMomentDiscrete_{0}^{\indexTime +1} = \tfrac{1}{2} \Bigl ( \sum_{\substack{\indexSpace = 0 \\ \indexSpace \neq 1}}^{\orderExtrapolation - 1} (\coefficientExtrapolation_{\indexSpace} + (1-\relaxationParameter)\coefficientCombinationElimination_{\indexSpace})\conservedMomentDiscrete_{\indexSpace}^{\indexTime} + (\coefficientExtrapolation_1 + 1 + (1-\relaxationParameter)\coefficientCombinationElimination_1) \conservedMomentDiscrete_{1}^{\indexTime} \Bigr ) \\
            + \tfrac{\relaxationParameter - 1}{2} \Bigl ((\coefficientExtrapolation_0 + \coefficientCombinationElimination_0 ) \conservedMomentDiscrete_{1}^{\indexTime - 1} + (\coefficientExtrapolation_1 + \coefficientCombinationElimination_1 - 1) \conservedMomentDiscrete_{2}^{\indexTime - 1} 
            + \sum_{\indexSpace = 3}^{\orderExtrapolation} (\coefficientExtrapolation_{\indexSpace - 1} + \coefficientCombinationElimination_{\indexSpace - 1}) \conservedMomentDiscrete_{\indexSpace}^{\indexTime - 1}\Bigr ) \\
            + \tfrac{\relaxationParameter}{2\latticeVelocity} \Bigl ( \sum_{\substack{\indexSpace = 0 \\ \indexSpace \neq 1}}^{\orderExtrapolation - 1} \coefficientExtrapolation_{\indexSpace} \flux(\conservedMomentDiscrete_{\indexSpace}^{\indexTime}) + (\coefficientExtrapolation_1 - 1) \flux(\conservedMomentDiscrete_{1}^{\indexTime}) \Bigr )  + \sourceTermBoundaryDiscrete_0^{\indexTime + 1} + (1-\relaxationParameter)\sourceTermBoundaryDiscrete_0^{\indexTime},
        \end{multline}
        where $\coefficientCombinationElimination_{0} = \catalanTwoIndices{\orderExtrapolation-2}{1}$, $ \coefficientCombinationElimination_1 = 1 - \catalanTwoIndices{\orderExtrapolation - 3}{2}$, $\coefficientCombinationElimination_{\indexSpace} = (-1)^{\indexSpace}\catalanTwoIndices{\orderExtrapolation-\indexSpace-2}{\indexSpace + 1}$ for $\indexSpace \in \integerInterval{2}{\lfloor \tfrac{\orderExtrapolation-1}{2}\rfloor - 1}$, and $\coefficientCombinationElimination_{\orderExtrapolation - \indexSpace} = (-1)^{\orderExtrapolation + 1 - \indexSpace} \catalanTwoIndices{\orderExtrapolation - \indexSpace}{\indexSpace - 1}$ for $\indexSpace \in \integerInterval{1}{\lfloor \tfrac{\orderExtrapolation-1}{2}\rfloor + 1 + \tfrac{1 + (-1)^{\orderExtrapolation}}{2}}$, with $\catalanTwoIndices{n}{p} = 
        \tfrac{(n+p)!(n-p+1)}{p!(n+1)!}$ 
        the entries of the Catalan's triangle, see \cite{weissteinCatalan}.
    \end{itemize}
\end{proposition}
For specific values of $\orderExtrapolation$ and $\relaxationParameter$, the schemes \eqref{eq:bulkBoundarySigma1} and \eqref{eq:bulkBoundarySigmaGenericTwo} read as conditions already introduced in the \fd{} literature.
In particular, for $(\orderExtrapolation,\relaxationParameter) = (1, 0)$, we recover  \cite[Equation (2.5)]{trefethen1984instability} and \cite[Equation (11.2.2b)]{strikwerda2004finite}; for  $(\orderExtrapolation,\relaxationParameter) = (1, 2)$, we obtain \cite[Equation (6.3b)]{gustafsson1972stability} and \cite[Equation (12.2.2d)]{strikwerda2004finite}; for $(\orderExtrapolation, \relaxationParameter) = (2, 1)$, we get \cite[Equation (6.3b)]{gustafsson1972stability} and \cite[Equation (12.2.2d)]{strikwerda2004finite}.

We now deal with \eqref{eq:conditionKinRod}.
For this condition, we rely on the specific problem at hand: a linear problem with $\flux(\solutionCauchyProblem) = \advectionVelocity \solutionCauchyProblem$.
Difficulties in establishing a more general result come from the fact that this boundary condition is built enforcing equilibrium in the future, thus strongly depends on the choice of $\flux$.
The corresponding \fd{} scheme at the boundary away from the initial time is conjectured using \strong{optimization}.
Specifically, ignoring source terms for the sake of presentation, we assume that the conserved moment of the \lbm{} scheme $\conservedMomentDiscrete$ fulfills, for $\indexTime \geq 2$, the constraint $\conservedMomentDiscrete_0^{\indexTime + 1} = \sum_{\indexSpace \in \naturals} \coefficientOutflowFDZero_{\indexSpace} (\relaxationParameter, \courantNumber)\conservedMomentDiscrete_{\indexSpace}^{\indexTime} + \sum_{\indexSpace \in \naturals} \coefficientOutflowFDMinusOne_{\indexSpace} (\relaxationParameter, \courantNumber) \conservedMomentDiscrete_{\indexSpace}^{\indexTime - 1}$, with $\courantNumber \definitionEquality \advectionVelocity/\latticeVelocity$ is the Courant number and coefficients $\coefficientOutflowFDZero_{\indexSpace}, \coefficientOutflowFDMinusOne_{\indexSpace}$ that are polynomials in $\relaxationParameter$ and $\courantNumber$. We run \lbm{} simulations with random initial data for different $\relaxationParameter$ and $\courantNumber$, each time  
\begin{equation*}
    \text{minimizing}\quad\sum_{\indexTime = 3}^{\finalTime/\timeStep} \Bigl |\conservedMomentDiscrete_0^{\indexTime} - \Bigl (\sum_{\indexSpace \in \naturals} \coefficientOutflowFDZero_{\indexSpace} (\relaxationParameter, \courantNumber)\conservedMomentDiscrete_{\indexSpace}^{\indexTime-1} + \sum_{\indexSpace \in \naturals} \coefficientOutflowFDMinusOne_{\indexSpace} (\relaxationParameter, \courantNumber) \conservedMomentDiscrete_{\indexSpace}^{\indexTime - 2} \Bigr ) \Bigr |^2.
\end{equation*}
We thus obtain the following expressions which systematically yield vanishing residuals.
\begin{conjecture}\label{conj:kinrod}
    Consider a linear problem with $\flux(\solutionCauchyProblem) = \advectionVelocity \solutionCauchyProblem$ and the \lbm{} scheme given by \eqref{eq:initializationEquilibrium}, \eqref{eq:collision}, \eqref{eq:transport}, the inflow boundary condition \eqref{eq:inflowConditionScheme}, and the outflow boundary condition \eqref{eq:conditionKinRod}.
    Then, the corresponding \fd{} scheme on the conserved moment $\conservedMomentDiscrete$ reads, for the first time step 
    \begin{align}
        \conservedMomentDiscrete_0^1 &= \tfrac{1}{4}(1+2\tfrac{\advectionVelocity}{\latticeVelocity} + \tfrac{\advectionVelocity^2}{\latticeVelocity^2}) \conservedMomentDiscrete_0^0 + \tfrac{1}{2}(1-\tfrac{\advectionVelocity}{\latticeVelocity}) \conservedMomentDiscrete_1^0 + \tfrac{1}{4}(1-\tfrac{\advectionVelocity^2}{\latticeVelocity^2}) \conservedMomentDiscrete_2^0 + \sourceTermBoundaryDiscrete_0^1, \label{eq:kinrodFDBoundaryInitial}\\
        \conservedMomentDiscrete_{\indexSpace}^1 &= \tfrac{1}{2}(\conservedMomentDiscrete_{\indexSpace - 1}^0 + \conservedMomentDiscrete_{\indexSpace + 1}^0) + \tfrac{\advectionVelocity}{2\latticeVelocity} (\conservedMomentDiscrete_{\indexSpace - 1}^0 -\conservedMomentDiscrete_{\indexSpace + 1}^0), \qquad \indexSpace \in \integerInterval{1}{\numberSpacePoints - 2}, \nonumber\\
        \conservedMomentDiscrete_{\numberSpacePoints - 1}^1 &= \boundaryDatumInflow(\timeGridPoint{1}).\nonumber
    \end{align}
    For the second time step
    \begin{multline}\label{eq:kinrodFDBoundarySecond}
        \conservedMomentDiscrete_0^2 = \tfrac{1}{16}   {( \tfrac{\advectionVelocity^{4} {\relaxationParameter}}{{\latticeVelocity}^{4}} +  \tfrac{2   \advectionVelocity^{3} {\relaxationParameter}}{{\latticeVelocity}^{3}} -  \tfrac{4   \advectionVelocity^{2} {\relaxationParameter}}{{\latticeVelocity}^{2}} -  \tfrac{2   \advectionVelocity {\relaxationParameter}}{{\latticeVelocity}} + 3   {\relaxationParameter} +  \tfrac{2   \advectionVelocity^{3}}{{\latticeVelocity}^{3}} +  \tfrac{6   \advectionVelocity^{2}}{{\latticeVelocity}^{2}} +  \tfrac{6   \advectionVelocity}{{\latticeVelocity}} + 2)} \conservedMomentDiscrete_0^0 \\
        -  \tfrac{1}{4}   {( \tfrac{\advectionVelocity^{3} {\relaxationParameter}}{{\latticeVelocity}^{3}} +  \tfrac{\advectionVelocity^{2} {\relaxationParameter}}{{\latticeVelocity}^{2}} -  \tfrac{\advectionVelocity {\relaxationParameter}}{{\latticeVelocity}} - {\relaxationParameter})} \conservedMomentDiscrete_1^0  -  \tfrac{1}{16}  {( \tfrac{\advectionVelocity^{4} {\relaxationParameter}}{{\latticeVelocity}^{4}} -  \tfrac{6   \advectionVelocity^{2} {\relaxationParameter}}{{\latticeVelocity}^{2}} + 5   {\relaxationParameter} +  \tfrac{2   \advectionVelocity^{3}}{{\latticeVelocity}^{3}} +  \tfrac{2   \advectionVelocity^{2}}{{\latticeVelocity}^{2}} +  \tfrac{6   \advectionVelocity}{{\latticeVelocity}} - 10)} \conservedMomentDiscrete_2^0 \\
        +  \tfrac{1}{8}   {( \tfrac{\advectionVelocity^{3} {\relaxationParameter}}{{\latticeVelocity}^{3}} +  \tfrac{\advectionVelocity^{2} {\relaxationParameter}}{{\latticeVelocity}^{2}} -  \tfrac{\advectionVelocity {\relaxationParameter}}{{\latticeVelocity}} - {\relaxationParameter} -  \tfrac{2   \advectionVelocity^{2}}{{\latticeVelocity}^{2}} + 2)} \conservedMomentDiscrete_3^0 + \tfrac{1}{4}(\tfrac{\advectionVelocity^2\relaxationParameter}{\latticeVelocity^2}-\relaxationParameter + \tfrac{2\advectionVelocity}{\latticeVelocity}+2)\sourceTermBoundaryDiscrete_0^1 + \sourceTermBoundaryDiscrete_0^2.
    \end{multline}
    \begin{multline}\label{eq:kinrodFDBoundaryPlusOneSecond}
        \conservedMomentDiscrete_1^2 = \tfrac{1}{8}  {(\tfrac{\advectionVelocity^{3} {\relaxationParameter}}{{\latticeVelocity}^{3}} + \tfrac{\advectionVelocity^{2} {\relaxationParameter}}{{\latticeVelocity}^{2}} - \tfrac{\advectionVelocity {\relaxationParameter}}{{\latticeVelocity}} - {\relaxationParameter} + \tfrac{2  \advectionVelocity^{2}}{{\latticeVelocity}^{2}} + \tfrac{4  \advectionVelocity}{{\latticeVelocity}} + 2)} \conservedMomentDiscrete_0^0  - \tfrac{1}{2}  {(\tfrac{\advectionVelocity^{2} {\relaxationParameter}}{{\latticeVelocity}^{2}} - {\relaxationParameter})} \conservedMomentDiscrete_1^0   \\
        - \tfrac{1}{8}  {(\tfrac{\advectionVelocity^{3} {\relaxationParameter}}{{\latticeVelocity}^{3}} - \tfrac{\advectionVelocity^{2} {\relaxationParameter}}{{\latticeVelocity}^{2}} - \tfrac{\advectionVelocity {\relaxationParameter}}{{\latticeVelocity}} + {\relaxationParameter} + \tfrac{2  \advectionVelocity^{2}}{{\latticeVelocity}^{2}} - 2)} \conservedMomentDiscrete_2^0 + \tfrac{1}{4}  {(\tfrac{\advectionVelocity^{2} {\relaxationParameter}}{{\latticeVelocity}^{2}} - {\relaxationParameter} - \tfrac{2  \advectionVelocity}{{\latticeVelocity}} + 2)} \conservedMomentDiscrete_3^0 + \tfrac{1}{2}  {(\tfrac{\advectionVelocity {\relaxationParameter}}{{\latticeVelocity}} - {\relaxationParameter} + 2)}\sourceTermBoundaryDiscrete_0^1
    \end{multline}
    \begin{equation*}
        \conservedMomentDiscrete_{\indexSpace}^{2} = \tfrac{2-\relaxationParameter}{2}(\conservedMomentDiscrete_{\indexSpace - 1}^{1} + \conservedMomentDiscrete_{\indexSpace + 1}^{1}) + (\relaxationParameter - 1) \conservedMomentDiscrete_{\indexSpace}^{0} + \tfrac{\relaxationParameter\advectionVelocity}{2\latticeVelocity} (\conservedMomentDiscrete_{\indexSpace - 1}^{0} - \conservedMomentDiscrete_{\indexSpace + 1}^{0}), \qquad \indexSpace \in \integerInterval{3}{\numberSpacePoints-2}, \qquad 
        \conservedMomentDiscrete_{\numberSpacePoints - 1}^2 = \boundaryDatumInflow(\timeGridPoint{2})
    \end{equation*}
    And eventually in time $\indexTime \geq 2$
    \begin{multline}\label{eq:kinrodFDBOundaryEventually}
        \conservedMomentDiscrete_0^{\indexTime + 1} = \tfrac{1}{4} (\tfrac{\advectionVelocity}{\latticeVelocity} +1)((\tfrac{\advectionVelocity}{\latticeVelocity} - 1)\relaxationParameter + 2) \conservedMomentDiscrete_0^{\indexTime} + (-\tfrac{1}{2} (\tfrac{\advectionVelocity}{\latticeVelocity} + 1)\relaxationParameter + 1) \conservedMomentDiscrete_1^{\indexTime} - \tfrac{1}{4} (\tfrac{\advectionVelocity}{\latticeVelocity} +1) ((\tfrac{\advectionVelocity}{\latticeVelocity} + 1)\relaxationParameter - 2) \conservedMomentDiscrete_2^{\indexTime} \\
        + \tfrac{1}{2}  (\tfrac{\advectionVelocity}{\latticeVelocity} +1) \relaxationParameter (\relaxationParameter-1)\conservedMomentDiscrete_0^{\indexTime-1} - \tfrac{1}{2}  (\tfrac{\advectionVelocity}{\latticeVelocity} +1) (\relaxationParameter-1)(\relaxationParameter-2)\conservedMomentDiscrete_1^{\indexTime-1} + \sourceTermBoundaryDiscrete_0^{\indexTime + 1} - (\relaxationParameter-1)^2 \sourceTermBoundaryDiscrete_0^{\indexTime - 1},
    \end{multline}
    \begin{equation}
        \conservedMomentDiscrete_{\indexSpace}^{\indexTime + 1} = \tfrac{2-\relaxationParameter}{2}(\conservedMomentDiscrete_{\indexSpace - 1}^{\indexTime} + \conservedMomentDiscrete_{\indexSpace + 1}^{\indexTime}) + (\relaxationParameter - 1) \conservedMomentDiscrete_{\indexSpace}^{\indexTime - 1} + \tfrac{\relaxationParameter\advectionVelocity}{2\latticeVelocity} (\conservedMomentDiscrete_{\indexSpace - 1}^{\indexTime} - \conservedMomentDiscrete_{\indexSpace + 1}^{\indexTime})\quad \indexSpace \in \integerInterval{1}{\numberSpacePoints - 2}, \qquad \conservedMomentDiscrete_{\numberSpacePoints - 1}^{\indexTime + 1} = \boundaryDatumInflow(\timeGridPoint{\indexTime + 1}).
    \end{equation}
\end{conjecture}

\begin{remark}[On a possible proof of \Cref{conj:kinrod}]
    Looking at \eqref{eq:kinrodFDBOundaryEventually}, we see that a possible proof of \Cref{conj:kinrod} would be non-standard, due to the presence of quadratic terms in $\advectionVelocity$.
    In the standard proofs germane to \Cref{prop:bulkAndInflow} and \ref{prop:outflowExtrapolation}, the dependence on $\advectionVelocity$ could only be linear because stemming from linear terms in $\flux$.
\end{remark}

\section{Consistency of the boundary conditions}\label{sec:consistency}

\lbm{} schemes being turned into \fd{} methods, we now study their consistency, theoretically, in \Cref{sec:consistencyTheoretical}, using modified equations \cite{warming1974modified}.
This analysis clarifies the role of initialization, which may bring order losses, that we correct in \Cref{sec:compensationEquilibrium} with \emph{ad hoc} source terms.
We corroborate these findings through numerical experiments as presented in \Cref{sec:numericalSimConsistency}.

\subsection{Modified equations}\label{sec:consistencyTheoretical}

\begin{proposition}[Modified equation for outflow \eqref{eq:extrapolationBoundaryCondition}]\label{prop:outflowExtrapolationModifiedEquation}
    Consider the \lbm{} scheme given by \eqref{eq:initializationEquilibrium}, \eqref{eq:collision}, \eqref{eq:transport}, and the outflow boundary condition \eqref{eq:extrapolationBoundaryCondition}.
    Take $\sourceTermBoundaryDiscrete_0^{\indexTime} = 0$ for $\indexTime \in \naturalsWithoutZero$.
    Then, the modified equations obtained using the corresponding \fd{} scheme  from \Cref{prop:outflowExtrapolation}, computed at the outflow $\spaceVariable = 0$, are as follows.
    \begin{itemize}
        \item For $\orderExtrapolation = 1$:
        \begin{align}
            \partial_{\timeVariable} \testFunction(0, 0) + \tfrac{1}{2} \partial_{\spaceVariable} (\flux(\testFunction) - \latticeVelocity \testFunction) (0, 0) &= \bigO{\spaceStep},\label{eq:modifiedSigmaOneInitial} \\
            \partial_{\timeVariable} \testFunction(\timeVariable, 0) + \tfrac{\latticeVelocity}{2}(\relaxationParameter - 2)\partial_{\spaceVariable} \testFunction(\timeVariable, 0) + \tfrac{\relaxationParameter}{2}\partial_{\spaceVariable}(\flux(\testFunction))(\timeVariable, 0) &= \bigO{\spaceStep}, \qquad \timeVariable > 0. \label{eq:modifiedSigmaOne}
        \end{align}
        \item For $\orderExtrapolation \geq 2$:
        \begin{equation}
            \partial_{\timeVariable} \testFunction(\timeVariable, 0) + \partial_{\spaceVariable}(\flux(\testFunction))(\timeVariable, 0) = \bigO{\spaceStep}, \qquad \timeVariable \geq 0.
        \end{equation}
    \end{itemize}
\end{proposition}
We use $\testFunction$ instead of $\solutionCauchyProblem$ in \Cref{prop:outflowExtrapolationModifiedEquation} to stress that the exact solution of the problem does not satisfy the modified equations.
The proof of \Cref{prop:outflowExtrapolationModifiedEquation}, which is based on simple Taylor expansions and properties of the entries of the Catalan's triangle, is given in the Supplementary Material.
Unlike physical inflow boundaries, schemes enforcing numerical outflow boundary conditions can \strong{lose one order of accuracy} without compromising the overall order, see \cite{gustafsson1975convergence}.
This is important while pondering \Cref{prop:outflowExtrapolationModifiedEquation} and \ref{prop:outflowExtrapolationModifiedEquationKinrod}.

For \eqref{eq:papillon1}, whenever $\relaxationParameter \in (0, 2)$, neither the modified equation \eqref{eq:modifiedSigmaOneInitial} nor \eqref{eq:modifiedSigmaOne} describe a scheme consistent with the target equation \eqref{eq:conservationLaw}. However, this does not reduce the order of the method, which is just first-order accurate.
When $\relaxationParameter = 2$, with second-order accurate bulk method, \eqref{eq:modifiedSigmaOne} indicates that the boundary scheme away from the initial time is first-order accurate, which is fine. 
Nevertheless, the initial boundary scheme is not consistent when $\sourceTermBoundaryDiscrete_0^{\indexTime}\equiv 0$, see \eqref{eq:modifiedSigmaOneInitial}. Therefore, this causes \strong{order reduction} whenever the initial datum $\solutionCauchyProblemInitial$ has  $\partial_{\spaceVariable} \solutionCauchyProblemInitial(0) \neq 0$.
In this circumstance, the order of convergence is $3/2$ in $\spaceStep$ for the $L^2$-norm.
This can be seen by constructing the local truncation error $\epsilon_{\indexSpace}^{\indexTime} \definitionEquality \solutionCauchyProblem(\timeGridPoint{\indexTime}, \spaceGridPoint{\indexSpace}) - \conservedMomentDiscrete_{\indexSpace}^{\indexTime}$ and assuming that we are solving the linear advection equation.
Thereby $\epsilon_{\indexSpace}^{\indexTime}$ fulfills the same numerical scheme as the solution $\conservedMomentDiscrete_{\indexSpace}^{\indexTime}$ with suitable source terms, all proportional to $\spaceStep^2$ except for the one---denoted by $s_0^1$---concerning $\epsilon_{0}^{1}$, of order $\spaceStep$.
To simplify, consider a semi-infinite problem to the right and the assumptions in \cite{coulombel2015leray} be fulfilled.
For smooth initial data, we have $s_{\indexSpace}^{0} = 0$ and $s_{\indexSpace}^{1} = 
-\tfrac{\spaceStep}{2}(\tfrac{\advectionVelocity}{\latticeVelocity} + 1) \partial_{\spaceVariable}\solutionCauchyProblemInitial(0) \mathds{1}_{\indexSpace = 0} + \bigO{\spaceStep^2}$, whereas all other truncation errors are $\bigO{\spaceStep^2}$, thus negligible.
Using \cite[Theorem 1]{coulombel2015leray}, we obtain
\begin{equation*}
    \sup_{\indexTime \in \naturals} \Bigl (\sum_{\indexSpace \in \naturals}\spaceStep|\epsilon_{\indexSpace}^{\indexTime}|^2 \Bigr )^{1/2} \leq \Bigl (C \Bigl ( \sum_{\indexSpace \in \naturals}\spaceStep|s_{\indexSpace}^{0}|^2 + \sum_{\indexSpace \in \naturals}\spaceStep|s_{\indexSpace}^{1}|^2\Bigr ) \Bigr)^{1/2} \leq \sqrt{C}\tfrac{\spaceStep^{3/2}}{2} \left |\tfrac{\advectionVelocity}{\latticeVelocity} + 1\right | |\partial_{\spaceVariable}\solutionCauchyProblemInitial(0)| + \bigO{\spaceStep^2},
\end{equation*}
where the first inequality relies the $L^2$ stability of the scheme, which does not hold for any other $L^p$ norms.

\begin{proposition}[Modified equation for outflow \eqref{eq:conditionKinRod}]\label{prop:outflowExtrapolationModifiedEquationKinrod}
    Consider the \lbm{} scheme given by \eqref{eq:initializationEquilibrium}, \eqref{eq:collision}, \eqref{eq:transport}, and the outflow boundary condition \eqref{eq:conditionKinRod}.
    Take $\sourceTermBoundaryDiscrete_0^{\indexTime} = 0$ for $\indexTime \in \naturalsWithoutZero$, and a linear flux $\flux(\solutionCauchyProblem) = \advectionVelocity\solutionCauchyProblem$.
    Then, the modified equation obtained using the corresponding \fd{} scheme from \Cref{conj:kinrod}, computed close to the outflow $\spaceVariable = 0$, are as follows.
    For \eqref{eq:kinrodFDBoundaryInitial}
    \begin{equation}
        \partial_{\timeVariable}\testFunction(0, 0) + \tfrac{1}{2} (\tfrac{\advectionVelocity^2}{\latticeVelocity} + \advectionVelocity - 2\latticeVelocity)\partial_{\spaceVariable}\testFunction(0, 0) = \bigO{\spaceStep}.
    \end{equation}
    For \eqref{eq:kinrodFDBoundarySecond}, we obtain 
    \begin{equation}
        \partial_{\timeVariable}\testFunction(0, 0) + \tfrac{1}{16} \bigl ( 2\tfrac{\advectionVelocity^3}{\latticeVelocity^2}  + 8 \tfrac{\advectionVelocity^2}{\latticeVelocity} + 6\advectionVelocity - 16\latticeVelocity  + \relaxationParameter (\tfrac{\advectionVelocity^4}{\latticeVelocity^3} - \tfrac{\advectionVelocity^3}{\latticeVelocity^2} - 7 \tfrac{\advectionVelocity^2}{\latticeVelocity} + \advectionVelocity + 6\latticeVelocity ) \bigr )\partial_{\spaceVariable}\testFunction(0, 0) = \bigO{\spaceStep}.
    \end{equation}
    For \eqref{eq:kinrodFDBoundaryPlusOneSecond}, we have 
    \begin{equation}
        \partial_{\timeVariable}\testFunction(0, \spaceStep) + \tfrac{1}{8} \bigl (  2\tfrac{\advectionVelocity^2}{\latticeVelocity} + 6  \advectionVelocity - 4\latticeVelocity + \relaxationParameter(\tfrac{\advectionVelocity^3}{\latticeVelocity^2} - 2\tfrac{\advectionVelocity^2}{\latticeVelocity} - \advectionVelocity + \latticeVelocity )  \bigr ) \partial_{\spaceVariable}\testFunction(0, \spaceStep) = \bigO{\spaceStep}.
    \end{equation}
    Eventually, for \eqref{eq:kinrodFDBOundaryEventually}, we have 
    \begin{equation*}
        \partial_{\timeVariable} \testFunction(\timeVariable, 0) - \latticeVelocity \frac{1 - \frac{\relaxationParameter}{2}\bigl (\frac{\advectionVelocity}{\latticeVelocity} + 1\bigr )  \bigl ( \relaxationParameter - 1 + \frac{\advectionVelocity}{\latticeVelocity}\bigr )}{1 + \bigl (\frac{\advectionVelocity}{\latticeVelocity} + 1\bigr )  \bigl ( \relaxationParameter - 1 \bigr )}\partial_{\spaceVariable} \testFunction(\timeVariable, 0)= \bigO{\spaceStep}, \qquad \timeVariable > 0.
    \end{equation*}
\end{proposition}

\subsection{Boundary sources to compensate initializations at equilibrium}\label{sec:compensationEquilibrium}

We now construct a \strong{source term} $\sourceTermBoundaryDiscrete_0^{\indexTime}$ to achieve second-order accuracy when $\relaxationParameter = 2$ and using \eqref{eq:papillon1}.
We request that the initial scheme at the boundary \eqref{eq:initialSchemeBoundaryExtrapolation} be an upwind scheme, obtaining
\begin{equation}\label{eq:correctionPapillon1}
    \sourceTermBoundaryDiscrete_0^1 = \tfrac{1}{2}(\conservedMomentDiscrete_0^0 - \conservedMomentDiscrete_1^0) + \tfrac{1}{2\latticeVelocity} (\flux(\conservedMomentDiscrete_0^0) - \flux(\conservedMomentDiscrete_1^0)).
\end{equation}
Since the bulk \fd{} scheme at the boundary \eqref{eq:bulkBoundarySigma1} is first-order consistent, which is fine, we want to perturb it as little as possible, naturally enforcing $\sourceTermBoundaryDiscrete_0^{\indexTime + 1} + (1-\relaxationParameter)\sourceTermBoundaryDiscrete_0^{\indexTime} = 0$, hence  $\sourceTermBoundaryDiscrete_0^{\indexTime} = (\relaxationParameter - 1)^{\indexTime - 1}\sourceTermBoundaryDiscrete_0^1$ for $\indexTime \in\naturalsWithoutZero$.

Condition \eqref{eq:conditionKinRod} shares the same issue with  \eqref{eq:papillon1}.
We present the procedure only in the linear case.
Nevertheless, the dependence of the bulk scheme at the boundary \eqref{eq:kinrodFDBOundaryEventually} on the source term $\sourceTermBoundaryDiscrete_0^{\indexTime}$ can also be used in a non-linear context, for we can compute the initial schemes corresponding to the choice of initial datum and the coefficients in front of $\sourceTermBoundaryDiscrete_0^{\indexTime}$ in \eqref{eq:kinrodFDBOundaryEventually} do not depend on the advection velocity $\advectionVelocity$.
\strong{Source terms} are tuned to obtain an upwind scheme as boundary initial scheme \eqref{eq:kinrodFDBoundaryInitial} and an upwind scheme---applied twice---for the second scheme at the boundary \eqref{eq:kinrodFDBoundarySecond}.
This gives 
\begin{multline}\label{eq:correctionKinrod1}
    \sourceTermBoundaryDiscrete_0^1 = \tfrac{1}{4} ((-\tfrac{\advectionVelocity^2}{\latticeVelocity^2}+2\tfrac{\advectionVelocity}{\latticeVelocity}+3)\conservedMomentDiscrete_0^0 + (-2\tfrac{\advectionVelocity}{\latticeVelocity} - 2) \conservedMomentDiscrete_1^0 + (\tfrac{\advectionVelocity^2}{\latticeVelocity^2} - 1)\conservedMomentDiscrete_2^0), \\
    \sourceTermBoundaryDiscrete_0^2 = (\tfrac{1}{2}+\tfrac{\advectionVelocity}{\latticeVelocity}+\tfrac{\advectionVelocity^2}{2\latticeVelocity^2} + \tfrac{\relaxationParameter\advectionVelocity}{4\latticeVelocity}(1 - \tfrac{\advectionVelocity^2}{\latticeVelocity^2})) \conservedMomentDiscrete_0^0 + \tfrac{1}{8}( 2 - 12 \tfrac{\advectionVelocity}{\latticeVelocity} - 14\tfrac{\advectionVelocity^2}{\latticeVelocity^2} + 3\relaxationParameter(-1-\tfrac{\advectionVelocity}{\latticeVelocity}+\tfrac{\advectionVelocity^2}{\latticeVelocity^2} + \tfrac{\advectionVelocity^3}{\latticeVelocity^3})) \conservedMomentDiscrete_1^0  \\
    - \tfrac{1}{4} (2 - 2\tfrac{\advectionVelocity}{\latticeVelocity} - 4\tfrac{\advectionVelocity^2}{\latticeVelocity^2} + \relaxationParameter (-1 + \tfrac{\advectionVelocity^2}{\latticeVelocity^2})) \conservedMomentDiscrete_2^0 - \tfrac{1}{8} (2 - 2\tfrac{\advectionVelocity^2}{\latticeVelocity^2} + \relaxationParameter (-1-\tfrac{\advectionVelocity}{\latticeVelocity} + \tfrac{\advectionVelocity^2}{\latticeVelocity^2} + \tfrac{\advectionVelocity^3}{\latticeVelocity^3} ))\conservedMomentDiscrete_3^0.
\end{multline}
With this choice, one can see that \eqref{eq:kinrodFDBoundaryPlusOneSecond} becomes first-order accurate.
We take $\sourceTermBoundaryDiscrete_0^{\indexTime} = (\relaxationParameter-1)^{\indexTime - 2}\sourceTermBoundaryDiscrete_0^2$ for $\indexTime$ even and $\sourceTermBoundaryDiscrete_0^{\indexTime} = (\relaxationParameter-1)^{\indexTime - 1}\sourceTermBoundaryDiscrete_0^1$ for $\indexTime$ odd.

\subsection{Numerical simulations}\label{sec:numericalSimConsistency}

\begin{table}[h]
    \begingroup
    \setlength{\tabcolsep}{3pt}
    \begin{center}\caption{\label{tab:convergenceBoundaryTransport}Error at final time for different simulations of the advection equation.}
        \begin{footnotesize}
            \begin{tabular}{|c|cc|cc|cc|cc|cc|}
                \hline
                \multicolumn{11}{|c|}{$\relaxationParameter = 2$ : \textsc{second-order bulk scheme}}\\
                \hline
                & \multicolumn{2}{|c|}{\eqref{eq:papillon1} w. $\sourceTermBoundaryDiscrete_0^{\indexTime} \equiv 0$} & \multicolumn{2}{|c|}{\eqref{eq:papillon1} w. \eqref{eq:correctionPapillon1}} & \multicolumn{2}{|c|}{\eqref{eq:papillon2} w. $\sourceTermBoundaryDiscrete_0^{\indexTime} \equiv 0$} & \multicolumn{2}{|c|}{\eqref{eq:conditionKinRod} w. $\sourceTermBoundaryDiscrete_0^{\indexTime} \equiv 0$} & \multicolumn{2}{|c|}{\eqref{eq:conditionKinRod} w. \eqref{eq:correctionKinrod1}}\\ 
                \hline
                $\spaceStep$ & $L^2$ error & Emp. order & $L^2$ error & Emp. order & $L^2$ error & Emp. order & $L^2$ error & Emp. order & $L^2$ error & Emp. order \\
                \hline
                2.041E-02 &	2.432E-04 &	     & 6.645E-05	& 	    & 6.561E-05	&      & 7.669E-04	&      & 7.581E-05 & 	  \\
                1.266E-02 &	1.189E-04 &	1.50 & 2.557E-05	& 2.00	& 2.525E-05	& 2.00 & 3.748E-04	& 1.50 & 2.905E-05 & 	2.01\\
                7.874E-03 &	5.840E-05 &	1.50 & 9.894E-06	& 2.00	& 9.772E-06	& 2.00 & 1.840E-04	& 1.50 & 1.121E-05 & 	2.01\\
                4.926E-03 &	2.895E-05 &	1.50 & 3.873E-06	& 2.00	& 3.825E-06	& 2.00 & 9.105E-05	& 1.50 & 4.382E-06 & 	2.00\\
                3.077E-03 &	1.429E-05 &	1.50 & 1.511E-06	& 2.00	& 1.492E-06	& 2.00 & 4.499E-05	& 1.50 & 1.708E-06 & 	2.00\\
                1.923E-03 &	7.095E-06 &	1.49 & 5.832E-07	& 2.03	& 5.754E-07	& 2.03 & 2.227E-05	& 1.50 & 6.605E-07 & 	2.02\\
                1.202E-03 &	3.501E-06 &	1.50 & 2.278E-07	& 2.00	& 2.248E-07	& 2.00 & 1.100E-05	& 1.50 & 2.579E-07 & 	2.00\\
                7.513E-04 &	1.725E-06 &	1.51 & 9.009E-08	& 1.97	& 8.899E-08	& 1.97 & 5.432E-06	& 1.50 & 1.017E-07 & 	1.98\\
                4.695E-04 &	8.533E-07 &	1.50 & 3.476E-08	& 2.03	& 3.432E-08	& 2.03 & 2.684E-06	& 1.50 & 3.934E-08 & 	2.02\\
                2.934E-04 &	4.215E-07 &	1.50 & 1.358E-08	& 2.00	& 1.341E-08	& 2.00 & 1.326E-06	& 1.50 & 1.536E-08 & 	2.00\\
                \hline
                Theoretical & --- &  3/2 & --- & 2 & --- & 2 & --- & 3/2 & --- & 2\\
                \hline
                \hline
                \multicolumn{11}{|c|}{$\relaxationParameter = 1.98$ : \textsc{first-order bulk scheme}}\\
                \hline
                & \multicolumn{2}{|c|}{\eqref{eq:papillon1} w. $\sourceTermBoundaryDiscrete_0^{\indexTime} \equiv 0$} & \multicolumn{2}{|c|}{\eqref{eq:papillon1} w. \eqref{eq:correctionPapillon1}} & \multicolumn{2}{|c|}{\eqref{eq:papillon2} w. $\sourceTermBoundaryDiscrete_0^{\indexTime} \equiv 0$} & \multicolumn{2}{|c|}{\eqref{eq:conditionKinRod} w. $\sourceTermBoundaryDiscrete_0^{\indexTime} \equiv 0$} & \multicolumn{2}{|c|}{\eqref{eq:conditionKinRod} w. \eqref{eq:correctionKinrod1}}\\ 
                \hline
                $\spaceStep$ & $L^2$ error & Emp. order & $L^2$ error & Emp. order & $L^2$ error & Emp. order & $L^2$ error & Emp. order & $L^2$ error & Emp. order \\
                \hline                
                2.041E-02	& 1.383E-04 &		    & 1.051E-04 &		    & 9.891E-05 &		    & 3.59E-04	& 	    & 1.694E-04	&      \\  
                1.266E-02	& 5.729E-05 &	1.85	& 5.190E-05 &	1.48	& 4.795E-05 &	1.52	& 1.23E-04	& 2.24	& 8.197E-05	& 1.52 \\
                7.874E-03	& 2.741E-05 &	1.55	& 2.708E-05 &	1.37	& 2.505E-05 &	1.37	& 4.73E-05	& 2.01	& 4.040E-05	& 1.49 \\
                4.926E-03	& 1.489E-05 &	1.30	& 1.489E-05 &	1.27	& 1.395E-05 &	1.25	& 2.29E-05	& 1.54	& 2.052E-05	& 1.44 \\
                3.077E-03	& 8.487E-06 &	1.19	& 8.487E-06 &	1.20	& 8.077E-06 &	1.16	& 1.21E-05	& 1.36	& 1.077E-05	& 1.37 \\
                1.923E-03	& 4.980E-06 &	1.13	& 4.980E-06 &	1.13	& 4.807E-06 &	1.10	& 6.58E-06	& 1.30	& 4.745E-06	& 1.74 \\
                1.202E-03	& 2.983E-06 &	1.09	& 2.983E-06 &	1.09	& 2.912E-06 &	1.07	& 3.67E-06	& 1.24	& 2.904E-06	& 1.04 \\
                7.513E-04	& 1.813E-06 &	1.06	& 1.813E-06 &	1.06	& 1.785E-06 &	1.04	& 2.10E-06	& 1.19	& 1.957E-06	& 0.84 \\
                4.695E-04	& 1.113E-06 &	1.04	& 1.113E-06 &	1.04	& 1.101E-06 &	1.03	& 1.23E-06	& 1.14	& 1.103E-06	& 1.22 \\
                2.934E-04	& 6.874E-07 &	1.02	& 6.874E-07 &	1.02	& 6.829E-07 &	1.02	& 7.35E-07	& 1.10	& 6.838E-07	& 1.02 \\
                \hline
                Theoretical & --- &  1 & --- & 1 & --- & 1 & --- & 1 & --- & 1\\
                \hline
            \end{tabular}
        \end{footnotesize}
    \end{center}
    \endgroup
\end{table}

We now verify the findings from \Cref{sec:consistencyTheoretical} and \ref{sec:compensationEquilibrium} using the original \lbm{} algorithm.
We first test convergence for the advection equation, using $\domainLength = 1$, $\courantNumber = -1/2$, and measuring the $L^2$ error at the final time $\finalTime = 1$.
The initial datum is given by $\solutionCauchyProblemInitial(\spaceVariable) = \sin(\spaceVariable)$.
The results in \Cref{tab:convergenceBoundaryTransport} are \strong{in agreement} with the expected convergence rates.
Empirical orders are computed in the usual fashion, \emph{i.e.} dividing the logarithm of the ratio of consecutive errors by the logarithm of the ratio of consecutive space steps.
In particular, we observe the order $3/2$ whenever $\relaxationParameter = 2$ and either \eqref{eq:papillon1} or \eqref{eq:conditionKinRod} is employed with source terms switched off.
When corrections are used, we see that the error constant for \eqref{eq:papillon1} is slightly better than the one for  \eqref{eq:conditionKinRod}.

\begin{table}[h] 
    \begingroup
    \setlength{\tabcolsep}{3pt}
    \begin{center}\caption{\label{tab:convergenceBoundaryBurgers}Error at final time for different simulations of the Burgers equation.}
        \begin{footnotesize}
            \begin{tabular}{|c|cc|cc|cc|cc|cc|}
                \hline
                \multicolumn{11}{|c|}{$\relaxationParameter = 2$ : \textsc{second-order bulk scheme}}\\
                \hline
                & \multicolumn{2}{|c|}{\eqref{eq:papillon1} w. $\sourceTermBoundaryDiscrete_0^{\indexTime} \equiv 0$} & \multicolumn{2}{|c|}{\eqref{eq:papillon1} w. \eqref{eq:correctionPapillon1}} & \multicolumn{2}{|c|}{\eqref{eq:papillon2} w. $\sourceTermBoundaryDiscrete_0^{\indexTime} \equiv 0$} & \multicolumn{2}{|c|}{\eqref{eq:conditionKinRod} w. $\sourceTermBoundaryDiscrete_0^{\indexTime} \equiv 0$} & \multicolumn{2}{|c|}{\eqref{eq:conditionKinRod} w. \eqref{eq:correctionKinrod1}} \\ 
                \hline
                $\spaceStep$ & $L^2$ error & Emp. order & $L^2$ error & Emp. order & $L^2$ error & Emp. order & $L^2$ error & Emp. order & $L^2$ error & Emp. order \\
                \hline
                2.041E-02& 	6.239E-04 &		    & 6.158E-04	& 	    & 6.184E-04 &		    & 8.697E-04 &	     & 6.189E-04 & 	\\
                1.266E-02& 	2.099E-04 &	2.28	& 2.046E-04	& 2.31	& 2.064E-04 &	2.30	& 3.669E-04 &	1.81 & 2.060E-04 & 	2.30\\
                7.874E-03& 	8.525E-05 &	1.90	& 7.902E-05	& 2.00	& 8.002E-05 &	2.00	& 1.712E-04 &	1.61 & 7.963E-05 & 	2.00\\
                4.926E-03& 	4.449E-05 &	1.39	& 3.049E-05	& 2.03	& 3.008E-05 &	2.09	& 9.472E-05 &	1.26 & 3.074E-05 & 	2.03\\
                3.077E-03& 	1.640E-05 &	2.12	& 1.206E-05	& 1.97	& 1.224E-05 &	1.91	& 4.103E-05 &	1.78 & 1.215E-05 & 	1.97\\
                1.923E-03& 	8.473E-06 &	1.41	& 4.717E-06	& 2.00	& 4.645E-06 &	2.06	& 2.132E-05 &	1.39 & 4.754E-06 & 	2.00\\
                1.202E-03& 	3.837E-06 &	1.69	& 1.835E-06	& 2.01	& 1.807E-06 &	2.01	& 1.032E-05 &	1.54 & 1.849E-06 & 	2.01\\
                7.513E-04& 	1.781E-06 &	1.63	& 7.184E-07	& 2.00	& 7.073E-07 &	2.00	& 5.032E-06 &	1.53 & 7.240E-07 & 	2.00\\
                4.695E-04& 	8.418E-07 &	1.59	& 2.808E-07	& 2.00	& 2.764E-07 &	2.00	& 2.464E-06 &	1.52 & 2.830E-07 & 	2.00\\
                2.934E-04& 	3.934E-07 &	1.62	& 1.092E-07	& 2.01	& 1.109E-07 &	1.94	& 1.200E-06 &	1.53 & 1.101E-07 & 	2.01\\
                \hline 
                Theoretical & --- & 3/2 & --- & 2 & --- & 2 & --- & 3/2 & --- & 2\\
                \hline
            \end{tabular}
        \end{footnotesize}
    \end{center}
    \endgroup
\end{table}

We also showcase a non-linear problem, the Burgers equation with $\flux(\solutionCauchyProblem) = -\solutionCauchyProblem^2/2$, using $\domainLength = 1$.
We take $\latticeVelocity = 1$, final time $\finalTime = 0.2$, and initial datum $\solutionCauchyProblemInitial(\spaceVariable) = \tfrac{1}{2} + \tfrac{1}{2}(\textnormal{tanh} \bigl ( \tfrac{2\spaceVariable}{1-4\spaceVariable^2}\bigr ) \mathds{1}_{|2\spaceVariable| < 1} + \textnormal{sgn} (2\spaceVariable) \mathds{1}_{|2\spaceVariable| \geq 1}\bigr ) $,  ensuring that the left (resp. right) boundary is an outflow (resp. inflow).
The results in \Cref{tab:convergenceBoundaryBurgers} agree with the theoretical predictions.

\section{Stability of the boundary conditions}\label{sec:stability}

Alongside consistency, stability is the other cornerstone of numerical analysis that we presently address.
We start---in \Cref{sec:GKS}---by using the GKS theory, which considers decoupled problems for \strong{each boundary}, set on the half-line.
The numerical validations in \Cref{sec:numericalSimulationsStability} sometimes feature  \strong{unforeseen} results.
In these circumstances, in \Cref{sec:matrixMethod}, we propose alternative tools, namely the so-called ``\strong{matrix method}'' and \strong{pseudo-spectra}, to provide a complementary point of view on this matter.
From the way we have obtained the \fd{} schemes, if $\conservedMomentDiscrete^{\indexTime}$ is the solution of the \lbm{}, then it \emph{also} exactly fulfills the \fd{} scheme. Therefore, instabilities of the \lbm{} schemes are also to be found using \fd{} schemes, so that, whatever the approach of choice to study stability, we apply it to \fd{} schemes.

Throughout the section, we consider a linear problem with $\flux(\solutionCauchyProblem) = \advectionVelocity \solutionCauchyProblem$, so that the bulk \fd{}s read 
\begin{equation}\label{eq:bulkSchemeAbstract}
    \conservedMomentDiscrete_{\indexSpace}^{\indexTime + 1} = \schemeMatrixFDBlockZeroEntry_{-1}\conservedMomentDiscrete_{\indexSpace - 1}^{\indexTime} + \schemeMatrixFDBlockZeroEntry_{1}\conservedMomentDiscrete_{\indexSpace + 1}^{\indexTime} + \schemeMatrixFDBlockMinusOneEntry_0 \conservedMomentDiscrete_{\indexSpace}^{\indexTime - 1}, \qquad \indexSpace \in \integerInterval{1}{\numberSpacePoints - 2},
\end{equation}
where $\schemeMatrixFDBlockZeroEntry_{\pm1} \definitionEquality \tfrac{1}{2}(2-\relaxationParameter \mp \relaxationParameter\courantNumber)$ and $\schemeMatrixFDBlockMinusOneEntry_0 \definitionEquality \relaxationParameter - 1$.
All the considered outflow boundary conditions \eqref{eq:extrapolationBoundaryCondition} and \eqref{eq:conditionKinRod} recast as boundary schemes of the form 
\begin{equation}\label{eq:outflowSchemeAbstract}
    \conservedMomentDiscrete_0^{\indexTime + 1} = \sum_{\indexSpace = 0}^{\numberCoefficientsOutFlowFDZero - 1}\coefficientOutflowFDZero_{\indexSpace}\conservedMomentDiscrete_{\indexSpace}^{\indexTime} + \sum_{\indexSpace = 0}^{\numberCoefficientsOutFlowFDMinusOne - 1}\coefficientOutflowFDMinusOne_{\indexSpace}\conservedMomentDiscrete_{\indexSpace}^{\indexTime - 1}, 
\end{equation}
with suitable weights $\coefficientOutflowFDZero_{\indexSpace}$ and $\coefficientOutflowFDMinusOne_{\indexSpace}$, and $\numberCoefficientsOutFlowFDZero, \numberCoefficientsOutFlowFDMinusOne$ independent of $\numberSpacePoints$.
For the stability constraints for periodic boundary conditions are the \emph{conditio sine qua non} even in presence of non-trivial boundary conditions, let us provide them (see \cite{bellotti2024initialisation} for the proof).
\begin{proposition}[$L^2$ stability conditions for \eqref{eq:bulkSchemeAbstract} with periodicity]\label{prop:stabilityConditionsPeriodic}
    The scheme \eqref{eq:bulkSchemeAbstract} with periodic boundary conditions is stable in the $L^2$ norm if and only if $\relaxationParameter \in (0, 2)$ and $|\courantNumber|\leq 1$, or $\relaxationParameter = 2$ and $|\courantNumber|<1$.
\end{proposition}
Before proceeding, let us provide some bibliographical landmarks concerning the study of stability for \lbm{} schemes without boundary conditions. 
A widespread approach is the one \emph{à la von Neumann}, see \cite{benzi1992lattice, sterling1996stability, wissocq2019extended}, which is very close to the one we employ for \fd{} schemes, \confer{} \cite{bellotti2022finite}. It applies to linearized schemes and it hardly provides explicit stability conditions when the number of parameters in the numerical schemes increases. Since it uses the ``Fourier symbol'' of the scheme, it cannot handle problems with boundaries.
More recently weighted $L^2$ stability analyses based on the ``stability structure'' have been developed \cite{junk2009weighted,rheinlander2010stability}. They rely on the construction of a weighted norm where both transport and collision are contractions.
This notion of stability is intrinsically linear and does not take into account the joint role of collision and stream, thus gives only sufficient stability conditions. Moreover, it has not been successfully utilized to study our scheme with boundaries.
Also, the stability and convergence of scalar schemes have been analyzed using the concept of monotonicity, which allows to tackle non-linear problems focusing on the $L^1$ and $L^{\infty}$ norms.
This is the object of \cite{elton1995convergence, caetano2024result, bellotti2023monotonicity, aregba2024convergence, aregba2025monotonicity}. There is hope to extend this analysis to certain boundary conditions, linked to equilibria, see \cite{aregba2004kinetic}.
We also report other approaches, such as brute-force stability analyses \cite{simonis2021linear}, to the readers.

\subsection{GKS stability}\label{sec:GKS}

GKS theory stipulates that we can study the stability of each boundary condition by looking at the corresponding \strong{semi-infinite} problem. Presenting the case of the left boundary, we analyze $[0, +\infty)$,  hence we take $\indexSpace \in \naturals$.
Slightly different definitions of GKS stability exist, sometimes called ``strong stability'' \cite{coulombel2009stability, coulombel2011stability, boutin2024stability}.
We consider the one introduced in \cite[Definition 3.3]{gustafsson1972stability}, \cite[Definition 4.6]{trefethen1984instability}, and \cite[Definition 2]{boutin2024stability}.
In our context, it reads: it exists $\alpha_0 \geq 0$ and $C > 0$ such that, for $\timeStep, \spaceStep$ small enough
\begin{multline}\label{eq:GKS}
    \Bigl ( \frac{\alpha-\alpha_0}{1+\alpha \timeStep}\Bigr ) \timeStep \sum_{\indexTime = 2}^{+\infty} e^{-2\alpha \indexTime\timeStep} |\conservedMomentDiscrete_0^{\indexTime}|^2 + \Bigl ( \frac{\alpha-\alpha_0}{1+\alpha \timeStep}\Bigr )^2 \timeStep \spaceStep\sum_{\indexTime = 2}^{+\infty} \sum_{\indexSpace = 0}^{+\infty} e^{-2\alpha \indexTime\timeStep} |\conservedMomentDiscrete_{\indexSpace}^{\indexTime}|^2 \\
    \leq C \Bigl ( \Bigl ( \frac{\alpha-\alpha_0}{1+\alpha \timeStep}\Bigr ) \timeStep \sum_{\indexTime = 2}^{+\infty} e^{-2\alpha (\indexTime + 1)\timeStep} |g_0^{\indexTime}|^2 + \timeStep \spaceStep\sum_{\indexTime = 2}^{+\infty} \sum_{\indexSpace = 1}^{+\infty} e^{-2\alpha (\indexTime+1)\timeStep} |s_{\indexSpace}^{\indexTime}|^2 \Bigr ),
\end{multline}
for every $\alpha > \alpha_0$, defined for zero-initial data, meaning $\conservedMomentDiscrete_{\indexSpace}^{0} = \conservedMomentDiscrete_{\indexSpace}^{1} = 0$.
In \eqref{eq:GKS}, $g_0^{\indexTime}$ represents a source term in the boundary scheme, whereas $s_{\indexSpace}^{\indexTime}$ encodes a source term in the bulk scheme.
\begin{remark}
    \begin{itemize}
        \item Quantities in \eqref{eq:GKS} are summed throughout time, since this approach relies on a \strong{Laplace transformation}.
        \item The inequality \eqref{eq:GKS} contains decay factors $e^{-2\alpha \indexTime\timeStep}$. This could make solutions exploding with $\indexTime$ having finite sums.
        \item The previous stability estimate must hold only for $\timeStep$ and $\spaceStep$ \strong{small enough}. Few authors have pointed out this fact, see \cite{trefethen1984instability} and \cite{beam1982stability}.
        Cases are where a scheme is GKS-stable---for $\timeStep$ and $\spaceStep$ small enough---but instabilities are observed for large $\spaceStep$.
        \item It has been conjectured that GKS stability \strong{may imply} $L^2$ stability, \confer{} \cite{trefethen1984instability}.
    \end{itemize}

\end{remark}

The strength of GKS theory is that it allows theoretical stability/instability proofs by a \strong{normal mode analysis} as below, obtaining necessary and sufficient conditions.
Practically, to check whether \eqref{eq:GKS} holds, we first introduce the $\timeShiftOperator$-transformation, which reads $ \laplaceTransformed{\conservedMomentDiscrete}(\timeShiftOperator) = \sum_{\indexTime \in \naturals}\timeShiftOperator^{-\indexTime}\conservedMomentDiscrete^{\indexTime}$, assuming $\conservedMomentDiscrete^0 = 0$.
Before proceeding, let us point out that we only have to check the outflow condition, since the inflow condition in \eqref{eq:bulkFDScheme} is---under the stability conditions by \Cref{prop:stabilityConditionsPeriodic} and thanks to the so-called ``Goldberg-Tadmor lemma'' \cite{goldberg1981scheme} and \cite[Prop. 4.2]{coulombel:cel-00616497}---strongly stable.

\subsubsection{Inside the domain}

The $\timeShiftOperator$-transformed bulk \fd{} scheme \eqref{eq:bulkSchemeAbstract} is
\begin{equation}\label{eq:resolventEquation}
    \timeShiftOperator \laplaceTransformed{\conservedMomentDiscrete}_{\indexSpace}(\timeShiftOperator) = \schemeMatrixFDBlockZeroEntry_{-1}\laplaceTransformed{\conservedMomentDiscrete}_{\indexSpace-1}(\timeShiftOperator) + \schemeMatrixFDBlockZeroEntry_1 \laplaceTransformed{\conservedMomentDiscrete}_{\indexSpace+1}(\timeShiftOperator) + \schemeMatrixFDBlockMinusOneEntry_0 \timeShiftOperator^{-1} \laplaceTransformed{\conservedMomentDiscrete}_{\indexSpace}(\timeShiftOperator), \qquad \indexSpace \in \naturalsWithoutZero,
\end{equation}
sometimes named ``resolvent equation''.
Introducing the ansatz $\laplaceTransformed{\conservedMomentDiscrete}_{\indexSpace}(\timeShiftOperator) = \fourierShift^{\indexSpace}$ yields the \strong{characteristic equation}
\begin{equation}\label{eq:bulkCharEquation}
    \modifiedTimeShiftOperator(\timeShiftOperator) = \schemeMatrixFDBlockZeroEntry_{-1}\fourierShift^{-1} + \schemeMatrixFDBlockZeroEntry_{1}\fourierShift, \qquad \text{with} \qquad \modifiedTimeShiftOperator(\timeShiftOperator) \definitionEquality \timeShiftOperator - \schemeMatrixFDBlockMinusOneEntry_0 \timeShiftOperator^{-1}.
\end{equation}
Equation \eqref{eq:bulkCharEquation} should be intended as a \strong{quadratic equation} on $\fourierShift = \fourierShift(\timeShiftOperator)$.
It is also the dispersion relation of the scheme when taking\footnote{Here, $\omega$ must not be confused with the relaxation parameter.} $\timeShiftOperator = e^{i\omega\timeStep}$ and $\fourierShift = e^{i\frequency\spaceStep}$ for $\frequency, \omega \in \reals$.
Each mode $(\timeShiftOperator, \fourierShift) \in \complex \times \complex$---where $\timeShiftOperator$ gives the time dynamics and $\fourierShift$ the space structure---can be classified \cite[Chapter 2]{boutin:tel-04157587} according to its position with respect to the unit disk.
Notice that when $\schemeMatrixFDBlockZeroEntry_{1} \neq 0$, we can define $\commonTerm(\relaxationParameter, \courantNumber) \definitionEquality \frac{\schemeMatrixFDBlockZeroEntry_{-1}}{\schemeMatrixFDBlockZeroEntry_1}$, the product of the two roots of \eqref{eq:bulkCharEquation}.
The following kind of classification result is called Hersh's lemma \cite{hersh63mixed} by \cite{boutin2024stability} and its proof, given in the Supplementary material, relies on well-known perturbation arguments to distinguish $\solutionCharStable$ from $\solutionCharUnstable$.

\begin{lemma}[Description of the roots of the characteristic equation]\label{lemma:bulkStudy}
   Under the stability conditions given by \Cref{prop:stabilityConditionsPeriodic}.
    \begin{itemize}
        \item When $\schemeMatrixFDBlockZeroEntry_{\mp 1} = 0$, thus for $\relaxationParameter = \tfrac{2}{1\mp\courantNumber} \in [1, 2]$ (and $\mp\courantNumber\in[0, 1]$), the characteristic equation \eqref{eq:bulkCharEquation} has one solution $\fourierShift_{\pm}(\timeShiftOperator)$ such that $\pm|\fourierShift_{\pm}(\timeShiftOperator)| > 1$ for $|\timeShiftOperator|>1$.
        \item Otherwise, the characteristic equation \eqref{eq:bulkCharEquation} has two solutions $\solutionCharStable(\timeShiftOperator)$ and $\solutionCharUnstable(\timeShiftOperator)$, such that 
        \begin{align*}
            &|\solutionCharStable(\timeShiftOperator)| < 1, \qquad \text{for} \quad |\timeShiftOperator|>1, \qquad \text{(stable root)}, \\
            &|\solutionCharUnstable(\timeShiftOperator)| > 1, \qquad \text{for} \quad |\timeShiftOperator|>1, \qquad \text{(unstable root)}.
        \end{align*}
        We have 
        \begin{align*}
            &\solutionCharStable(\pm 1) = 
            \begin{cases}
                \pm \commonTerm(\relaxationParameter, \courantNumber), \qquad &\text{if}\quad \courantNumber < 0, \\
                \pm 1, \qquad &\text{if}\quad \courantNumber > 0.
            \end{cases} \qquad 
            &&\solutionCharUnstable(\pm 1) = 
            \begin{cases}
                \pm 1, \qquad &\text{if}\quad \courantNumber < 0, \\
                \pm \commonTerm(\relaxationParameter, \courantNumber), \qquad &\text{if}\quad \courantNumber > 0.
            \end{cases}\\
            &\solutionCharStable(\pm (1-\relaxationParameter)) = 
            \begin{cases}
                \pm 1, \qquad &\text{if}\quad \courantNumber < 0, \\
                \pm \commonTerm(\relaxationParameter, \courantNumber), \qquad &\text{if}\quad \courantNumber > 0.
            \end{cases} \qquad 
            &&\solutionCharUnstable(\pm (1-\relaxationParameter)) = 
            \begin{cases}
                \pm \commonTerm(\relaxationParameter, \courantNumber), \qquad &\text{if}\quad \courantNumber < 0, \\
                \pm 1, \qquad &\text{if}\quad \courantNumber > 0.
            \end{cases}
        \end{align*}
    \end{itemize}
\end{lemma}
Under the stability conditions from \Cref{prop:stabilityConditionsPeriodic}, the scheme is dissipative (but not totally dissipative, since $(\timeShiftOperator, \fourierShift) = (-1, -1) $ is a root of \eqref{eq:bulkCharEquation}) when $\relaxationParameter \in (0, 2)$, and non-dissipative ($|\timeShiftOperator| = 1$ for every $|\fourierShift| = 1$) when $\relaxationParameter = 2$.
Thanks to \cite[Lemma 3.2]{trefethen1984instability}, the group velocity of a given mode $(\targetEigenvalue, \targetFourier)$ fulfilling \eqref{eq:bulkCharEquation} with $|\targetEigenvalue| = |\targetFourier| = 1$ is defined regardless of the dissipativity of the scheme, using
\begin{equation*}
    \groupVelocity(\targetEigenvalue, \targetFourier) \definitionEquality - \latticeVelocity \frac{\targetFourier}{\targetEigenvalue} \frac{\differential\timeShiftOperator}{\differential\fourierShift}(\targetEigenvalue, \targetFourier),
\end{equation*}
For the values in \Cref{lemma:bulkStudy}, we obtain $\groupVelocity(\pm 1, \pm 1) =\latticeVelocity\courantNumber =  \advectionVelocity$.
Only in the non-dissipative case $\relaxationParameter = 2$, we can also consider $\groupVelocity(\pm 1, \mp 1) = -\advectionVelocity$.
Regardless of the sign of $\courantNumber$, the group velocity is positive for $\solutionCharStable$ (meaning it is right-going), and negative for $\solutionCharUnstable$ (left-going).

\begin{lemma}[Properties of $\commonTerm(\relaxationParameter, \courantNumber) $]\label{lemma:propertiesCommonTerm}
    Under the stability conditions given by \Cref{prop:stabilityConditionsPeriodic}, the function $\commonTerm(\relaxationParameter, \courantNumber)$ has the following properties.
    \begin{itemize}
        \item{{Domain of definition}.} For $\courantNumber < 0$, it is defined for all $\relaxationParameter$. For $\courantNumber > 0$, it is defined for all $\relaxationParameter$ except at $\relaxationParameter = \tfrac{2}{\courantNumber + 1} \in [1, 2]$, where a vertical asymptote exists such that $\lim_{\relaxationParameter \to ( \tfrac{2}{\courantNumber + 1})^{\pm}} \commonTerm(\relaxationParameter, \courantNumber)  = \mp\infty$.
        \item{{Monotonicity}.} For $\courantNumber < 0$ (resp., $\courantNumber > 0$) the function is monotonically decreasing (resp., increasing) in $\relaxationParameter$.
        \item{{Sign}.} 
        \begin{center}
            \begin{tabular}{p{7cm}p{7cm}}
                \multicolumn{1}{c}{For $\courantNumber < 0$} & \multicolumn{1}{c}{For $\courantNumber > 0$} \\ \vspace{-0.5cm}
                {\begin{align*}
                    \commonTerm(\relaxationParameter, \courantNumber) &\leq 0, \qquad \text{for} \quad \tfrac{2}{1-\courantNumber}\leq \relaxationParameter\leq 2, \\
                    \commonTerm(\relaxationParameter, \courantNumber) &> 0, \qquad \text{for} \quad 0 < \relaxationParameter < \tfrac{2}{1-\courantNumber}.
                \end{align*}} 
                &  \vspace{-0.5cm} {\begin{align*}
                    \commonTerm(\relaxationParameter, \courantNumber) &< 0, \qquad \text{for} \quad \tfrac{2}{1+\courantNumber}< \relaxationParameter\leq 2, \\
                    \commonTerm(\relaxationParameter, \courantNumber) &> 0, \qquad \text{for} \quad 0 < \relaxationParameter < \tfrac{2}{1+\courantNumber}.
                \end{align*}}
            \end{tabular}
        \end{center}
        
        \vspace{-.75cm}
        \item{{Key values and bounds}.}  We have $\commonTerm(0, \courantNumber) = 1$ and $\commonTerm(2, \courantNumber) = -1$.
        Moreover, $\sign(\courantNumber)|\commonTerm(\relaxationParameter, \courantNumber)|>\sign(\courantNumber)$ for $\relaxationParameter \in (0, 2)$.
    \end{itemize}
\end{lemma}

Now that we have classified and characterized the roots of the characteristic equation \eqref{eq:bulkCharEquation} through Lemmas \ref{lemma:bulkStudy} and 
\ref{lemma:propertiesCommonTerm}, we look for general solutions of the resolvent equation \eqref{eq:resolventEquation}, which---see \cite[Lemma 3.1]{trefethen1984instability}---are of the form
\begin{equation}\label{eq:generalSolutionResolvent}
    \laplaceTransformed{\conservedMomentDiscrete}_{\indexSpace}(\timeShiftOperator) = \coefficientStable (\timeShiftOperator)\solutionCharStable(\timeShiftOperator)^{\indexSpace} + \coefficientUnstable (\timeShiftOperator)\solutionCharUnstable(\timeShiftOperator)^{\indexSpace}, \qquad \indexSpace \in \naturals,
\end{equation}
as long as $\solutionCharStable(\timeShiftOperator)$ and $\solutionCharUnstable(\timeShiftOperator)$ are distinct.
Since we would like that $(\laplaceTransformed{\conservedMomentDiscrete}_{\indexSpace}(\timeShiftOperator))_{\indexSpace\in\naturals}  \in L^2(\spaceStep\naturals)$ for $|\timeShiftOperator|\geq 1$, we take $\coefficientUnstable (\timeShiftOperator)= 0$.
This yields a so-called ``admissible solution'' \cite[Chapter 11]{strikwerda2004finite} and fulfills the following definition.
\begin{definition}[Admissible solution \cite{strikwerda2004finite}]\label{def:admissibleSolution}
    Let $\laplaceTransformed{\conservedMomentDiscrete}_{\indexSpace}(\timeShiftOperator)$ fulfill the resolvent equation \eqref{eq:resolventEquation}.
    Then, $\laplaceTransformed{\conservedMomentDiscrete}_{\indexSpace}(\timeShiftOperator)$ is said to be an ``admissible solution'' if, when $|\timeShiftOperator|>1$, then $(\laplaceTransformed{\conservedMomentDiscrete}_{\indexSpace}(\timeShiftOperator))_{\indexSpace\in\naturals}  \in L^2(\spaceStep\naturals)$, whereas, once $|\timeShiftOperator| = 1$, then $\laplaceTransformed{\conservedMomentDiscrete}_{\indexSpace}(\timeShiftOperator)$ is the limit of an admissible solution defined for arguments outside the closed unit disk. Said differently, when  $|\timeShiftOperator| = 1$, $\laplaceTransformed{\conservedMomentDiscrete}_{\indexSpace}(\timeShiftOperator) = \lim_{\epsilon \to 0^{+}} \laplaceTransformed{\discrete{w}}_{\indexSpace}(\timeShiftOperator (1+\epsilon))$ with $(\laplaceTransformed{\discrete{w}}_{\indexSpace}(\timeShiftOperator (1+\epsilon)))_{\indexSpace\in\naturals} \in L^2(\spaceStep\naturals)$ for all $\epsilon > 0$.
\end{definition}

\subsubsection{At the outflow boundary}

The admissible solution features---in our case---one free parameter $\coefficientStable (\timeShiftOperator)$, fixed considering the $\timeShiftOperator$-transformation of the boundary scheme \eqref{eq:outflowSchemeAbstract}
\begin{equation}\label{eq:boundaryTransformed}
    \timeShiftOperator \laplaceTransformed{\conservedMomentDiscrete}_0(\timeShiftOperator) = \sum_{\indexSpace = 0}^{\numberCoefficientsOutFlowFDZero - 1}\coefficientOutflowFDZero_{\indexSpace}\laplaceTransformed{\conservedMomentDiscrete}_{\indexSpace}(\timeShiftOperator) + \timeShiftOperator^{-1}\sum_{\indexSpace = 0}^{\numberCoefficientsOutFlowFDMinusOne - 1}\coefficientOutflowFDMinusOne_{\indexSpace}\laplaceTransformed{\conservedMomentDiscrete}_{\indexSpace}(\timeShiftOperator).
\end{equation}
The so-called ``eigenvalue problem'', \confer{} \cite[Equation (13.1.35)]{gustafsson2013time} is made up of the transformed bulk equation \eqref{eq:resolventEquation} plus the transformed boundary scheme \eqref{eq:boundaryTransformed}.
An admissible solution, satisfying the boundary equation \eqref{eq:boundaryTransformed}, and not identically zero is called ``eigensolution'', \confer{} \cite{vilar2015development, li2022stability}.
Eventually \strong{GKS stability/instability} is checked using the following result.
\begin{theorem}[Theorem 11.3.3 from \cite{strikwerda2004finite}]
    A scheme with boundary, tackling the advection equation and $L^2$-stable in its periodic version, is GKS-stable (\emph{i.e.} \eqref{eq:GKS} holds) if and only if it does not admit any non-trivial admissible solution satisfying the homogeneous boundary condition (\emph{i.e.} an eigensolution).
\end{theorem}

Instabilities arising at the outflow are thus associated with right-going modes propagating inside the domain.
We are now able to study GKS stability for the outflow boundary conditions that we have considered.
For \eqref{eq:extrapolationBoundaryCondition}, the following result allows to conclude.
Its proof is standard \cite{strikwerda2004finite}, given in the Supplementary material, and relies on inserting $\laplaceTransformed{\conservedMomentDiscrete}_{\indexSpace}(\timeShiftOperator) = \fourierShift^{\indexSpace}$ into \eqref{eq:boundaryTransformed}.
\begin{lemma}\label{lemma:boundaryRoots}
    For $\orderExtrapolation\geq 1$, the eigenvalue problems associated with \eqref{eq:extrapolationBoundaryCondition} have the following roots: $(\timeShiftOperator, \fourierShift) = (1, 1)$ and $(\timeShiftOperator, \fourierShift) = (\relaxationParameter - 1, -\commonTerm(\relaxationParameter, \courantNumber))$, and for $\orderExtrapolation \geq 2$, $(\timeShiftOperator, \fourierShift) = (1-\relaxationParameter, 1)$.
    Moreover, when $\orderExtrapolation = 1, 2$, all the roots of the eigenvalue problems are those listed above. 
\end{lemma}

The only thing left to understand---utilizing \Cref{lemma:bulkStudy}---is whether the roots in \Cref{lemma:boundaryRoots} correspond to the strictly right-going $\solutionCharStable$, leading to a GKS-unstable scheme, or to $\solutionCharUnstable$, ensuring stability.
\begin{proposition}[GKS stability-instability of \eqref{eq:extrapolationBoundaryCondition}]\label{prop:stabInstGKS}
    The GKS stability of the boundary conditions \eqref{eq:extrapolationBoundaryCondition} is as follows.
        \begin{align*}
            \bullet\quad\text{For}~\orderExtrapolation = 1.\qquad&\text{When}~\courantNumber < 0, \qquad \text{GKS-stable}.\\
            &\text{When}~\courantNumber > 0, \qquad \text{GKS-unstable, with unstable mode}~ \timeShiftOperator = 1, \quad \fourierShift=\solutionCharStable = 1.
        \end{align*}
        \begin{align*}
            \bullet\quad\text{For}~\orderExtrapolation = 2.\qquad&\text{When}~\courantNumber < 0,  &&\text{GKS-stable for}~\relaxationParameter \in (0, 2).\\
            & &&\text{GKS-unstable for}~\relaxationParameter = 2~\text{with unstable mode}~\timeShiftOperator = -1, \quad \fourierShift=\solutionCharStable = 1.\\
            &\text{When}~\courantNumber > 0,  &&\text{GKS-unstable, with unstable mode}~ \timeShiftOperator = 1, \quad \fourierShift=\solutionCharStable = 1.
        \end{align*}
        \begin{align*}
            \bullet\quad\text{For}~\orderExtrapolation \geq 3.\qquad&\text{When}~\courantNumber < 0,  &&\text{(probably) GKS-stable for}~\relaxationParameter \in (0, 2).\\
            & &&\text{GKS-unstable for}~\relaxationParameter = 2~\text{with unstable mode}~\timeShiftOperator = -1, \quad \fourierShift=\solutionCharStable = 1.\\
            &\text{When}~\courantNumber > 0,  &&\text{GKS-unstable, with unstable mode}~ \timeShiftOperator = 1, \quad \fourierShift=\solutionCharStable = 1.
        \end{align*}
\end{proposition}
The  result when $\orderExtrapolation = 1$ is interesting, for a Neumann boundary condition applied to a leap-frog scheme is \strong{unstable} \cite{trefethen1984instability, coulombel2020neumann}.
Now, the boundary condition stemming from a Neumann boundary condition on the \strong{incoming distribution function} $\distributionFunctionDiscrete^{+}$ is different from the same condition on  $\conservedMomentDiscrete$.
Differently, employing the outflow conditions when the boundary is unfortunately an inflow ($\courantNumber > 0$), entails instability regardless of the value of $\relaxationParameter$.

We now tackle the analysis of \eqref{eq:conditionKinRod}.
Regrettably, the procedure does not give fully analytical results but is complemented by numerical studies.
The transformed boundary equation reads $(\courantNumber + 1) ( \tfrac{1}{2} (2-\relaxationParameter + \relaxationParameter\courantNumber) + \tfrac{1}{2}(2-\relaxationParameter - \relaxationParameter\courantNumber)\fourierShift(\timeShiftOperator)^2) + \dots = 0$, where $\dots$ terms are not made explicit for compactness.
In the listed term, we recognize part of the bulk equation \eqref{eq:bulkCharEquation}. Replacing this part gives a first order equation in $\fourierShift(\timeShiftOperator)$, providing
    \begin{equation}\label{eq:tmp5}
        \fourierShift(\timeShiftOperator) = 
        \frac{(\courantNumber + 1)\relaxationParameter^2 - (\courantNumber + 1)\relaxationParameter-2\timeShiftOperator^2}{(\courantNumber+1)\relaxationParameter^2 - (\courantNumber + 1)\timeShiftOperator^2 - 2(\courantNumber + 1)\relaxationParameter + ((\courantNumber + 1)\relaxationParameter - 2)\timeShiftOperator + \courantNumber + 1}.
    \end{equation}
    Inserting this expression into \eqref{eq:bulkCharEquation} gives a sixth-order polynomial equation in $\timeShiftOperator$.
    Explicit computations show that $\timeShiftOperator = 1, \pm(1-\relaxationParameter)$ are three out of the six roots of this polynomial.
    Inserting these values into \eqref{eq:tmp5} gives 
    \begin{equation}\label{eq:tmp6}
        \fourierShift(1) = 1, \qquad \fourierShift(1-\relaxationParameter) = \commonTerm(\relaxationParameter, \courantNumber), \qquad   \fourierShift(\relaxationParameter - 1) = - \commonTerm(\relaxationParameter, \courantNumber).
    \end{equation}
    \begin{itemize}
        \item When $\courantNumber < 0$, the roots from \eqref{eq:tmp6} correspond, for every unstable candidate $\timeShiftOperator = 1, \pm( 1-\relaxationParameter)$, to $\solutionCharUnstable$ rather than $\solutionCharStable$. These modes are consequently stable---as they are not admissible eigensolutions.
        We are left to check the \strong{remaining three roots} to conclude on stability.
        \item When $\courantNumber > 0$, \eqref{eq:tmp6} gives to a full set of admissible eigensolutions, hence the boundary condition is \strong{GKS-unstable}, and we can end the study here.
    \end{itemize}
    Assume that $\courantNumber < 0$, where we hope proving GKS stability.
    Factoring the roots $\timeShiftOperator = 1, \pm(1-\relaxationParameter)$ out from the sixth-order equation, we are left with a third-order equation, which reads, excluding the trivial case where $\courantNumber = -1$:
    \begin{multline}\label{eq:tmp7}
        4 \, \timeShiftOperator^{3}- {\left(\courantNumber^{2} {\relaxationParameter} + 2 \, \courantNumber - {\relaxationParameter} - 2\right)} \timeShiftOperator^{2}  
        + {\left(2 \, \courantNumber^{2} {\relaxationParameter}^{2} - \courantNumber^{2} {\relaxationParameter} - 2 \, \courantNumber {\relaxationParameter}^{2} + 2 \, \courantNumber {\relaxationParameter} - 4 \, {\relaxationParameter}^{2} - 2 \, \courantNumber + 7 \, {\relaxationParameter} - 2\right)} \timeShiftOperator \\
         -2 \, \courantNumber {\relaxationParameter}^{3} + 4 \, \courantNumber {\relaxationParameter}^{2} - 2 \, {\relaxationParameter}^{3}  - 2 \, \courantNumber {\relaxationParameter} + 4 \, {\relaxationParameter}^{2} - 2 \, {\relaxationParameter} = 0.
    \end{multline}
    Unfortunately, is it hard to analytically study the roots of \eqref{eq:tmp7} for general $\relaxationParameter$ and $\courantNumber$.
    We can numerically check that the condition is stable for given values of $\relaxationParameter$, see \Cref{fig:kinrodGKS}.
    Here, the ``solitary'' real root $\timeShiftOperator_3$ always lays inside the unit disk without causing instability.
    The same holds for the two complex conjugate roots $\timeShiftOperator_1$ and $\timeShiftOperator_2$ for small $|\courantNumber|$.
    When their modulii exceed one (the GKS procedure only studies $\timeShiftOperator\in\complex$ such that $|\timeShiftOperator|\geq 1$) for $|\courantNumber|$ close to one, the modulii of the corresponding $\fourierShift$'s from \eqref{eq:tmp5} are larger than one, so they are indeed not $\solutionCharStable$, but $\solutionCharUnstable$. The boundary condition is thus \strong{stable}.

    \begin{figure} 
        \begin{center}
            \includegraphics[width = 0.48\textwidth]{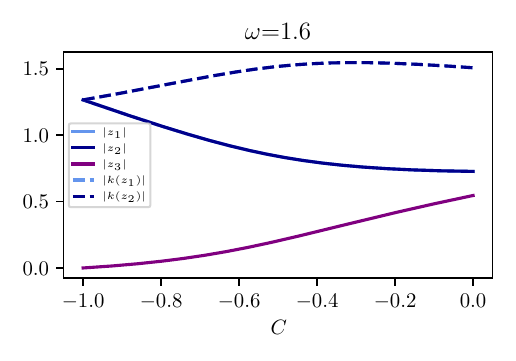}
            \includegraphics[width = 0.48\textwidth]{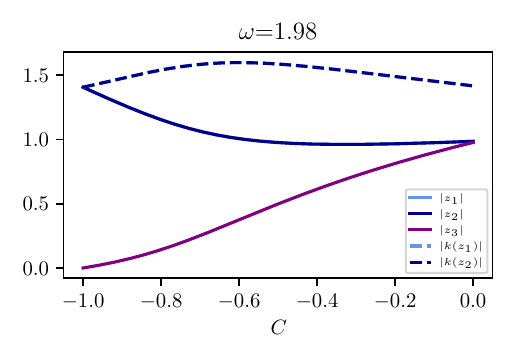} 
        \end{center}\caption{\label{fig:kinrodGKS}Modulus of the solutions of \eqref{eq:tmp7} and corresponding $\fourierShift$ obtained by \eqref{eq:tmp5}.}
    \end{figure}

\subsubsection{Reflection coefficients}\label{sec:reflectionCoefficient}

We finally introduce \strong{reflection coefficients} \cite{trefethen1984instability}, a powerful tool and subsequent evolution of GKS theory.
On the left boundary, inserting the general solution of the resolvent equation \eqref{eq:generalSolutionResolvent} into the transformed boundary condition \eqref{eq:boundaryTransformed}, the reflection coefficient is defined as the ratio between the amplitude $\coefficientStable(\timeShiftOperator)$ radiating from the boundary towards the inner domain, and the incident amplitude $\coefficientUnstable(\timeShiftOperator)$ from the bulk towards the boundary:
\begin{equation}\label{eq:reflectionCoefficient}
    \reflectionCoefficientLetterOut(\timeShiftOperator)\definitionEquality \frac{\coefficientStable(\timeShiftOperator)}{\coefficientUnstable(\timeShiftOperator)} = -\frac{\timeShiftOperator  - \sum_{\indexSpace = 0}^{\numberCoefficientsOutFlowFDZero - 1}\coefficientOutflowFDZero_{\indexSpace}\solutionCharUnstable(\timeShiftOperator)^{\indexSpace} - \timeShiftOperator^{-1}\sum_{\indexSpace = 0}^{\numberCoefficientsOutFlowFDMinusOne - 1}\coefficientOutflowFDMinusOne_{\indexSpace}\solutionCharUnstable(\timeShiftOperator)^{\indexSpace}}{\timeShiftOperator  - \sum_{\indexSpace = 0}^{\numberCoefficientsOutFlowFDZero - 1}\coefficientOutflowFDZero_{\indexSpace}\solutionCharStable(\timeShiftOperator)^{\indexSpace} - \timeShiftOperator^{-1}\sum_{\indexSpace = 0}^{\numberCoefficientsOutFlowFDMinusOne - 1}\coefficientOutflowFDMinusOne_{\indexSpace}\solutionCharStable(\timeShiftOperator)^{\indexSpace}}.
\end{equation}
The reflection coefficient is particularly interesting when $\timeShiftOperator$ is an unstable mode.
For the right boundary, we have
\begin{equation}\label{eq:reflectionCoefficientIn}
    \reflectionCoefficientLetterIn(\timeShiftOperator) \definitionEquality \frac{\coefficientUnstable(\timeShiftOperator) }{\coefficientStable(\timeShiftOperator)} =-\Bigl (\frac{\solutionCharStable(\timeShiftOperator)}{\solutionCharUnstable(\timeShiftOperator)}\Bigr )^{\numberSpacePoints - 1} = - \commonTerm(\relaxationParameter, \courantNumber)^{-\numberSpacePoints + 1} \solutionCharStable(\timeShiftOperator)^{2\numberSpacePoints - 2},
\end{equation}
when $\timeShiftOperator\neq0$, by virtue of the fact that, except when $\commonTerm$ is not defined, $\solutionCharStable(\timeShiftOperator)\solutionCharUnstable(\timeShiftOperator) = \commonTerm(\relaxationParameter, \courantNumber)$.

\subsection{Numerical simulations}\label{sec:numericalSimulationsStability}

Knowing whether boundary conditions are GKS-stable/unstable, we numerically assess  the relevance of this analysis on the actual behavior of \lbm{} schemes.
We simulate on a domain of length $\domainLength = 1$.
The initial datum is $\distributionFunctionDiscrete_{\indexSpace}^{\pm, 0} = \tfrac{1}{2}(1\pm\courantNumber)\mathds{1}_{\indexSpace=1}$ and is chosen to quickly develop possible instabilities. It represents a perturbation close to the left boundary.
Sources are switched off both at the left and right boundary, \emph{i.e.} $\sourceTermBoundaryDiscrete_0 \equiv 0$ and $\boundaryDatumInflow\equiv 0$. 
We start using Courant number $\courantNumber = -1/2$.

\begin{figure} 
    \definecolor{colorS2}{RGB}{246,156,115}
    \definecolor{colorS198}{RGB}{232,62,62}
    \definecolor{colorS196}{RGB}{161,25,91}
    \definecolor{colorS194}{RGB}{76,29,74}
    \begin{center}
        \includegraphics[width = 1.\textwidth]{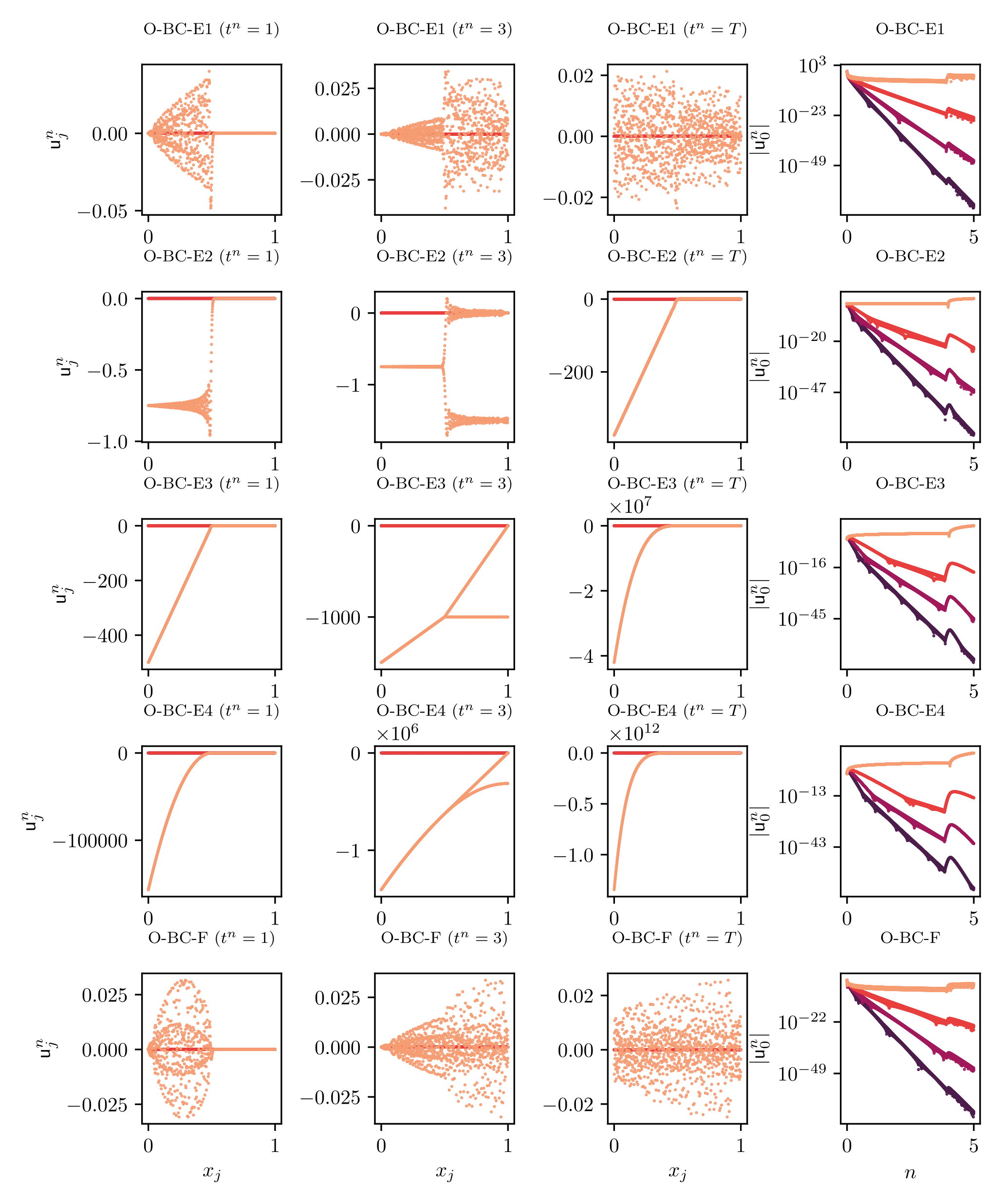}
    \end{center}\caption{\label{fig:instab_neg_vel}Snapshots of the solution (first three columns) and absolute value of the solution at the left boundary node as function of time (last column). Here, we employ $\courantNumber  = -1/2< 0$ and $\numberSpacePoints = 1000$. Colors: \textcolor{colorS2}{$\bullet$ for $\relaxationParameter = 2$}, \textcolor{colorS198}{$\bullet$ for $\relaxationParameter = 1.98$}, \textcolor{colorS196}{$\bullet$ for $\relaxationParameter = 1.96$}, and \textcolor{colorS194}{$\bullet$ for $\relaxationParameter = 1.94$}. In the first three columns, data for $\relaxationParameter<2$ are all superimposed.}
\end{figure}
\begin{figure} 
    \definecolor{colorS198}{RGB}{245,136,95}
    \definecolor{colorS196}{RGB}{203,26,79}
    \definecolor{colorS194}{RGB}{96,30,83}
    \begin{center}
        \includegraphics[width = 1.\textwidth]{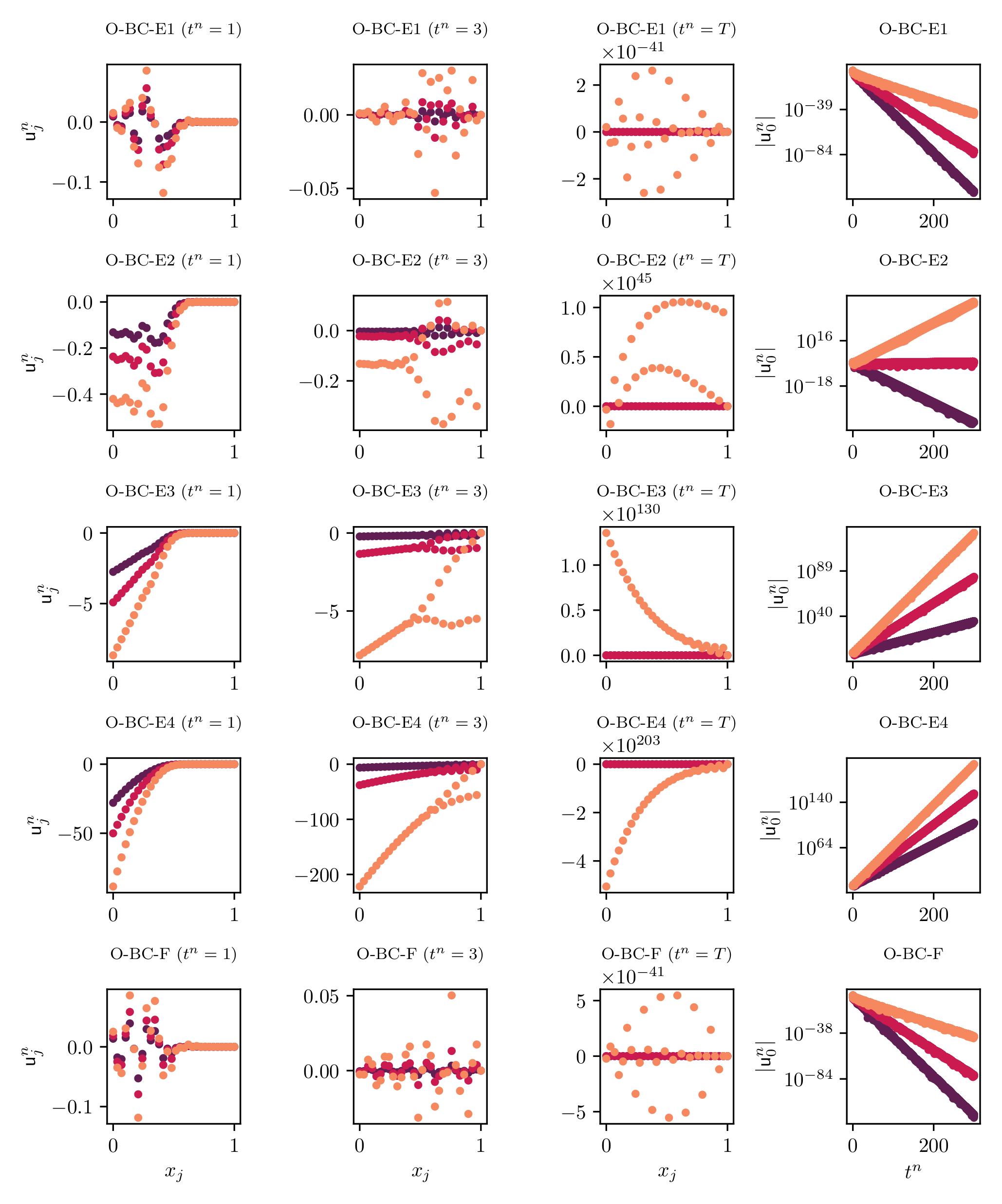}
    \end{center}\caption{\label{fig:instab_neg_vel_few_points}Snapshots of the solution (first three columns) and absolute value of the solution at the left boundary node as function of time (last column). Here, we employ $\courantNumber  = -1/2< 0$ and $\numberSpacePoints = 30$. Colors: \textcolor{colorS198}{$\bullet$ for $\relaxationParameter = 1.98$}, \textcolor{colorS196}{$\bullet$ for $\relaxationParameter = 1.96$}, and \textcolor{colorS194}{$\bullet$ for $\relaxationParameter = 1.94$}.}
\end{figure}

\begin{itemize}
    \item \strong{Fine resolution}: $\numberSpacePoints = 1000$ and $\finalTime = 5$.
    In \Cref{fig:instab_neg_vel}, we plot snapshots of the solution and the absolute value of the solution at the boundary cell $|\conservedMomentDiscrete_0^{\indexTime}|$ as function of time, for $\relaxationParameter \in (0, 2]$.

    Looking at the snapshot at time $\timeGridPoint{\indexTime} = 1$, the results are \strong{in agreement} with the GKS stability analysis, \confer{} \Cref{prop:stabInstGKS}: both \eqref{eq:papillon1} and \eqref{eq:conditionKinRod} are stable for any value of $\relaxationParameter$ and do not feature a stationary perturbation at the outflow, whereas \eqref{eq:extrapolationBoundaryCondition} for $\orderExtrapolation \geq 2$ is unstable for $\relaxationParameter = 2$ and stable otherwise. 
    In the unstable case $\relaxationParameter = 2$, the stationary perturbation at the outflow trends like $|\conservedMomentDiscrete_0^{\indexTime}| \propto \indexTime^{\orderExtrapolation - 2}$ for $\indexTime \lesssim \tfrac{2(\numberSpacePoints - 1)}{|\courantNumber|}$, thus grows when $\orderExtrapolation\geq 3$.
    This is unsurprising considering that $\reflectionCoefficientLetterOut(-1) = \infty$, and reads $\reflectionCoefficientLetterOut(\timeShiftOperator) \sim (\timeShiftOperator + 1)^{1-\orderExtrapolation} + \bigO{(\timeShiftOperator + 1)^{2-\orderExtrapolation}}$ about $\timeShiftOperator = -1$.
    The instability propagates to the right with group velocity $\groupVelocity(-1, 1) =-\advectionVelocity>0$, brought by the strictly right-going mode $\solutionCharStable$.
    If the domain were infinite to the right, we would see a growth of the $L^2$ norm $\propto \indexTime^{\orderExtrapolation - 3/2}$. 
    For $\orderExtrapolation = 2$, this would yield the claim by \cite[Theorem 2a]{trefethen1984instability}, since the unstable mode is strictly right-going. However, similarly to the results in \cite[Demonstration 4.1]{trefethen82phd}, this goes \strong{against} the prediction of \cite[Theorem 3a]{trefethen1984instability}, which applies because of the infinite reflection coefficient.
    This discrepancy is difficult to explain---except that maybe we do not employ the right initial datum to observe the expected growth---but has already been documented in the literature in similar contexts.

    Looking at the next snapshot at $\timeGridPoint{\indexTime} = 3$, in the case $\relaxationParameter = 2$, we see that the inflow reflects the incoming perturbation brought by $\solutionCharStable$ into one carried by $\solutionCharUnstable$, and those two interact.
    For instance, $\reflectionCoefficientLetterIn(-1) = -\commonTerm(2, -1/2)^{-\numberSpacePoints + 1}\solutionCharStable(-1)^{2\numberSpacePoints - 2} = (-1)^{\numberSpacePoints}$: the inflow reflects (though without amplification) the unstable mode generated by the inflow, and therefore triggers the \strong{exponential} growth envisioned by \cite[Proposition 8]{trefethen1985stability}, also empirically noticeable on the last column of \Cref{fig:instab_neg_vel}.

    \item \strong{Coarse resolution}: $\numberSpacePoints = 30$ and $\finalTime = 300$. 
    The results in \Cref{fig:instab_neg_vel_few_points} are provided for $\relaxationParameter < 2$.

    Both \eqref{eq:papillon1} and \eqref{eq:conditionKinRod} are stable.
    Still, quite surprisingly, \eqref{eq:extrapolationBoundaryCondition} for $\orderExtrapolation \geq 2$ may be unstable for values of $\relaxationParameter < 2$ quite close to two, which is \strong{not predicted} by the GKS theory. 
    Otherwise said, as heralded by \cite{trefethen1985stability}, GKS-stability does not imply P-stability (this notion is introduced in the sequel).
    Even at the beginning of the computation (\confer{} snapshot at $\timeGridPoint{\indexTime} = 1$), where the inflow does not play any role, we witness $|\conservedMomentDiscrete_0^{\indexTime}| \propto \indexTime^{\orderExtrapolation-2}$ for $\indexTime$ small (here $\indexTime \lesssim \tfrac{2(\numberSpacePoints - 1)}{|\courantNumber|}$ is not restrictive enough on this coarse mesh, for actual waves may travel faster than expected).
    For this is a transient behavior, pseudo-spectra are a relevant tool, see \Cref{sec:plotsPseudoSpectra}.
    
    Eventually, for the unstable simulations, the growth is \strong{exponential} in $\indexTime$, having $\reflectionCoefficientLetterOut(1-\relaxationParameter) = \infty$ and $\reflectionCoefficientLetterIn(1-\relaxationParameter) = -\commonTerm(\relaxationParameter, -1/2)^{-\numberSpacePoints+1}$ with $-1<\commonTerm(\relaxationParameter, -1/2) \approx -1$. This can again be explained using repeated reflections between boundaries, and  is also studied in \Cref{sec:plotsPseudoSpectra} with spectra.
\end{itemize}

\begin{figure} 
    \definecolor{colorS16}{RGB}{203,26,79}
    \begin{center}
        \includegraphics[width = 1.\textwidth]{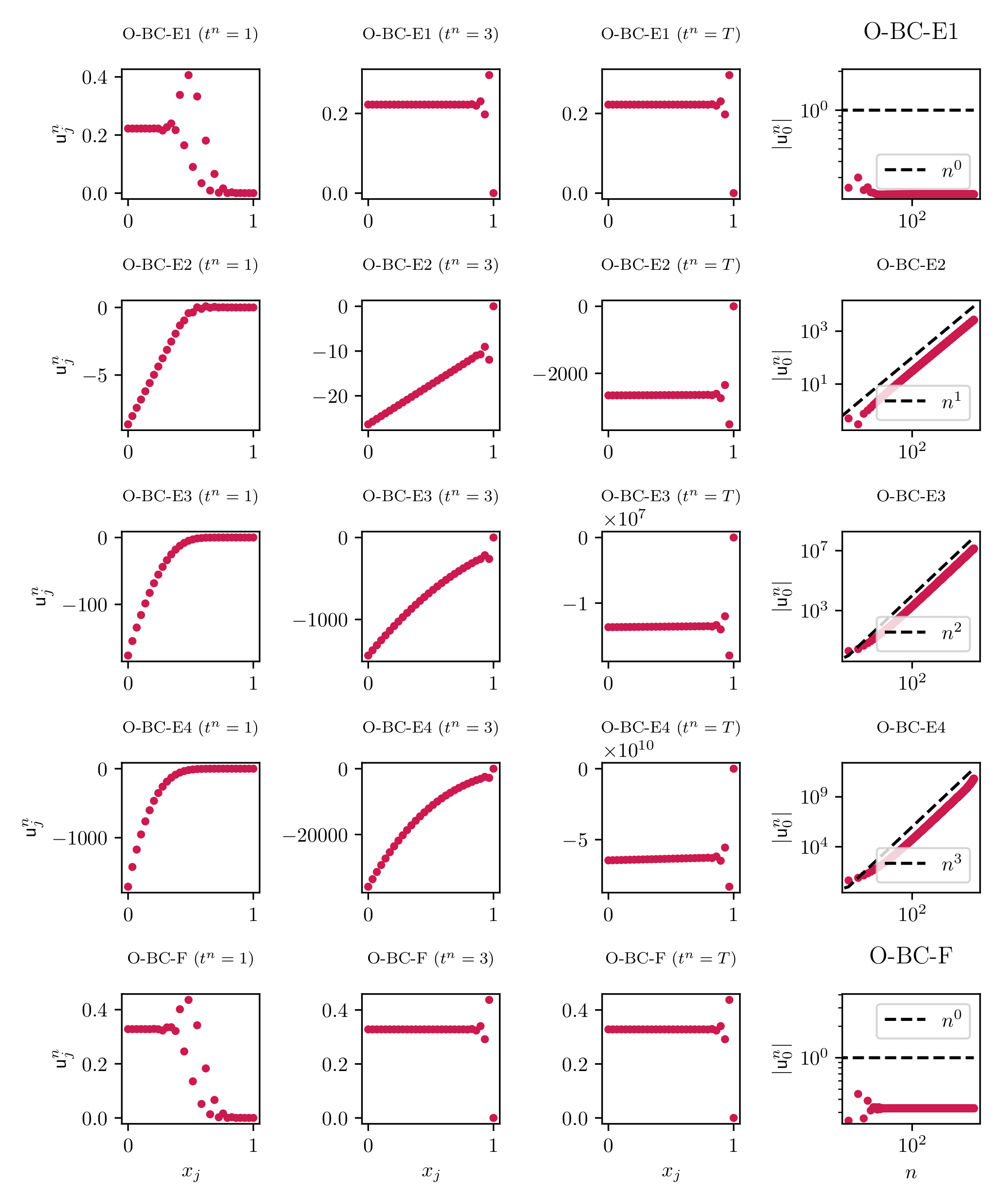}
    \end{center}\caption{\label{fig:instab_pos_vel_few_points}Snapshots of the solution (first three columns) and absolute value of the solution at the left boundary node as function of time (last column). Here, we employ $\courantNumber  = 1/2> 0$ and $\numberSpacePoints = 30$. Color: \textcolor{colorS16}{$\bullet$ for $\relaxationParameter = 1.6$}}
\end{figure}

Even if not much physically funded, we now consider $\courantNumber = 1/2$.
We employ a coarse resolution $\numberSpacePoints = 30$ and a final time $\finalTime = 300$.
The results are given in \Cref{fig:instab_pos_vel_few_points} using $\relaxationParameter = 1.6$.

We observe that $|\conservedMomentDiscrete_0^{\indexTime}| \propto \indexTime^{\orderExtrapolation-1}$ also for $\indexTime \gg 1$, whence boundaries have already interacted by repeated reflections. These growth rates are explained in \Cref{sec:reflCoeff} and \ref{sec:plotsPseudoSpectra} bridging spectral features, which are inherent to the long-time behavior of the scheme, with \strong{reflection coefficients}.
Formal computations show that---for $\courantNumber > 0$ and $\relaxationParameter \in (0, 2)$---they are
    \begin{align}
        &\reflectionCoefficientLetterOut(\timeShiftOperator) \sim (\timeShiftOperator - 1)^{-\orderExtrapolation} + \bigO{(\timeShiftOperator - 1)^{-\orderExtrapolation+1}}, \qquad &&\text{for}\quad \eqref{eq:extrapolationBoundaryCondition},\label{eq:tmp18} \\
        &\reflectionCoefficientLetterOut(\timeShiftOperator) \sim (\timeShiftOperator - 1)^{-1} + \bigO{1}, \qquad &&\text{for}\quad \eqref{eq:conditionKinRod}, \label{eq:tmp19}
    \end{align}
    close to $\timeShiftOperator = 1$.
    Despite these infinite reflection coefficients, the observed polynomial trends in $\indexTime$ are far from the catastrophic exponential growth predicted by \cite[Propositions 8]{trefethen1985stability}.
    This mild behavior stems from the inflow boundary condition that does not allow catastrophic repeated reflections take place.
    Indeed, $\reflectionCoefficientLetterIn(1) = - \commonTerm(\relaxationParameter, \courantNumber)^{-\numberSpacePoints + 1}$.
    For $\courantNumber>0$, we have $|\commonTerm(\relaxationParameter, \courantNumber)|>1$, thus $\lim_{\numberSpacePoints\to+\infty}\reflectionCoefficientLetterIn(1)=0$ exponentially, thus $\reflectionCoefficientLetterIn(1) \approx 0$ at finite $\numberSpacePoints$.
    This is precisely what the snapshot at $\timeGridPoint{\indexTime} = 3$ in \Cref{fig:instab_pos_vel_few_points} conveys: no reflected wave originates from the right boundary.
    
Somewhat surprisingly in view of their GKS-instability, we see that \eqref{eq:papillon1} and \eqref{eq:conditionKinRod} remain stable, because we operate on a bounded domain and the unstable mode lodged at the left boundary does not increase in time, \confer{} $|\conservedMomentDiscrete_0^{\indexTime}| \sim 1$, with the inflow boundary that does not reflect.
On a semi-infinite domain to the right, we would witness a growth of the $L^2$ norm $\propto \indexTime^{\orderExtrapolation - 1/2}$.
In the case $\orderExtrapolation = 1$, this is milder than the prediction of \cite[Theorem 3a]{trefethen1984instability}, which is relevant here.
Despite this dissimilitude, the outcomes of the simulations can be perfectly and very convincingly explained by the developments of the following section.

\subsection{Matrix method and pseudo-spectra}\label{sec:matrixMethod}

To explain the link between the order of poles of the reflection coefficient and order-of-growth of instabilities, as well as unexpected instabilities at low resolutions, we use the so-called ``matrix method'' \cite{sousa2009edge, vilar2015development, dakin2018inverse, lebarbenchon:tel-04214887}, where the numerical scheme is \strong{represented by a matrix}, whose powers we would like to bound. 
The main difficulty is to do so uniformly in $\numberSpacePoints$.
We also provide plots of the so-called pseudo-spectra \cite{trefethen2005spectra}.
Let us first introduce different matrices.

\begin{itemize}
    \item \strong{Original \lbm{} matrix}. The iteration matrix associated with the original \lbm{} scheme---endowed with boundary conditions---can be written both using the distribution functions $\distributionFunctionDiscrete^{\pm}$ or based on $\conservedMomentDiscrete$ and $\nonConservedMomentDiscrete$.
    Introducing $\twoVariablesSolutionVector^{\indexTime} = \transpose{(\distributionFunctionDiscrete_0^{+, \indexTime}, \dots, \distributionFunctionDiscrete_{\numberSpacePoints-1}^{+, \indexTime}, \distributionFunctionDiscrete_0^{-, \indexTime}, \dots, \distributionFunctionDiscrete_{\numberSpacePoints-1}^{-, \indexTime})}$, we write 
    \begin{equation*}
        \twoVariablesSolutionVector^{\indexTime + 1} = \schemeMatrixLBM \twoVariablesSolutionVector^{\indexTime}, \qquad \textnormal{with} \qquad
        \schemeMatrixLBM = 
        \left [
        \begin{array}{c|c}
            \schemeMatrixLBMBlock{++} & \schemeMatrixLBMBlock{+-} \\
            \hline
            \schemeMatrixLBMBlock{-+} & \schemeMatrixLBMBlock{--} 
        \end{array}
        \right ] \in \matrixSpace{2\numberSpacePoints}{\reals}, 
    \end{equation*}
    whose blocks are given in \Cref{app:exprMatrixLBM}.
    The matrix $\schemeMatrixLBM$ is block banded \emph{quasi}-Toeplitz.
    \begin{remark}[Zero eigenvalues of $\schemeMatrixLBM$]\label{rem:zeroEigLBM}
        When using \eqref{eq:extrapolationBoundaryCondition} and for $\numberSpacePoints$ large enough compared to $\orderExtrapolation$, the matrix $\schemeMatrixLBM$ is singular, both due to the outflow boundary condition and the inflow boundary condition.
    \end{remark}

    We cannot use the Perron-Frobenius theorem to characterize the spectrum $\spectrum(\schemeMatrixLBM)$ of $\schemeMatrixLBM$ since the matrix features negative coefficients due to the inflow condition, and is reducible.

    \item \strong{Corresponding \fd{} matrix}. 
    Introducing the vector $\twoStepSolutionVector^{\indexTime} = \transpose{(\conservedMomentDiscrete_0^{\indexTime}, \dots, \conservedMomentDiscrete_{\numberSpacePoints-1}^{\indexTime}, \conservedMomentDiscrete_0^{\indexTime - 1}, \dots, \conservedMomentDiscrete_{\numberSpacePoints-1}^{\indexTime - 1})}$ across two time-steps, we write the corresponding \fd{} scheme with vanishing boundary data as 
    \begin{equation}\label{eq:schemeMatricialFD}
        \twoStepSolutionVector^{\indexTime + 1} = \schemeMatrixFD \twoStepSolutionVector^{\indexTime}, \qquad \textnormal{with} \qquad
        \schemeMatrixFD = 
        \left [
        \begin{array}{c|c}
            \schemeMatrixFDBlockZero & \schemeMatrixFDBlockMinusOne \\
            \hline
            \identityMatrix{\numberSpacePoints} & \nullMatrix{\numberSpacePoints}
        \end{array}
        \right ] \in \matrixSpace{2\numberSpacePoints}{\reals}.
    \end{equation}
    The matrices $\schemeMatrixFDBlockZero$ and $\schemeMatrixFDBlockMinusOne$ are \emph{quasi}-Toeplitz (at most tridiagonal in their Toeplitz part) given by 
    \begin{equation*}
        \schemeMatrixFDBlockZero = 
        \begin{pmatrix}
            \coefficientOutflowFDZero_0 & \coefficientOutflowFDZero_1 & \coefficientOutflowFDZero_2 & \cdots &\coefficientOutflowFDZero_{\numberSpacePoints-1}\\
            \schemeMatrixFDBlockZeroEntry_{-1} & 0 & \schemeMatrixFDBlockZeroEntry_{1} & \\
            & \ddots & \ddots & \ddots & \\
            & & \schemeMatrixFDBlockZeroEntry_{-1} & 0 & \schemeMatrixFDBlockZeroEntry_{1} \\
            0 & \cdots & 0 & 0 & 0
        \end{pmatrix},  \qquad
        \schemeMatrixFDBlockMinusOne = 
        \begin{pmatrix}
            \coefficientOutflowFDMinusOne_0 & \coefficientOutflowFDMinusOne_1 & \coefficientOutflowFDMinusOne_2 & \cdots &\coefficientOutflowFDMinusOne_{\numberSpacePoints-1}\\
            0 & \schemeMatrixFDBlockMinusOneEntry_0 & 0 & \\
            & \ddots & \ddots & \ddots & \\
            & & 0 & \schemeMatrixFDBlockMinusOneEntry_0 & 0 \\
            0 & \cdots & 0 & 0 & 0
        \end{pmatrix},
    \end{equation*}
    so that $\schemeMatrixFD$ could be grandiloquently called ``block banded \emph{quasi}-Toeplitz companion''.
    The explicit determination of $\spectrum(\schemeMatrixFD)$ is out of reach: using the formula for the determinant of a block matrix and a row permutation, we come to 
    \begin{equation}\label{eq:quadraticEigenvalue}
        \determinant(\timeShiftOperator\identityMatrix{2\numberSpacePoints} - \schemeMatrixFD) = \determinant (\timeShiftOperator(\timeShiftOperator \identityMatrix{\numberSpacePoints} - \schemeMatrixFDBlockZero ) - \schemeMatrixFDBlockMinusOne ) = \determinant (\timeShiftOperator^2 \identityMatrix{\numberSpacePoints} - \timeShiftOperator\schemeMatrixFDBlockZero - \schemeMatrixFDBlockMinusOne ).
    \end{equation}
    This emphasizes that we face a \strong{quadratic eigenvalue problem} \cite{tisseur2001quadratic}.
    For $\schemeMatrixFD$ has real entries, its eigenvalues are either real or show up as pairs of conjugate complex numbers.
    \begin{remark}[Zero eigenvalues of $\schemeMatrixFD$]\label{rem:zeroEigFD}
        In this case, the singularity of $\schemeMatrixFD$ comes from the inflow boundary condition but generally not from the outflow.
    \end{remark}

We can fruitfully see $\schemeMatrixFD$ as a \strong{perturbation} of the matrix $\schemeMatrixFDToeplitz \in \matrixSpace{2\numberSpacePoints}{\reals}$ with the same block structure as in \eqref{eq:schemeMatricialFD}, with $\schemeMatrixFDBlockZero$ and $\schemeMatrixFDBlockMinusOne$ replaced by their Toeplitz versions $\schemeMatrixFDToeplitzBlockZero$ and $\schemeMatrixFDToeplitzBlockMinusOne$. This boils down to replace the first and last rows with the patterns inside the matrices, and obtain
    \begin{equation}\label{eq:boundaryEnteringAsPerturbations}
        \schemeMatrixFD = \schemeMatrixFDToeplitz + \canonicalBasisVector{1}\transpose{\parturbationBoundaryOutflow} + \canonicalBasisVector{\numberSpacePoints}\transpose{\parturbationBoundaryInflow},
    \end{equation}
    where $\transpose{\parturbationBoundaryOutflow} \definitionEquality (\coefficientOutflowFDZero_0, \coefficientOutflowFDZero_1 -\schemeMatrixFDBlockZeroEntry_1,  \coefficientOutflowFDZero_2, \cdots,\coefficientOutflowFDZero_{\numberSpacePoints-1}, \coefficientOutflowFDMinusOne_0 - \schemeMatrixFDBlockMinusOneEntry_0, \coefficientOutflowFDMinusOne_1, \cdots,\coefficientOutflowFDMinusOne_{\numberSpacePoints-1})$ and $\transpose{\parturbationBoundaryInflow} \definitionEquality (0, \cdots, 0, -\schemeMatrixFDBlockZeroEntry_{-1}, 0, 0, \cdots, 0, -\schemeMatrixFDBlockMinusOneEntry_0)$.
    Since $\schemeMatrixFD$ is a finite-rank perturbation of the block banded Toeplitz companion matrix $\schemeMatrixFDToeplitz$, which can be conveniently seen as an operator in the limit $\numberSpacePoints\to+\infty$ \cite{bottcher2005spectral}, in this limit, the spectrum of $\schemeMatrixFD$ is the one of $\schemeMatrixFDToeplitz$ plus possibly a bunch of \strong{isolated points}, stemming from boundary conditions, \confer{} \cite[Chapter 2]{lebarbenchon:tel-04214887}.
    \item Finally, we consider the block circulant banded Toeplitz companion version $\schemeMatrixFDCirculant$, with boundary conditions replaced by \strong{periodicity}.
\end{itemize}

\begin{figure} 
    \begin{center}
        \includegraphics[width = 0.48\textwidth]{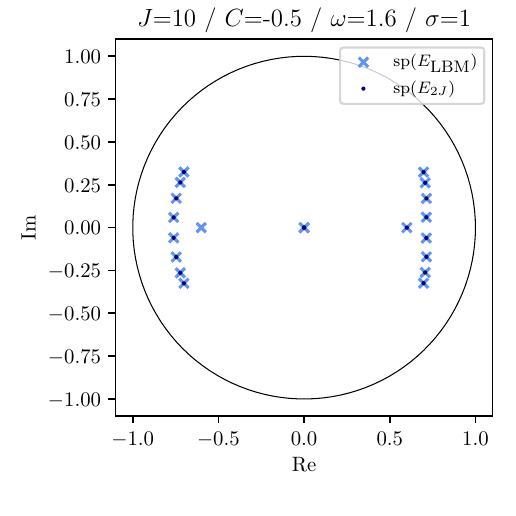}
        \includegraphics[width = 0.48\textwidth]{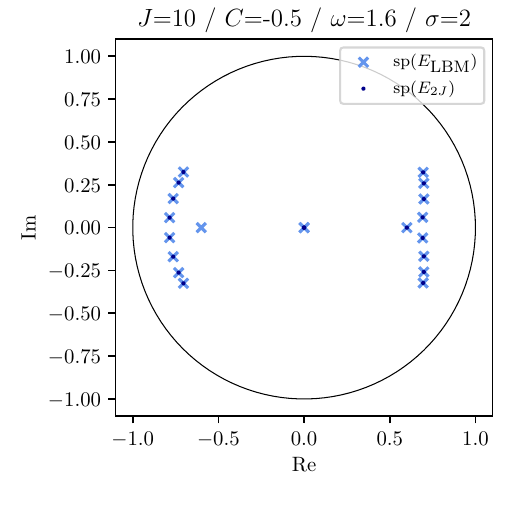} \\
        \includegraphics[width = 0.48\textwidth]{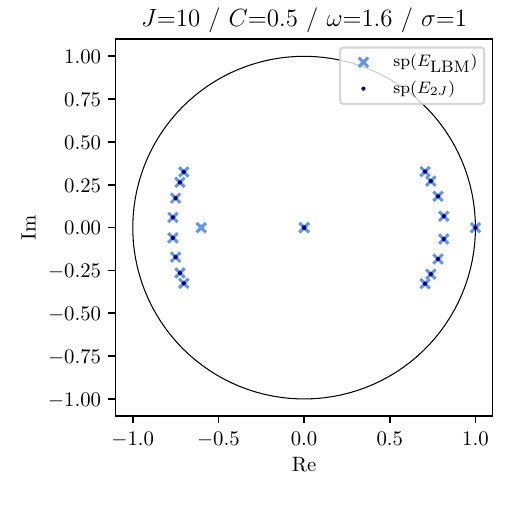}
        \includegraphics[width = 0.48\textwidth]{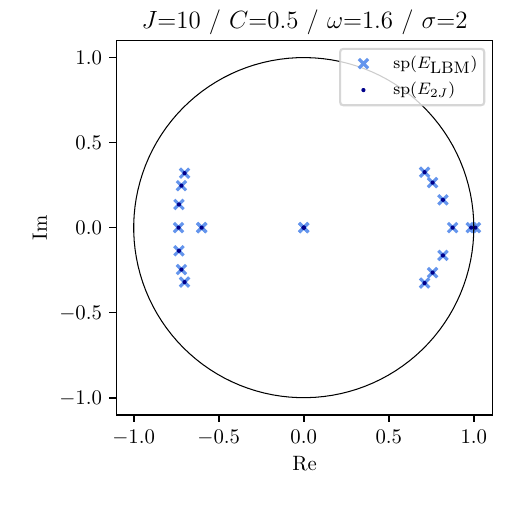}
    \end{center}\caption{\label{fig:spectrum1}Example of spectra for \eqref{eq:extrapolationBoundaryCondition} with $\orderExtrapolation = 1, 2$ that show slight differences between the original \lbm{} scheme and the \fd{} scheme.}
\end{figure}

With these different matrices at our disposal, a first question concerns differences between $\schemeMatrixLBM$ and $\schemeMatrixFD$ as far as their  spectra are concerned.
\strong{Which one do we need to consider?}
Discrepancies, if any, come from non-trivial boundary conditions, because in the periodic framework, eigenvalues are the same \cite{bellotti2022finite}.
In \Cref{fig:spectrum1}, we see that the only expected difference may concern the one about $1-\relaxationParameter$ that might  be in $\spectrum(\schemeMatrixLBM)$ but not in $\spectrum(\schemeMatrixFD)$.
Taking for example the case \eqref{eq:papillon1} with $\courantNumber < 0$, the reason is that, \confer{} \cite[Chapter 12]{bellotti:tel-04266822}, this boundary condition leads GKS-instability on the non-conserved variable $\nonConservedMomentDiscrete$ when $\relaxationParameter = 2$. However, for our \strong{sole interest} is on $\conservedMomentDiscrete$, we want to filter this instability on $\nonConservedMomentDiscrete$ pointless on $\conservedMomentDiscrete$.
Once again, the conserved moment $\conservedMomentDiscrete$ is the same regardless of computing it using the \lbm{} or the \fd{} scheme, and spectral discrepancies can only concern $\nonConservedMomentDiscrete$, which we are not interested in.
Another difference, \confer{} \Cref{rem:zeroEigLBM} and \ref{rem:zeroEigFD}, may show up in terms of multiplicity of the eigenvalue $\timeShiftOperator = 0$, which replaces the one close to $1-\relaxationParameter$ when the latter disappears.
In the sequel, we focus on the spectrum of $\schemeMatrixFD$.

\begin{figure} 
    \begin{center}
        \includegraphics[width = 0.48\textwidth]{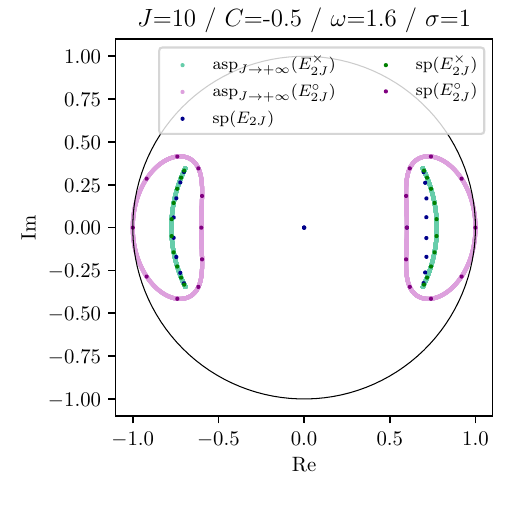}
    \end{center}\caption{\label{fig:spectrum2}Example of spectra for \eqref{eq:papillon1}}
\end{figure}


We now compare the spectrum of $\schemeMatrixFD$ to the ones of $\schemeMatrixFDToeplitz$ and $\schemeMatrixFDCirculant$ (and their asymptotic limits as $\numberSpacePoints \to +\infty$, \confer{} \Cref{sec:asymptoticSpectra}), see \Cref{fig:spectrum2}.
Despite $\numberSpacePoints = 10$, $\spectrum(\schemeMatrixFDToeplitz)$ is remarkably close to the asymptotic one.
This is true to a lesser extent for $\schemeMatrixFD$, because separation between boundary independent spectrum (asymptotically coinciding with the one of $\schemeMatrixFDToeplitz$) and boundary dependent one is still weak. 
This means that there is still a \strong{residual coupling} between boundaries.
As expected, the asymptotic spectrum of $\schemeMatrixFDCirculant$ encloses\footnote{This is not generally true at finite $\numberSpacePoints$.} the asymptotic spectrum of $\schemeMatrixFDToeplitz$ \cite{schmidt1960toeplitz}.
We identify three spectral regimes for $\schemeMatrixFD$, according to $\numberSpacePoints$:
\begin{enumerate}
    \item $\numberSpacePoints \sim 1$, \strong{unclustered} regime, where eigenvalues are still not perfectly separated into two classes (boundary independent and dependent).
    Moreover, there can be a strong interaction between boundaries.
    \item $\numberSpacePoints \sim 10$, \strong{clustered non-asymptotic} regime, where eigenvalues have separated into two classes but are still slightly away from limit structures, explaining discrepancies from GKS theory.
    \item $\numberSpacePoints \to +\infty$, \strong{clustered asymptotic} regime, with a strong correspondence between the decoupling of boundaries and the fact that GKS theory allows---provided that $\spaceStep$ be small---to consider each boundary on its own.
\end{enumerate}

At this stage, a comment on the so-called \strong{P-stability}---which \cite{beam1982stability, trefethen1985stability} see as the right one to be employed at finite $\numberSpacePoints$ on problems featuring two boundaries---is needed.
A P-stable scheme must have two independently GKS-stable boundaries and the eigenvalues $\timeShiftOperator$ of the overall (meaning with two boundary conditions and the bulk scheme) eigenvalue problems at fixed $\numberSpacePoints$ lay inside the closed unit disk.
As in \cite{beam1982stability}, the analysis starts from the general solution \eqref{eq:generalSolutionResolvent}, since we look at both boundaries simultaneously, which reads ${\conservedMomentDiscrete}_{\indexSpace}^{\indexTime} = \timeShiftOperator^{\indexTime} (\coefficientStable (\timeShiftOperator)\solutionCharStable(\timeShiftOperator)^{\indexSpace} + \coefficientUnstable (\timeShiftOperator)\commonTerm(\relaxationParameter, \courantNumber)^{\indexSpace}\solutionCharStable(\timeShiftOperator)^{-\indexSpace})$.
One then inserts it into the inflow homogeneous condition $\conservedMomentDiscrete_{\numberSpacePoints - 1}^{\indexTime + 1} = 0$, obtaining---upon neglecting $\timeShiftOperator = 0$ (of course an eigenvalue)---the relation $\coefficientStable (\timeShiftOperator)\solutionCharStable(\timeShiftOperator)^{\numberSpacePoints - 1} + \coefficientUnstable (\timeShiftOperator)\commonTerm(\relaxationParameter, \courantNumber)^{\numberSpacePoints - 1}\solutionCharStable(\timeShiftOperator)^{-\numberSpacePoints + 1} = 0$.
The general solution thus becomes ${\conservedMomentDiscrete}_{\indexSpace}^{\indexTime} = \timeShiftOperator^{\indexTime} \coefficientStable (\timeShiftOperator) (\solutionCharStable(\timeShiftOperator)^{\indexSpace} - \commonTerm(\relaxationParameter, \courantNumber)^{-\numberSpacePoints + 1 + \indexSpace}\solutionCharStable(\timeShiftOperator)^{2\numberSpacePoints - 2-\indexSpace})$.
Into the left boundary scheme \eqref{eq:outflowSchemeAbstract}, we obtain 
\begin{equation*}
    \timeShiftOperator = \sum_{\indexSpace = 0}^{\numberCoefficientsOutFlowFDZero - 1}\coefficientOutflowFDZero_{\indexSpace}(\solutionCharStable(\timeShiftOperator)^{\indexSpace} - \commonTerm(\relaxationParameter, \courantNumber)^{-\numberSpacePoints + 1 + \indexSpace}\solutionCharStable(\timeShiftOperator)^{2\numberSpacePoints - 2-\indexSpace}) + \timeShiftOperator^{-1}\sum_{\indexSpace = 0}^{\numberCoefficientsOutFlowFDMinusOne - 1}\coefficientOutflowFDMinusOne_{\indexSpace}(\solutionCharStable(\timeShiftOperator)^{\indexSpace} - \commonTerm(\relaxationParameter, \courantNumber)^{-\numberSpacePoints + 1 + \indexSpace}\solutionCharStable(\timeShiftOperator)^{2\numberSpacePoints - 2-\indexSpace}),
\end{equation*}
which is solved in $\solutionCharStable(\timeShiftOperator)$, and then plugged into the bulk scheme \eqref{eq:bulkSchemeAbstract} to yield the values of $\timeShiftOperator$.
We see that all the equations are of very high order $\propto \numberSpacePoints$, and practically difficult to deal with explicitly.
The eigenvalues $\timeShiftOperator$ would eventually be those of $\spectrum(\schemeMatrixFD)$ at finite $\numberSpacePoints$, hence the P-stability analysis boils down to check the modulii of the eigenvalues of $\schemeMatrixFD$.
Iterative procedures \cite{miller1971location} to study whether the roots of \eqref{eq:quadraticEigenvalue} belong to the unit disk, or---after a change of variable \cite{lin2021multiple}---the Routh-Hurwitz criterion, are too involved---similarly to a complete rigorous proof of P-stability/instability---and trivially prescribe $\relaxationParameter \in (0, 2]$.
We also acknowledge that---during the peer-review of this paper---a colleague made us discovering the work of \cite{benoit2022stability}, which deals with the stability of one-step schemes with two boundary conditions. Though we believe that these results can be adapted to our framework of multi-step schemes, and would account for the threshold effects that we observed on $\numberSpacePoints$, it would be very difficult to find a sharp value $\numberSpacePoints_0$ for which the scheme is unstable if $\numberSpacePoints\leq\numberSpacePoints_0$ and stable if $\numberSpacePoints>\numberSpacePoints_0$. Indeed, in the simple examples discussed in \cite{benoit2022stability}, stability for the two-boundaries problem is obtained for $\numberSpacePoints$ large enough.

\subsubsection{Asymptotic spectra for $\numberSpacePoints \to +\infty$}\label{sec:asymptoticSpectra}

Asymptotic spectra as $\numberSpacePoints\to+\infty$ are way easier to characterize that those for finite $\numberSpacePoints$'s.
We start by the scheme matrix in presence of \strong{periodic} boundary conditions $\schemeMatrixFDCirculant$.
The proof of the following result is classical within the context of circulant banded Toeplitz matrices, and given in the Supplementary material.

\begin{lemma}[Spectrum and asymptotic spectrum of $\schemeMatrixFDCirculant$]\label{lemma:spectrumCirculant}
    The spectrum of $\schemeMatrixFDCirculant$ is given by
    \begin{equation}\label{eq:spectrumCirculant}
        \spectrum(\schemeMatrixFDCirculant)  = \Bigl \{\tfrac{1}{2} \Bigl ( \bigl (\schemeMatrixFDBlockZeroEntry_{-1}e^{\tfrac{2\pi i k}{\numberSpacePoints}} + \schemeMatrixFDBlockZeroEntry_{1}e^{\tfrac{2\pi i (\numberSpacePoints - 1) k}{\numberSpacePoints}}  \bigr ) 
        \pm \sqrt{ \Bigl( \schemeMatrixFDBlockZeroEntry_{-1}e^{\tfrac{2\pi i k}{\numberSpacePoints}} + \schemeMatrixFDBlockZeroEntry_{1}e^{\tfrac{2\pi i (\numberSpacePoints - 1) k}{\numberSpacePoints}} \Bigr)^2 + 4\schemeMatrixFDBlockMinusOneEntry_0}\Bigr ) \quad \text{for} \quad k \in \integerInterval{0}{\numberSpacePoints - 1} \Bigr \}.
    \end{equation}
    Its asymptotic spectrum as $\numberSpacePoints \to +\infty$ reads
    \begin{equation}\label{eq:asymptoticSpectrumCirculant}
        \asymptoticSpectrum(\schemeMatrixFDCirculant) \definitionEquality \Bigl \{ \tfrac{1}{2} \Bigl ( \bigl (\schemeMatrixFDBlockZeroEntry_{-1}e^{-i \vartheta} + \schemeMatrixFDBlockZeroEntry_1e^{i \vartheta}  \bigr ) 
        \pm \sqrt{ \Bigl( \schemeMatrixFDBlockZeroEntry_{-1}e^{-i \vartheta} + \schemeMatrixFDBlockZeroEntry_{1}e^{i \vartheta} \Bigr)^2 + 4 \schemeMatrixFDBlockMinusOneEntry_0}\Bigr ) \quad \text{for} \quad  \vartheta \in [-\pi, \pi] \Bigr \}.
    \end{equation}
\end{lemma}
It is interesting to consider the case $\relaxationParameter = 2$, where the round bean-like shapes described by \eqref{eq:asymptoticSpectrumCirculant} and visible in \Cref{fig:spectrum2} degenerate into bent segments.
We obtain 
    \begin{equation}\label{eq:tmp4}
        \asymptoticSpectrum(\schemeMatrixFDCirculant)|_{\relaxationParameter = 2} = \Bigl \{ -i \courantNumber \sin( \vartheta)  \pm \sqrt{ -\courantNumber^2\sin^2( \vartheta) + 1} \quad \text{for} \quad  \vartheta \in [-\pi, \pi] \Bigr \}.
    \end{equation}

\begin{figure}
    \begin{center}
        \begin{tikzpicture}[scale=0.82]
            \draw[->] (0,0) -- (14,0) node[anchor=west] {$\relaxationParameter$}; 
    
            \draw[<-] (1,-4.75) -- (1,0.1);
            \node[above] at (1,0.1) {$0$}; 
    
            \draw[<-] (5,-4.75) -- (5,0.1);
            \node[above] at (5,0.1) {$1$}; 
    
            \draw[<-] (9,-4.75) -- (9,0.1);
            \node[above] at (9,0.1) {$\tfrac{2}{1+|\courantNumber|}$}; 
    
            \draw[<-] (13,-4.75) -- (13,0.1);
            \node[above] at (13,0.1) {$2$}; 

            \draw[ProcessBlue] (3,-2) circle (1.5); 
            \draw[WildStrawberry, very thick] (1.7,-2) -- (2.8,-2); 
            \draw[WildStrawberry, very thick] (3.2,-2) -- (4.3,-2); 
            \draw[WildStrawberry, very thick] (3,-2) circle (0.7); 
            \draw[|->|] (3,-2) -- (3,-1.3) node[right, midway] {\footnotesize$\sqrt{1-\relaxationParameter}$} ; 
            \draw[|<->|] (1.7,-2.2) -- (2.8,-2.2); 
            \draw[<-, densely dotted] (2.25,-2.35) .. controls (2.2,-3.3)  .. (2.4,-3.7) node[below] {\footnotesize$2\sqrt{\relaxationParameter - 1 + \schemeMatrixFDBlockZeroEntry_{-1}\schemeMatrixFDBlockZeroEntry_{1}}$}; 

            \draw[ProcessBlue] (7,-2) circle (1.5); 
            \draw[WildStrawberry, very thick] (5.7,-2) -- (6.8,-2); 
            \draw[WildStrawberry, very thick] (7.2,-2) -- (8.3,-2); 
            \draw[|<->|] (7.2,-2.2) -- (8.3,-2.2) node[below, midway] {\footnotesize$2\sqrt{\schemeMatrixFDBlockZeroEntry_{-1}\schemeMatrixFDBlockZeroEntry_{1}}$}; 

            \draw[ProcessBlue] (11,-2) circle (1.5); 
            \draw [WildStrawberry, very thick,domain=-45:45] plot ({11+cos(\x)}, {-2+sin(\x)}); 
            \draw [WildStrawberry, very thick,domain=135:225] plot ({11+cos(\x)}, {-2+sin(\x)}); 
            \draw[|->|] (11,-2) -- ({11+cos(45)}, {-2+sin(45)}) node[above, midway, sloped] {\footnotesize$\sqrt{\relaxationParameter - 1}$} ; 
            \draw[<->, domain=0:-45] plot ({11+1.2*cos(\x)}, {-2+1.2*sin(\x)});
            \draw[<-, densely dotted] ({11+1.2*cos(-45/2) + 0.1}, {-2+1.2*sin(-45/2)-0.1}) .. controls (13.2,-3.3)  .. (11.4,-3.7) node[below] {\footnotesize$\text{arcsin}\Bigl (\sqrt{\tfrac{\schemeMatrixFDBlockZeroEntry_{-1}\schemeMatrixFDBlockZeroEntry_{1}}{1-\relaxationParameter}} \Bigr )$}; 

            \draw[WildStrawberry, very thick] (1,-6.75) circle (1.5); 
    
            \draw[ProcessBlue] (5,-6.75) circle (1.5); 
            \draw[WildStrawberry, very thick] (4.5,-6.75) -- (5.5,-6.75); 
            \draw[|<->|] (4.5,-6.95) -- (5.5,-6.95) node[below, midway] {\footnotesize$2\sqrt{1-\courantNumber^2}$}; 
    
            \draw[ProcessBlue] (9,-6.75) circle (1.5); 
            \fill[WildStrawberry] (8.5,0-6.75) circle (1.2pt); 
            \fill[WildStrawberry] (9.5,0-6.75) circle (1.2pt); 
            \draw[|<->|] (8.5,-6.95) -- (9.5,-6.95) node[below, midway] {\footnotesize$2\sqrt{1-\relaxationParameter}$}; 

            \draw[ProcessBlue] (13,-6.75) circle (1.5); 
            \draw [WildStrawberry, very thick,domain=-45:45] plot ({13+1.5*cos(\x)}, {-6.75+1.5*sin(\x)}); 
            \draw [WildStrawberry, very thick,domain=135:225] plot ({13+1.5*cos(\x)}, {-6.75+1.5*sin(\x)}); 
            \draw[<->, domain=0:-45] plot ({13+1.3*cos(\x)}, {-6.75+1.3*sin(\x)}); 
            \node[left] at ({13+1.2*cos(-45/2)}, {-6.75+1.2*sin(-45/2)}) {\footnotesize$\text{arcsin}(|\courantNumber|)$} ; 
    
        \end{tikzpicture}
    \end{center}\caption{\label{fig:asymptSpectrumNoBoundary}Sketch of the different shapes of the asymptotic spectrum of $\schemeMatrixFDToeplitz$ as $\numberSpacePoints\to+\infty$ (red color) for a given $\courantNumber$, varying $\relaxationParameter \in [0, 2]$. The light blue line is the unit circle.}
\end{figure}

As far as $\schemeMatrixFDToeplitz$ is concerned, we are only able to characterize its asymptotic spectrum---which is also the part of that of $\schemeMatrixFD$  \strong{independent} of the boundary conditions---with the following result.
\begin{lemma}[Asymptotic spectrum of $\schemeMatrixFDToeplitz$]\label{lemma:limitSpectrumToeplitz}
    Under the stability conditions given by \Cref{prop:stabilityConditionsPeriodic}, the asymptotic spectrum of $\schemeMatrixFDToeplitz$ is as follows.
    \begin{itemize}
        \item If $\schemeMatrixFDBlockZeroEntry_{\mp 1} = 0$, thus for $\relaxationParameter = \tfrac{2}{1\mp \courantNumber} \in [1, 2]$ (and $\mp\courantNumber\in[0, 1]$), then it is made up of isolated points on the real axis:
        \begin{equation}\label{eq:tmp10}
            \asymptoticSpectrum(\schemeMatrixFDToeplitz) \definitionEquality \Bigl \{ \sqrt{\tfrac{1+\courantNumber}{1-\courantNumber}}  = \sqrt{\relaxationParameter - 1}, -\sqrt{\tfrac{1+\courantNumber}{1-\courantNumber}}  = -\sqrt{\relaxationParameter - 1} \Bigr \}.
        \end{equation}
        \item Otherwise, with $\sqrt{\cdot}$ the principal square root:
        \begin{multline}\label{eq:limitSpectrumNoBoundary}
            \asymptoticSpectrum(\schemeMatrixFDToeplitz) \definitionEquality \Bigl \{  \frac{1}{2}\Bigl ( \Bigl ( \tfrac{\schemeMatrixFDBlockZeroEntry_{-1}}{\sqrt{\commonTerm(\relaxationParameter, \courantNumber)}}e^{-i \vartheta} + \schemeMatrixFDBlockZeroEntry_{1} \sqrt{\commonTerm(\relaxationParameter, \courantNumber)}e^{i \vartheta} \Bigr ) \\
            \pm  \sqrt{\Bigl ( \tfrac{\schemeMatrixFDBlockZeroEntry_{-1}}{\sqrt{\commonTerm(\relaxationParameter, \courantNumber)}}e^{-i \vartheta} + \schemeMatrixFDBlockZeroEntry_{1} \sqrt{\commonTerm(\relaxationParameter, \courantNumber)}e^{i \vartheta} \Bigr )^2 + 4\schemeMatrixFDBlockMinusOneEntry_0}   \Bigr ) \quad \text{for} \quad  \vartheta \in [0, \pi] \Bigr \}.
        \end{multline}
    \end{itemize}
\end{lemma}
The asymptotic spectrum \eqref{eq:limitSpectrumNoBoundary} can be characterized in a more explicit fashion---see \Cref{fig:asymptSpectrumNoBoundary}---as follows.
    \begin{lemma}[Asymptotic spectrum of $\schemeMatrixFDToeplitz$]\label{lemma:limitSpectrumToeplitzMoreExplicit}
        Under the stability conditions given by \Cref{prop:stabilityConditionsPeriodic}, considering $\courantNumber$ as given, the asymptotic spectrum of $\schemeMatrixFDToeplitz$ is symmetric with respect both to the real and complex axes, and as follows.
        \begin{itemize}
            \item $\relaxationParameter = 0$. The asymptotic spectrum is the unit circle: $\asymptoticSpectrum(\schemeMatrixFDToeplitz)\definitionEquality \{ \timeShiftOperator \in \complex~:~|\timeShiftOperator| = 1\}$.
            \item $\relaxationParameter \in (0, 1)$. The asymptotic spectrum is the circle of radius $\sqrt{1-\relaxationParameter}$ and two segments on the real axis contained inside the unit closed disk:
            \begin{align*}
                \asymptoticSpectrum(\schemeMatrixFDToeplitz)\definitionEquality \{ \timeShiftOperator \in \complex~:~|\timeShiftOperator| = \sqrt{1-\relaxationParameter}\} &\cup [\sqrt{\schemeMatrixFDBlockZeroEntry_{-1}\schemeMatrixFDBlockZeroEntry_1} - \sqrt{\relaxationParameter-1 + \schemeMatrixFDBlockZeroEntry_{-1}\schemeMatrixFDBlockZeroEntry_1}, \sqrt{\schemeMatrixFDBlockZeroEntry_{-1}\schemeMatrixFDBlockZeroEntry_1} + \sqrt{\relaxationParameter-1 + \schemeMatrixFDBlockZeroEntry_{-1}\schemeMatrixFDBlockZeroEntry_1}] \\
                &\cup [-\sqrt{\schemeMatrixFDBlockZeroEntry_{-1}\schemeMatrixFDBlockZeroEntry_1} - \sqrt{\relaxationParameter-1 + \schemeMatrixFDBlockZeroEntry_{-1}\schemeMatrixFDBlockZeroEntry_1}, -\sqrt{\schemeMatrixFDBlockZeroEntry_{-1}\schemeMatrixFDBlockZeroEntry_1} + \sqrt{\relaxationParameter-1 + \schemeMatrixFDBlockZeroEntry_{-1}\schemeMatrixFDBlockZeroEntry_1}].
            \end{align*}
            \item $\relaxationParameter = 1$. The asymptotic spectrum is a segment on the real axis contained inside the unit closed disk:
            \begin{equation*}
                \asymptoticSpectrum(\schemeMatrixFDToeplitz)\definitionEquality [-\sqrt{1-\courantNumber^2}, \sqrt{1-\courantNumber^2}].
            \end{equation*}
            \item $\relaxationParameter \in (1, \tfrac{2}{1+|\courantNumber|})$. The asymptotic spectrum is made up of two segments on the real axis contained inside the unit closed disk:
            \begin{align*}
                \asymptoticSpectrum(\schemeMatrixFDToeplitz)\definitionEquality  &[-\sqrt{\schemeMatrixFDBlockZeroEntry_{-1}\schemeMatrixFDBlockZeroEntry_1} + \sqrt{\relaxationParameter-1 + \schemeMatrixFDBlockZeroEntry_{-1}\schemeMatrixFDBlockZeroEntry_1}, \sqrt{\schemeMatrixFDBlockZeroEntry_{-1}\schemeMatrixFDBlockZeroEntry_1} + \sqrt{\relaxationParameter-1 + \schemeMatrixFDBlockZeroEntry_{-1}\schemeMatrixFDBlockZeroEntry_1}] \\
                \cup &[-\sqrt{\schemeMatrixFDBlockZeroEntry_{-1}\schemeMatrixFDBlockZeroEntry_1} - \sqrt{\relaxationParameter-1 + \schemeMatrixFDBlockZeroEntry_{-1}\schemeMatrixFDBlockZeroEntry_1}, \sqrt{\schemeMatrixFDBlockZeroEntry_{-1}\schemeMatrixFDBlockZeroEntry_1} - \sqrt{\relaxationParameter-1 + \schemeMatrixFDBlockZeroEntry_{-1}\schemeMatrixFDBlockZeroEntry_1}].
            \end{align*}
            \item $\relaxationParameter = \tfrac{2}{1+|\courantNumber|}$. The asymptotic spectrum is made up of two real points: $\asymptoticSpectrum(\schemeMatrixFDToeplitz)\definitionEquality \{\sqrt{\relaxationParameter - 1}, -\sqrt{\relaxationParameter - 1}\}$.
            \item $\relaxationParameter\in (\tfrac{2}{1+|\courantNumber|}, 2)$. The asymptotic spectrum is made up of two circular arcs of radius $\sqrt{\relaxationParameter - 1}$:
            \begin{equation*}
                \asymptoticSpectrum(\schemeMatrixFDToeplitz)\definitionEquality \Bigl \{  \sqrt{\relaxationParameter - 1} \times e^{i\vartheta}~:~|\vartheta| \leq \textnormal{arcsin}\Bigl ( \sqrt{\tfrac{\schemeMatrixFDBlockZeroEntry_{-1}\schemeMatrixFDBlockZeroEntry_1}{1-\relaxationParameter} }\Bigr ) \Bigr \} \cup \Bigl \{  \sqrt{\relaxationParameter - 1}\times e^{i\vartheta + \pi}~:~|\vartheta| \leq \textnormal{arcsin}\Bigl ( \sqrt{\tfrac{\schemeMatrixFDBlockZeroEntry_{-1}\schemeMatrixFDBlockZeroEntry_1}{1-\relaxationParameter} }\Bigr ) \Bigr \}.
            \end{equation*}
            \item $\relaxationParameter = 2$. The asymptotic spectrum is made up of two circular arcs of radius one:
            \begin{equation*}
                \asymptoticSpectrum(\schemeMatrixFDToeplitz)\definitionEquality \{  e^{i\vartheta}~:~|\vartheta| \leq \textnormal{arcsin}(|\courantNumber|) \} \cup  \{  e^{i\vartheta + \pi}~:~|\vartheta| \leq \textnormal{arcsin}(|\courantNumber|)  \}.
            \end{equation*}
        \end{itemize}
    \end{lemma}

Notice that the result for $\relaxationParameter = 2$ in \Cref{lemma:limitSpectrumToeplitzMoreExplicit} is simply another parametrization of the limit profile in \eqref{eq:tmp4}.
This entails that---from the spectral standpoint, in the limit $\numberSpacePoints\to+\infty$, and when $\relaxationParameter = 2$---the periodic case $\schemeMatrixFDCirculant$ behaves like the Toeplitz case $\schemeMatrixFDToeplitz$.

\begin{proof}[Proof of \Cref{lemma:limitSpectrumToeplitz}]
    We follow \cite{beam1993asymptotic}.
    This part of the spectrum is associated with distinct $\fourierShift_{a}, \fourierShift_{b} \in \complex$ of same modulus $|\fourierShift_{a}| = |\fourierShift_{b}|$, both fulfilling \eqref{eq:bulkCharEquation}.
    Therefore, there exist a phase shift $ \vartheta \in [0, \pi]$ and $\tilde{\fourierShift} \in \complex$ such that $|\tilde{\fourierShift}| = |\fourierShift_{a}| = |\fourierShift_{b}|$, so that $\fourierShift_{a} = \tilde{\fourierShift}e^{i \vartheta}$ and $\fourierShift_{b} = \tilde{\fourierShift}e^{-i \vartheta}$.
    Inserting $\fourierShift_{a}$ and $\fourierShift_{b}$ into \eqref{eq:bulkCharEquation} provides
    \begin{equation*}
        \begin{cases}
            \modifiedTimeShiftOperator(\timeShiftOperator) = \schemeMatrixFDBlockZeroEntry_{-1}\tilde{\fourierShift}^{-1}e^{-i \vartheta} +\schemeMatrixFDBlockZeroEntry_{1} \tilde{\fourierShift}e^{i \vartheta},  \\
            \modifiedTimeShiftOperator(\timeShiftOperator) = \schemeMatrixFDBlockZeroEntry_{-1}\tilde{\fourierShift}^{-1}e^{i \vartheta} +\schemeMatrixFDBlockZeroEntry_{1} \tilde{\fourierShift}e^{-i \vartheta},
        \end{cases}
    \end{equation*}
    where we stress that the same $\timeShiftOperator$ corresponds to distinct  $\fourierShift_{a}$ and $\fourierShift_{b}$.
    Solving and simplifying yields $\schemeMatrixFDBlockZeroEntry_1 \tilde{\fourierShift}^2 = \schemeMatrixFDBlockZeroEntry_{-1}$: 
    \begin{equation*}
        \tilde{\fourierShift} = 
        \begin{cases}
            \pm \sqrt{\commonTerm(\relaxationParameter, \courantNumber)}, \qquad &\text{if}\quad \schemeMatrixFDBlockZeroEntry_1\neq 0, \\
            \textnormal{undefined}\qquad &\text{if} \quad \schemeMatrixFDBlockZeroEntry_1 = 0,
        \end{cases}
    \end{equation*}
    where in the first case, we have $\tilde{\fourierShift} \in \reals$ whenever $0<\relaxationParameter<\tfrac{1}{1+|\courantNumber|}$ and $\tilde{\fourierShift} \in i \reals$ for $\frac{2}{1+|\courantNumber|} < \relaxationParameter < 2$, see \Cref{lemma:propertiesCommonTerm}.
    When $\schemeMatrixFDBlockZeroEntry_{-1} = 0$, hence $\tilde{\fourierShift} = 0$, the asymptotic spectrum is given by $\timeShiftOperator^2  + (1-\relaxationParameter) = \timeShiftOperator^2  - \frac{1+\courantNumber}{1-\courantNumber}  = 0 $, which gives \eqref{eq:tmp10}.
    When $\schemeMatrixFDBlockZeroEntry_{1} = 0$, hence $\tilde{\fourierShift}$ is not defined, we also obtain \eqref{eq:tmp10}.
    Otherwise, we consider---without lack of generality---the principal square root with a plus sign in front, hence $\tilde{\fourierShift} = \sqrt{\commonTerm(\relaxationParameter, \courantNumber)}$. Solving a quadratic equation, the spectrum satisfies \eqref{eq:limitSpectrumNoBoundary}.
\end{proof}

We now study the isolated points due to the \strong{outflow} boundary condition, thus excluding the eigenvalue $\timeShiftOperator = 0$ from the inflow and the numerous ones unrelated to boundary conditions. The next claim is coherent with \Cref{fig:spectrum1}.
\begin{lemma}[Asymptotic spectrum of $\schemeMatrixFD$ linked with the outflow boundary condition \eqref{eq:extrapolationBoundaryCondition}]\label{lemma:limitSpectrumOutflow}
    Under the stability conditions given by \Cref{prop:stabilityConditionsPeriodic}.
    We have the following.
    \begin{itemize}
        \item For $\orderExtrapolation = 1$, the asymptotic spectrum of $\schemeMatrixFD$ created by \eqref{eq:papillon1} is given by
        \begin{equation*}
            \asymptoticSpectrum(\schemeMatrixFD) \smallsetminus \Bigl (\asymptoticSpectrum(\schemeMatrixFDToeplitz) \cup \{0\} \Bigr ) \definitionEquality 
            \begin{cases}
                \varnothing, \qquad &\text{if} \quad \relaxationParameter = 2, \\
                \{\relaxationParameter-1\}, \qquad &\text{if} \quad \relaxationParameter \in (0, 2) ~\text{and}~\courantNumber<0,  \\
                \{1\}, \qquad &\text{if} \quad \relaxationParameter \in (0, 2) ~\text{and}~\courantNumber>0.
            \end{cases}
        \end{equation*}
        \item For $\orderExtrapolation = 2$, the asymptotic spectrum of $\schemeMatrixFD$ created by \eqref{eq:papillon2} is given by
        \begin{equation*}
            \asymptoticSpectrum(\schemeMatrixFD) \smallsetminus \Bigl (\asymptoticSpectrum(\schemeMatrixFDToeplitz) \cup \{0\} \Bigr ) \definitionEquality 
            \begin{cases}
                \varnothing, \qquad &\text{if} \quad \relaxationParameter = 2, \\
                \{\relaxationParameter-1\}, \qquad &\text{if} \quad \relaxationParameter \in (0, 2) ~\text{and}~\courantNumber<0,  \\
                \{1, 1-\relaxationParameter\}, \qquad &\text{if} \quad \relaxationParameter \in (0, 2) ~\text{and}~\courantNumber>0.
            \end{cases}
        \end{equation*}
        \item For $\orderExtrapolation \geq 3$, the asymptotic spectrum of $\schemeMatrixFD$ created by \eqref{eq:extrapolationBoundaryCondition} includes
        \begin{equation*}
            \asymptoticSpectrum(\schemeMatrixFD) \smallsetminus \Bigl (\asymptoticSpectrum(\schemeMatrixFDToeplitz) \cup \{0\} \Bigr ) \supset 
            \begin{cases}
                \varnothing, \qquad &\text{if} \quad \relaxationParameter = 2, \\
                \{\relaxationParameter-1\}, \qquad &\text{if} \quad \relaxationParameter \in (0, 2) ~\text{and}~\courantNumber<0,  \\
                \{1, 1-\relaxationParameter\}, \qquad &\text{if} \quad \relaxationParameter \in (0, 2) ~\text{and}~\courantNumber>0.
            \end{cases}
        \end{equation*}
    \end{itemize}
\end{lemma}
\begin{proof}[Proof of \Cref{lemma:limitSpectrumOutflow}]
    We follow the procedure by \cite[Algorithm 4.2]{beam1993asymptotic}, which---as observed by these authors---parallels the GKS analysis that we have already performed.

        Consider \eqref{eq:papillon1}.
        From \Cref{lemma:boundaryRoots}, the candidate eigenvalues are $\timeShiftOperator = 1, \relaxationParameter - 1$.
        By virtue of \Cref{lemma:bulkStudy}:
        \begin{align*}
            \timeShiftOperator = 1, \qquad &\fourierShift = 1 \quad \text{(boundary and bulk)}, \qquad \fourierShift = \commonTerm(\relaxationParameter, \courantNumber) \quad \text{(bulk)}, \\
            \timeShiftOperator = \relaxationParameter - 1, \qquad &\fourierShift = -\commonTerm(\relaxationParameter, \courantNumber) \quad \text{(boundary and bulk)}, \qquad \fourierShift = -1 \quad \text{(bulk)}.
        \end{align*}
        To find the boundary dependent asymptotic spectrum, we have to compare---for the same eigenvalue $\timeShiftOperator$---the boundary and the exclusively bulk $\fourierShift$'s, so that the former be strictly smaller than the latter in modulus.
        \begin{itemize}
            \item $\timeShiftOperator = 1$, we want $|1| = 1 < |-\commonTerm(\relaxationParameter, \courantNumber)| = |\commonTerm(\relaxationParameter, \courantNumber)|$.
            By virtue of \Cref{lemma:propertiesCommonTerm}, when $\courantNumber < 0$, the inequality cannot be fulfilled, whereas if $\courantNumber > 0$, then  the inequality is met for $\relaxationParameter \in (0, 2)$.
            \item $\timeShiftOperator = \relaxationParameter - 1$, we want $|-\commonTerm(\relaxationParameter, \courantNumber)| = |\commonTerm(\relaxationParameter, \courantNumber)| < |-1| = 1$.
            Again by virtue of \Cref{lemma:propertiesCommonTerm}, when $\courantNumber < 0$, the inequality is met for $\relaxationParameter \in (0, 2)$, whereas if $\courantNumber > 0$, then  the inequality cannot be fulfilled.
        \end{itemize}
        
        Consider \eqref{eq:extrapolationBoundaryCondition} for $\orderExtrapolation \geq 2$.
        \Cref{lemma:boundaryRoots} gives the candidates $\timeShiftOperator = 1,\pm(1-\relaxationParameter)$.
        For $\orderExtrapolation\geq 3$, we cannot guarantee that these are the only ones, yet it is likely to be the case.
        We are just left with the last one.
        \begin{equation*}
            \timeShiftOperator = 1-\relaxationParameter, \qquad \fourierShift = 1 \quad \text{(boundary and bulk)}, \qquad \fourierShift = \commonTerm(\relaxationParameter, \courantNumber) \quad \text{(bulk)}.
        \end{equation*}
        Proceeding as in the previous case yields the claim.
\end{proof}

\subsubsection{A link between the reflection coefficient and the transition to $\numberSpacePoints\to+\infty$}\label{sec:reflCoeff}

Let us finish with a result connecting the \strong{reflection coefficient}---\confer{} \Cref{sec:reflectionCoefficient}---with the number of eigenvalues of $\schemeMatrixFD$ relative to the outflow boundary condition tending to a single point in the asymptotic spectrum.

\begin{theorem}\label{thm:countingThroughReflection}
    Let $\targetEigenvalue \in\complex\smallsetminus \{ 0\}$ be an isolated point of the asymptotic spectrum of $\schemeMatrixFD$ prescribed by \Cref{lemma:limitSpectrumOutflow}.
    Assume that the stability conditions given by \Cref{prop:stabilityConditionsPeriodic} are satisfied.
    Take $\epsilon > 0$ small.
    Then, for $\numberSpacePoints$ large enough, the number of eigenvalues of $\schemeMatrixFD$, inside $\ball{\epsilon}{\targetEigenvalue}$,  is given as follows.
    \begin{itemize}
        \item When $\schemeMatrixFDBlockZeroEntry_{-1} = 0$, thus for $\relaxationParameter = \tfrac{2}{1-\courantNumber} \in [1, 2]$ (and $\courantNumber\in[-1, 0]$), the number of eigenvalues of $\schemeMatrixFD$, counted with their multiplicity, inside $\ball{\epsilon}{\targetEigenvalue}$, is the order of the zero of the function 
        \begin{equation*}
            \timeShiftOperator \mapsto \frac{\timeShiftOperator^2 - \coefficientOutflowFDZero_0\timeShiftOperator - \coefficientOutflowFDMinusOne_0}{(\timeShiftOperator - i\sqrt{\relaxationParameter-1})(\timeShiftOperator + i\sqrt{\relaxationParameter-1})} \qquad \text{at}\quad \timeShiftOperator = \targetEigenvalue.
        \end{equation*}
        \item Otherwise, the number of eigenvalues of $\schemeMatrixFD$, counted with their multiplicity, inside $\ball{\epsilon}{\targetEigenvalue}$ is the order of the zero of the function 
        \begin{equation*}
            \timeShiftOperator \mapsto \timeShiftOperator -\sum_{\indexRow = 0}^{\numberCoefficientsOutFlowFDZero - 1} \coefficientOutflowFDZero_{\indexRow}\fourierShift(\timeShiftOperator)^{\indexRow} - \timeShiftOperator^{-1} \sum_{\indexRow = 0}^{\numberCoefficientsOutFlowFDMinusOne- 1} \coefficientOutflowFDMinusOne_{\indexRow} \fourierShift(\timeShiftOperator)^{\indexRow}, \qquad \text{at} \quad \timeShiftOperator = \targetEigenvalue,
        \end{equation*}
        where $\fourierShift(\timeShiftOperator)$ is a solution of the characteristic equation \eqref{eq:bulkCharEquation} being 
        \begin{itemize}
            \item if $\schemeMatrixFDBlockZeroEntry_{1} = 0$, thus for $\relaxationParameter = \tfrac{2}{1+\courantNumber} \in [1, 2]$ (and $\courantNumber\in[0, 1]$), the only root $\fourierShift(\timeShiftOperator) = \schemeMatrixFDBlockZeroEntry_{-1}\modifiedTimeShiftOperator(\timeShiftOperator)^{-1} =  (2-\relaxationParameter)\modifiedTimeShiftOperator(\timeShiftOperator)^{-1}$;
            \item otherwise, the one such that it (certainly) exists $\tilde{\timeShiftOperator} \in \partial\ball{\epsilon}{\targetEigenvalue}$ such that
            \begin{equation}\label{eq:zeroToSeek}
                \lim_{\numberSpacePoints\to+\infty}\frac{1}{\sqrt{\schemeMatrixFDBlockZeroEntry_{-1}\schemeMatrixFDBlockZeroEntry_1}}\frac{\chebyshevPolynomialSecondKind{\numberSpacePoints-1}(\frac{\modifiedTimeShiftOperator(\tilde{\timeShiftOperator})}{2\sqrt{\schemeMatrixFDBlockZeroEntry_{-1}\schemeMatrixFDBlockZeroEntry_1}})}{\chebyshevPolynomialSecondKind{\numberSpacePoints}(\frac{\modifiedTimeShiftOperator(\tilde{\timeShiftOperator})}{2\sqrt{\schemeMatrixFDBlockZeroEntry_{-1}\schemeMatrixFDBlockZeroEntry_1}})} = \frac{\fourierShift(\tilde{\timeShiftOperator})}{\schemeMatrixFDBlockZeroEntry_{-1}},
            \end{equation}
            with $\chebyshevPolynomialSecondKind{\numberSpacePoints}$ the Chebyshev polynomial of second kind of degree $\numberSpacePoints$.
        \end{itemize}
    \end{itemize}
\end{theorem}
The claim of \Cref{thm:countingThroughReflection} may look overly technical and maybe off-putting.
However, what it heralds is rather straightforward: if close to $\targetEigenvalue$, the root $\fourierShift(\timeShiftOperator)$ in \eqref{eq:zeroToSeek} is indeed $\solutionCharStable(\timeShiftOperator)$, and the numerator in the reflection coefficient \eqref{eq:reflectionCoefficient} does not vanish, then the order of the pole of $\reflectionCoefficientLetterOut(\timeShiftOperator)$ at $\targetEigenvalue$ readily provides the number of eigenvalues of $\schemeMatrixFD$ tending towards $\targetEigenvalue$ as $\numberSpacePoints\to+\infty$.
This defines \strong{(a sort of) algebraic multiplicity} of the asymptotic eigenvalue $\targetEigenvalue$, and easily predicts the \strong{growth rate} of the solution when $|\targetEigenvalue| = 1$.
    The proof of \Cref{thm:countingThroughReflection} (given in \Cref{proof:thm:countingThroughReflection}) being long and technical, we provide the main steps and ideas beforehand.
    \begin{enumerate}
        \item We use the \strong{Cauchy's argument principle} on the characteristic polynomial of $\schemeMatrixFD$ to count the number of eigenvalues of this matrix within the ball.
        This commands the study of the trace of the resolvent of $\schemeMatrixFD$.
        \item The \strong{clustering and separation of the eigenvalues} of $\schemeMatrixFD$ when $\numberSpacePoints$ increases allows to cancel several contributions, relative to boundary independent eigenvalues and eigenvalues relative to the inflow.
        \item Since boundary conditions enter into the resolvent of $\schemeMatrixFD$ as rank-one perturbations of $\schemeMatrixFDToeplitz$, \confer{} \eqref{eq:boundaryEnteringAsPerturbations}, we apply the \strong{Sherman-Morrison formula} to separate different contributions.
        \item Using the formula for the \strong{inverse of a block matrix} made up of four blocks, the resolvent of $\schemeMatrixFDToeplitz$ can be expressed in terms of the inverse of a tridiagonal Toeplitz matrix.
        \item The entries of the inverse of a tridiagonal Toeplitz matrix can be expressed in terms of \strong{Chebyshev polynomials of second kind}, of which we can compute limits as $\numberSpacePoints$ increases.
    \end{enumerate}

\subsubsection{Plots of spectra and pseudo-spectra}\label{sec:plotsPseudoSpectra}

\begin{figure} 
    \begin{center}
        \includegraphics[width = 0.45\textwidth]{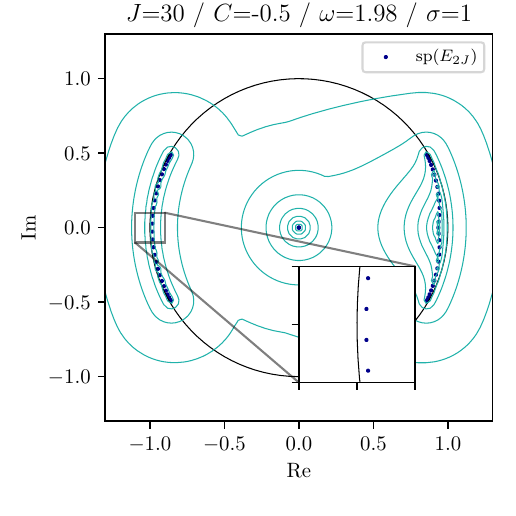}
        \includegraphics[width = 0.45\textwidth]{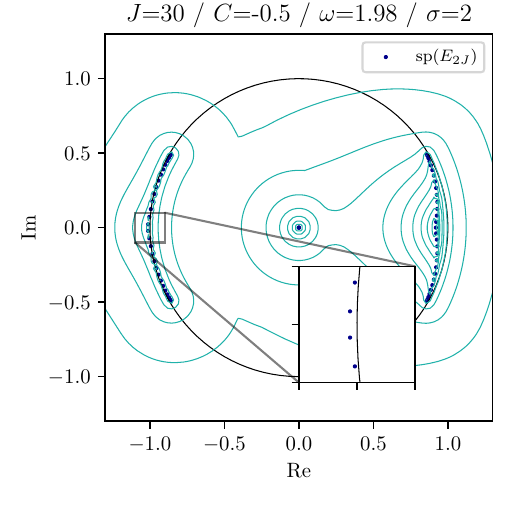} \\
        \includegraphics[width = 0.45\textwidth]{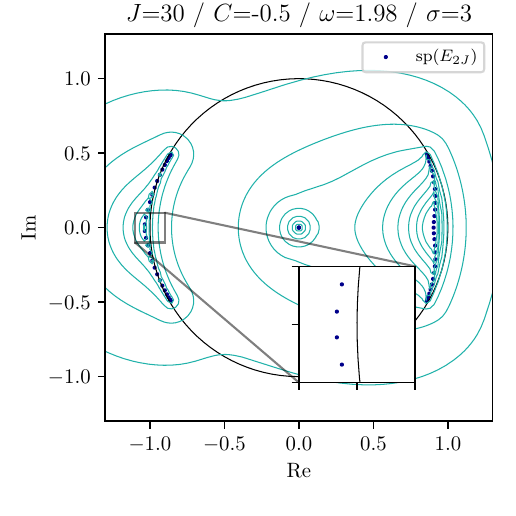}
        \includegraphics[width = 0.45\textwidth]{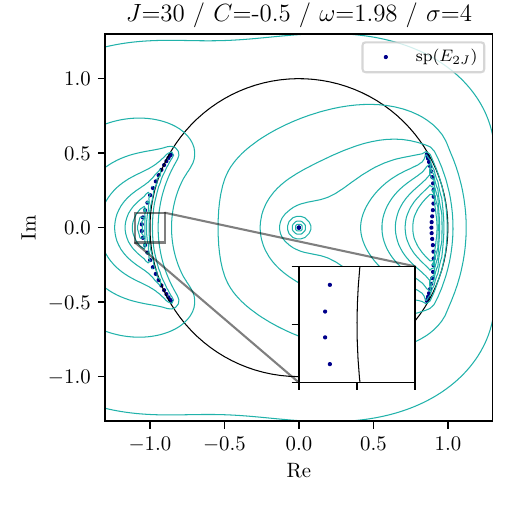} \\
        \includegraphics[width = 0.45\textwidth]{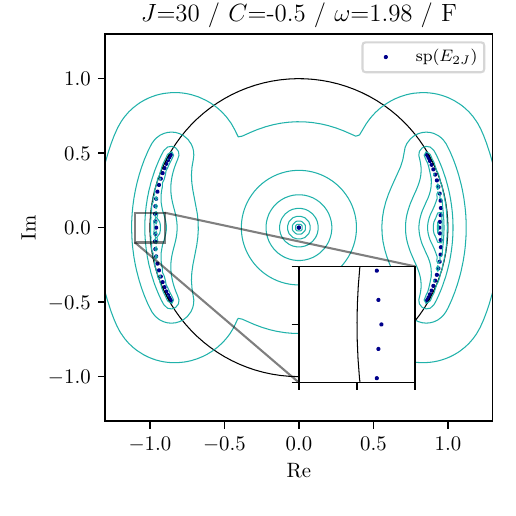} 
    \end{center}\caption{\label{fig:spectrum3}Spectra (dots) and $L^2$ pseudo-spectra (turquoise lines) when $\courantNumber<0$.}
\end{figure}

\begin{figure} 
    \begin{center}
        \includegraphics[width = 0.45\textwidth]{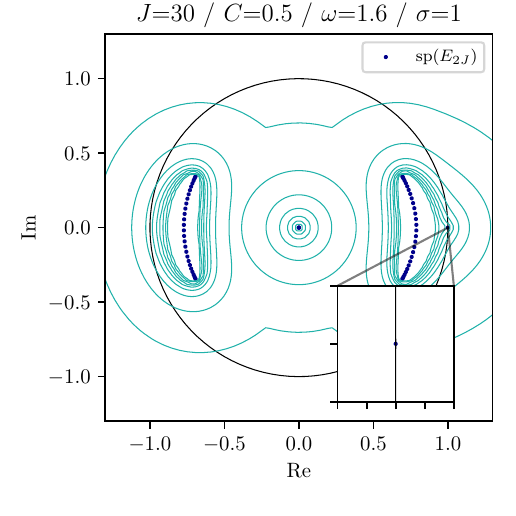}
        \includegraphics[width = 0.45\textwidth]{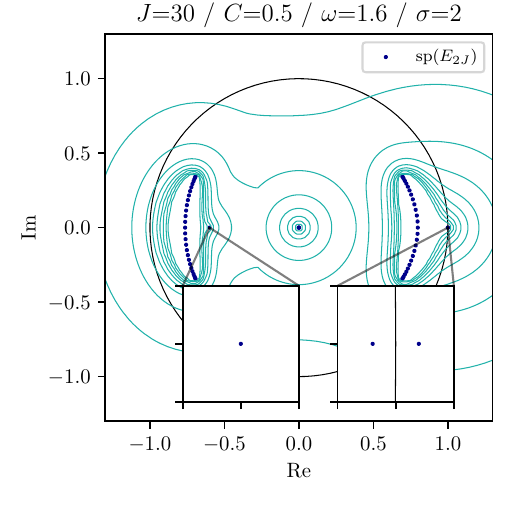} \\
        \includegraphics[width = 0.45\textwidth]{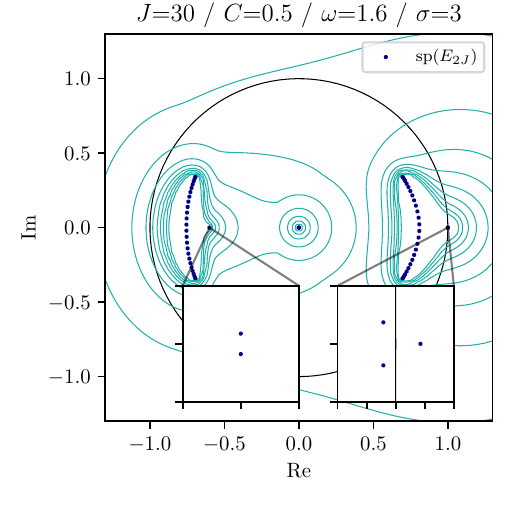}
        \includegraphics[width = 0.45\textwidth]{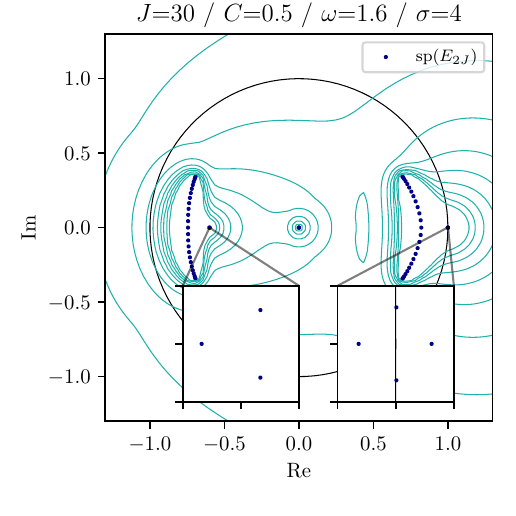} \\
        \includegraphics[width = 0.45\textwidth]{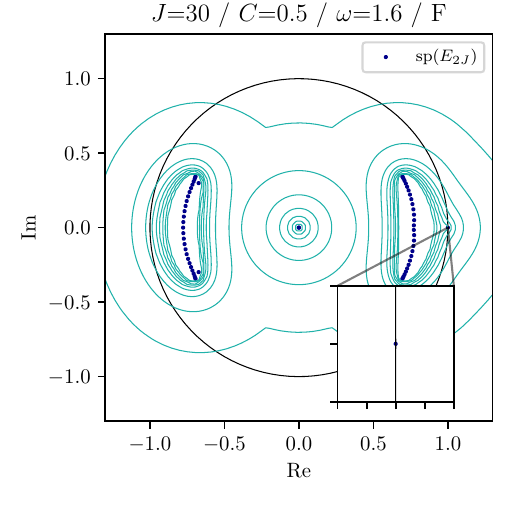} 
    \end{center}\caption{\label{fig:spectrum4}Spectra and $L^2$ pseudo-spectra (turquoise lines) when $\courantNumber > 0$.}
\end{figure}

As the structure of $\asymptoticSpectrumLetter_{\numberSpacePoints\to\infty}(\schemeMatrixFD)$, closely linked with GKS theory, has been described in \Cref{sec:asymptoticSpectra}, it is time to gain a deeper understanding on the surprising results gathered in \Cref{sec:numericalSimulationsStability}.
These atypical outcomes are  explained by finite dimensional effect with respect to $\numberSpacePoints$ and consequent deviations from the asymptotic spectra of \Cref{sec:asymptoticSpectra}.
Generally, problems arise on the real axis (see \Cref{fig:spectrum3} and \ref{fig:spectrum4}), as spectra are symmetric with respect to it, so we look at perturbations of a target \emph{quasi}-eigenvalue $\reals \ni \targetEigenvalue \not\in \spectrum(\schemeMatrixFD)$ to choose according to the situation at hand, such that $\targetEigenvalue \in \asymptoticSpectrumLetter_{\numberSpacePoints\to+\infty} (\schemeMatrixFD)$.
We would like to gauge the smallest deviation $\epsilon = \epsilon(\targetEigenvalue, \numberSpacePoints)$ such that $\targetEigenvalue+ \epsilon \in \spectrum(\schemeMatrixFD)$, which is linked with the problem of distance to singularity, see \cite{higham89matrix} and \cite[Section 49]{trefethen2005spectra}.
Notice that $\epsilon\in\complex$ \emph{a priori} to ensure that the previous problem has a solution. Nevertheless, we look for real approximations of it and their sign.
Into the characteristic equation \eqref{eq:quadraticEigenvalue}, assuming $|\epsilon|\ll 1$, and using a Taylor expansion, all this yields $0=\determinant ((\targetEigenvalue + \epsilon)\identityMatrix{2\numberSpacePoints} - \schemeMatrixFD) 
= \determinant (\targetEigenvalue\identityMatrix{2\numberSpacePoints} - \schemeMatrixFD) + \trace (\adjugate(\targetEigenvalue\identityMatrix{2\numberSpacePoints} - \schemeMatrixFD)) \epsilon + \bigO{\epsilon^2}$, thus, at leading order 
\begin{equation}\label{eq:bulgeEstimation}
    \epsilon \approx -\frac{\determinant (\targetEigenvalue\identityMatrix{2\numberSpacePoints} - \schemeMatrixFD)}{\trace (\adjugate(\targetEigenvalue\identityMatrix{2\numberSpacePoints} - \schemeMatrixFD))} = - (\trace ((\targetEigenvalue\identityMatrix{2\numberSpacePoints} - \schemeMatrixFD)^{-1}))^{-1} 
    =- \Biggl ( \sum_{\lambda \in \spectrum(\targetEigenvalue\identityMatrix{2\numberSpacePoints} - \schemeMatrixFD)} \frac{1}{\realPart{\lambda}}\Biggr )^{-1}.
\end{equation}
The approximation of $\epsilon$ in \eqref{eq:bulgeEstimation} is linked to the harmonic mean of the real part of the eigenvalues of the ``reminder'' matrix $\targetEigenvalue\identityMatrix{2\numberSpacePoints} - \schemeMatrixFD$, and is indeed the \strong{first iterate of a Newton's method}\footnote{Happily we neither consider $\epsilon \in \complex$, nor pursue Newton's iterations any further, avoiding to cope with Newton's fractals \cite{hubbard2001find}.} for \eqref{eq:quadraticEigenvalue}, with initial guess $\targetEigenvalue$.
For $\targetEigenvalue \in \asymptoticSpectrumLetter_{\numberSpacePoints\to\infty}(\schemeMatrixFD)$, then $\lim_{\numberSpacePoints \to +\infty}\epsilon = 0$, as---luckily---the right-hand side of \eqref{eq:bulgeEstimation}. Still, the fact that $\epsilon$ (or its approximation) converges to zero from above or below changes the situation and, when $|\targetEigenvalue| = 1$, the problem is very close to the one of checking \Cref{def:admissibleSolution}.
Observe that \eqref{eq:bulgeEstimation} is generally affected by an odd-even decoupling in $\numberSpacePoints$, for is looks at real perturbations $\epsilon$, whereas real eigenvalues close to the $\targetEigenvalue$ could alternatively appear and disappear according to the parity of $\numberSpacePoints$.
In the case $\relaxationParameter = 2$, much more can be said on \eqref{eq:bulgeEstimation} for small $\numberSpacePoints$'s, in particular when $\orderExtrapolation = 1$. This is the main advantage of this procedure against numerical computations of spectra or dominant eigenvalues through a power iteration. Showing this is the aim of the following result, proved in the Supplementary material.
\begin{proposition}\label{prop:perturbationSigma1}
    Let $\relaxationParameter = 2$ and consider \eqref{eq:papillon1}.
    For $\targetEigenvalue = -1$ (which is interesting whenever $\courantNumber < 0$), the estimation from \eqref{eq:bulgeEstimation} reads 
        \begin{equation*}
            \epsilon \approx 
            \begin{cases}
                \frac{\courantNumber + 2}{(\courantNumber + 3)\numberSpacePoints + \courantNumber + 1}, \qquad &\text{for} \quad \numberSpacePoints~\text{even},\\
                \frac{\courantNumber^2}{(\courantNumber^2 + \courantNumber + 2)\numberSpacePoints + \courantNumber^2 - \courantNumber - 2}, \qquad &\text{for} \quad \numberSpacePoints~\text{odd}.
            \end{cases}
        \end{equation*}
  On the other hand, for $\targetEigenvalue = 1$ (which is interesting whenever $\courantNumber > 0$), the estimation from \eqref{eq:bulgeEstimation} reads $\epsilon \approx -\tfrac{\courantNumber}{(\courantNumber-1)\numberSpacePoints + \courantNumber + 1}$.
\end{proposition}

\begin{figure}[h]
    \begin{center}
        \includegraphics[width = 0.48\textwidth]{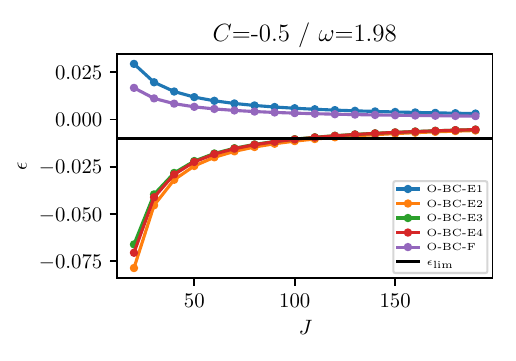}
        \includegraphics[width = 0.48\textwidth]{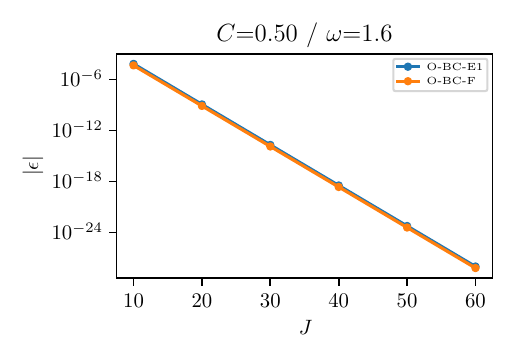}
        \includegraphics[width = 0.48\textwidth]{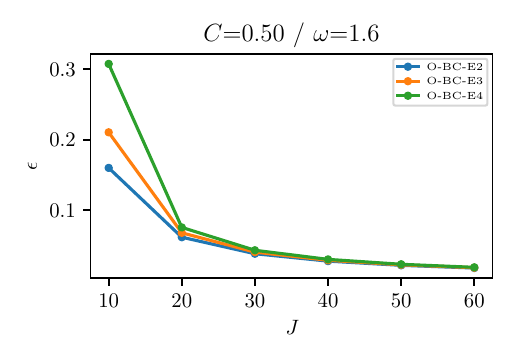}
    \end{center}\caption{\label{fig:bulges}Values of $\epsilon$ estimated using \eqref{eq:bulgeEstimation}. 
    For the first plot in the clockwise order, we observe that the conditions \eqref{eq:extrapolationBoundaryCondition} for $\orderExtrapolation\geq 2$ become stable starting from a certain $\numberSpacePoints\sim 100$, as empirically observed. For the second plot, we took the absolute value in order to use a logarithmic scale: the values corresponding to \eqref{eq:conditionKinRod} are indeed negative.}
\end{figure}

Let us now put all the previous analyses together to fill missing gaps in the understanding of the results of \Cref{sec:numericalSimulationsStability}.

\begin{itemize}
    \item For $\courantNumber < 0$. We focus on the case $\courantNumber = -1/2$, $\relaxationParameter = 1.98$, and $\numberSpacePoints = 30$, where \strong{unexpected instability} has been observed.

    The transient behavior $|\conservedMomentDiscrete_0^{\indexTime}|\propto \indexTime^{\orderExtrapolation - 2}$ can be explained by pseudo-spectra.
    Hinting at \Cref{fig:spectrum3}, we remark that the presence (resp. absence) of \strong{bulges} in pseudo-spectra, protruding outside the unit circle close to $\targetEigenvalue = -\sqrt{\relaxationParameter - 1}\approx - 1$ (this expression may lack justification at this stage), is coherent---as pointed out in \cite[Section 34]{trefethen2005spectra}, with the GKS instability (resp. stability) of the boundary conditions at hand when $\relaxationParameter = 2$.
    According to \cite[Section 14]{trefethen2005spectra}, the fact that protrusions outside the unit disk are increasingly severe while $\orderExtrapolation$ grows entails a more marked transient growth.

    Looking at the long-time behavior, thus at spectra, in the stable contexts \eqref{eq:papillon1} and \eqref{eq:conditionKinRod}, the eigenvalues tending towards the left asymptotic cluster do so \strong{from inside} the unit disk.
    Quite the opposite, in the unstable cases \eqref{eq:extrapolationBoundaryCondition} for $\orderExtrapolation \geq 2$, these eigenvalues converge to the limit object, located inside the unit disk, \strong{from outside}, yielding the instability at fixed $\numberSpacePoints$.
    More generally, when $\tfrac{2}{1-\courantNumber}\leq \relaxationParameter\leq 2$, issues arise because non-isolated eigenvalues with negative real part can be outside the unit circle for small $\numberSpacePoints$'s.
    Observing that one intersection of the asymptotic spectrum independent of the boundary condition---given by \Cref{lemma:limitSpectrumToeplitz}---with the real axis, is given by $-\sqrt{\relaxationParameter-1}$, we consider $\targetEigenvalue = -\sqrt{\relaxationParameter-1}$ to apply \eqref{eq:bulgeEstimation}.
    Thus, if $\epsilon$, practically estimated using \eqref{eq:bulgeEstimation}, is such that $\epsilon > 0$, we can infer that we are in a stable situation. Conversely, if $\epsilon < 0$, we can face instability if the amplitude of $\epsilon$ is too large.
    In particular, we look for $\epsilon < \sqrt{\relaxationParameter - 1} - 1 = \epsilon_{\text{lim}}\leq 0$.
    Of course, this is just a rough estimation by one iteration of the Newton's method from an educated asymptotically-inspired guess $\targetEigenvalue$.
    Still, the results from \Cref{fig:bulges} corroborate empirical observations, showing that while we expect \eqref{eq:papillon1} and \eqref{eq:conditionKinRod} be stable essentially for every $\numberSpacePoints$, the conditions \eqref{eq:extrapolationBoundaryCondition} for $\orderExtrapolation\geq 2$ become stable starting from a certain $\numberSpacePoints\sim 100$. Moreover, since $1.98\approx 2$, results are coherent with \Cref{prop:perturbationSigma1}.

    To summarize, we see that the GKS-unexpected instability comes from the impact of interactions between boundaries on the spectra of the scheme matrix when $\numberSpacePoints$ is small.
    The larger the extrapolation stencil, the later we converge to the asymptotic spectrum, with potentially more severe instability.
    Notice that the repeated-reflections theory by \cite{trefethen1985stability}, mentioned in \Cref{sec:numericalSimulationsStability}, analyzes interactions between boundaries in a time-extensive fashion, and is totally compatible with the inspection of the spectra, where interactions are ``instantaneously'' taken into consideration, for the latter conveys information away from the initial transient (\emph{i.e.} $\indexTime\gg 1$), after enough reflections have taken place.

    \item For $\courantNumber > 0$. We focus on the case $\courantNumber = 1/2$, $\relaxationParameter = 1.6$, and $\numberSpacePoints = 30$, where---unexpectedly---stable behaviors have been observed.

    In \Cref{fig:spectrum4}, all plots feature \strong{bulges} in the pseudo-spectra, protruding outside the unit circle close to $\targetEigenvalue = 1$. 
    This indicates that all boundary conditions are GKS-unstable.

    However in the long-time limit, for the empirically-stable cases \eqref{eq:papillon1} and \eqref{eq:conditionKinRod}, it could seem that $\targetEigenvalue = 1 \in \spectrum(\schemeMatrixFD)$  and simple.
    This is almost true, for we obtain (in exact arithmetic) $0\neq\determinant(\identityMatrix{60}-\schemeMatrixFDSize{60}) \sim -10^{-21}$ for the former and $0\neq\determinant(\identityMatrix{60}-\schemeMatrixFDSize{60}) \sim 10^{-21}$ for the latter boundary condition.
    These determinants being extremely tiny, one must be aware when utilizing floating point arithmetic. 
    From \Cref{fig:bulges}, we see that $\epsilon \sim 10^{-14}>0$ for the former and $\epsilon \sim -10^{-14}<0$ for the latter using \eqref{eq:bulgeEstimation}.
    Though these predictions need not be highly accurate, they suggest that \eqref{eq:papillon1} is theoretically unstable, whereas \eqref{eq:conditionKinRod} is stable. Practically, these isolated eigenvalues dwell so close ($10^{-14}$!) to $\targetEigenvalue = 1$ that simulations remain stable. Moreover, we see that $\epsilon$ goes to zero exponentially with $\numberSpacePoints$.
    For \eqref{eq:conditionKinRod}, \Cref{fig:spectrum4} features two isolated complex conjugate eigenvalues close to the left cluster, which are nothing but (almost) the two complex conjugate roots of \eqref{eq:tmp7}.

    The situation turns out radically different for \eqref{eq:extrapolationBoundaryCondition} with $\orderExtrapolation \geq 2$.
    Here, instabilities show up with growth $\propto \indexTime^{\orderExtrapolation - 1}$. Looking at \Cref{fig:spectrum4}, we see $\orderExtrapolation$ isolated eigenvalues close to the target $\targetEigenvalue = 1$, some outside the unit disk.
    They should theoretically generate an exponential growth; however, they are so close to $\targetEigenvalue = 1$ (although quite not as for $\orderExtrapolation = 1$, \confer{} \Cref{fig:bulges}) that their \emph{effective} role is the one of a multiple eigenvalue on the unit circle with algebraic multiplicity $\orderExtrapolation$, explaining the growth $\propto \indexTime^{\orderExtrapolation - 1}$.
    This asymptotic regime is rapidly reached, see \Cref{fig:instab_pos_vel_few_points}, analogously to the convergence of a power iteration, which is exponential in the base given by the ratio of the modulus of the second largest eigenvalue (in modulus) and the largest eigenvalue. For the considered parameters, the basis is roughly $\sqrt{\relaxationParameter-1} \approx 0.775$.
    The presence of $\orderExtrapolation$ eigenvalues close to $\targetEigenvalue = 1$ could have been predicted without plots of spectra, taking advantage of \Cref{thm:countingThroughReflection}, after a glance at \eqref{eq:tmp18}.
    A second explanation comes from observing the pseudo-spectra bulges around $\targetEigenvalue = 1$: the larger the number of almost-coinciding eigenvalues here, the ``more singular'' the resolvent $(\timeShiftOperator\identityMatrix{2\numberSpacePoints} - \schemeMatrixFD)^{-1}$ is in this neighborhood, due to the almost-superimposed poles associated to each eigenvalue. Thus, for a given level-set, the bulge around $\targetEigenvalue$ is larger.
\end{itemize}

\FloatBarrier

\section{Conclusions and perspectives}\label{sec:conclusions}

We have introduced \strong{boundary conditions} for a two-velocities 1D lattice Boltzmann scheme.
Rewriting things in terms of Finite Difference schemes on the \strong{variable of interest}, we have seen that GKS-stable boundary conditions \strong{decrease} the order of second-order accurate bulk schemes when initializing at equilibrium, whereas unstable boundary conditions do not.
Therefore, for the toy numerical scheme at hand, the best combination of boundary conditions in terms of consistency and stability, upon performing corrections due to the initilization at equilibrium, are either \eqref{eq:papillon1} or \eqref{eq:conditionKinRod}.
Inspired by numerical simulations, we have observed that GKS theory is \strong{not} the end of the story about stability when grids are coarse \strong{and} two boundaries are present---with waves having travelled several times through the domain.
    In this framework, there is \strong{no finalized decoupling} between inner/boundary schemes and left/right boundaries in terms of spectrum, and the choice of boundary condition has a significant impact.
Finally, we have linked---to the best of our knowledge for the first time---the order of poles of the \strong{reflection coefficient} with the number of eigenvalues in the scheme matrix tending to isolated points in the asymptotic spectrum.
This explains the growth of  instabilities in some circumstances.

Two possible perspectives are envisioned at this stage.
The first one will be to extend the work to \strong{more involved schemes}, such as three-velocities schemes and schemes with an arbitrary number of discrete velocities having a two-relaxation-times (TRT) link-structure with ``magic parameters'' equal to $1/4$.
We hope being able to turn these schemes into Finite Difference ones.
The second and more ambitious path will be to develop a GKS theory \strong{directly on lattice Boltzmann schemes} without having to transform them.
To this end, the modal ansatz must carefully identify unstable modes linked solely with non-conserved moments, which are not interesting, and filter them while considering the conserved moment.

\section*{Acknowledgements}

The author thanks Victor Michel-Dansac for useful advice on the manuscript and Jean-François Coulombel for the enlightening discussions on this topic.
This work of the Interdisciplinary Thematic Institute IRMIA++, as part of the ITI 2021-2028 program of the University of Strasbourg, CNRS and Inserm, was supported by IdEx Unistra (ANR-10-IDEX-0002), and by SFRI-STRAT’US project (ANR-20-SFRI-0012) under the framework of the French Investments for the Future Program.

\bibliographystyle{alpha}
\bibliography{biblio}

\appendix

\section{Proofs}

\subsection{Proof of \Cref{prop:bulkAndInflow}}\label{app:proof:bulkAndInflow}

\begin{proofWithoutProof}
    The bulk schemes in \eqref{eq:bulkFDInitial} and \eqref{eq:bulkFDScheme} can be found as described in \cite{bellotti2022finite, bellotti2024initialisation}, except for the one in \eqref{eq:bulkFDScheme} for $\indexSpace = 1, \numberSpacePoints - 2$. In this case, the scheme might be different  due to boundary conditions. 
    This is verified in the proof of \Cref{prop:outflowExtrapolation}.
    We now check $\indexSpace = \numberSpacePoints - 2$.
    The boundary scheme on the distribution functions reads 
    \begin{align}
        \distributionFunctionDiscrete^{+, \indexTime + 1}_{\numberSpacePoints - 1} &= \distributionFunctionDiscrete^{+, \indexTime \collided}_{\numberSpacePoints - 2} = \tfrac{1}{2} \conservedMomentDiscrete_{\numberSpacePoints - 2}^{\indexTime} + \tfrac{1-\relaxationParameter}{2\latticeVelocity} \nonConservedMomentDiscrete_{\numberSpacePoints - 2}^{\indexTime} + \tfrac{\relaxationParameter}{2\latticeVelocity} \flux(\conservedMomentDiscrete_{\numberSpacePoints - 2}^{\indexTime}), \label{eq:tmp8}\\
        \distributionFunctionDiscrete_{\numberSpacePoints - 1}^{-, \indexTime + 1} &= - \distributionFunctionDiscrete_{\numberSpacePoints - 2}^{+, \indexTime \collided} + \boundaryDatumInflow(\timeGridPoint{\indexTime + 1}) = -\tfrac{1}{2} \conservedMomentDiscrete_{\numberSpacePoints - 2}^{\indexTime} - \tfrac{1-\relaxationParameter}{2\latticeVelocity} \nonConservedMomentDiscrete_{\numberSpacePoints - 2}^{\indexTime} - \tfrac{\relaxationParameter}{2\latticeVelocity} \flux(\conservedMomentDiscrete_{\numberSpacePoints - 2}^{\indexTime}) + \boundaryDatumInflow(\timeGridPoint{\indexTime + 1}). \label{eq:tmp9}
    \end{align}
    Taking sum $\eqref{eq:tmp8} + \eqref{eq:tmp9}$ and difference  $\latticeVelocity\eqref{eq:tmp8} - \latticeVelocity \eqref{eq:tmp9}$ of these equations yields 
    \begin{align}
        \conservedMomentDiscrete_{\numberSpacePoints - 1}^{\indexTime + 1} &= \boundaryDatumInflow(\timeGridPoint{\indexTime + 1}),  \label{eq:boundaryRightU}\\
        \nonConservedMomentDiscrete_{\numberSpacePoints - 1}^{\indexTime + 1} &= \latticeVelocity\conservedMomentDiscrete_{\numberSpacePoints - 2}^{\indexTime} + (1-\relaxationParameter) \nonConservedMomentDiscrete_{\numberSpacePoints - 2}^{\indexTime} + \relaxationParameter \flux(\conservedMomentDiscrete_{\numberSpacePoints - 2}^{\indexTime}) - \latticeVelocity \boundaryDatumInflow(\timeGridPoint{\indexTime + 1}).\label{eq:boundaryRightV}
    \end{align}
    The bulk scheme written on the moments reads, for $\indexSpace \in \integerInterval{1}{\numberSpacePoints - 2}$
    \begin{align}
        \conservedMomentDiscrete_{\indexSpace}^{\indexTime + 1} &= \tfrac{1}{2} (\conservedMomentDiscrete_{\indexSpace-1}^{\indexTime} + \conservedMomentDiscrete_{\indexSpace+1}^{\indexTime}) + \tfrac{1-\relaxationParameter}{2\latticeVelocity} (\nonConservedMomentDiscrete_{\indexSpace-1}^{\indexTime} - \nonConservedMomentDiscrete_{\indexSpace+1}^{\indexTime}) + \tfrac{\relaxationParameter}{2\latticeVelocity} (\flux(\conservedMomentDiscrete_{\indexSpace-1}^{\indexTime}) - \flux(\conservedMomentDiscrete_{\indexSpace+1}^{\indexTime})), \label{eq:bulkSchemeForU}\\
        \nonConservedMomentDiscrete_{\indexSpace}^{\indexTime + 1} &= \tfrac{\latticeVelocity}{2} (\conservedMomentDiscrete_{\indexSpace-1}^{\indexTime} - \conservedMomentDiscrete_{\indexSpace+1}^{\indexTime}) + \tfrac{1-\relaxationParameter}{2} (\nonConservedMomentDiscrete_{\indexSpace-1}^{\indexTime} + \nonConservedMomentDiscrete_{\indexSpace+1}^{\indexTime}) + \tfrac{\relaxationParameter}{2} (\flux(\conservedMomentDiscrete_{\indexSpace-1}^{\indexTime}) + \flux(\conservedMomentDiscrete_{\indexSpace+1}^{\indexTime})). \label{eq:bulkSchemeForV}
    \end{align}
    Writing \eqref{eq:bulkSchemeForU} at $\indexSpace = \numberSpacePoints - 2$, we have to eliminate the term in $\nonConservedMomentDiscrete_{\numberSpacePoints - 3}^{\indexTime} - \nonConservedMomentDiscrete_{\numberSpacePoints - 1}^{\indexTime}$.
    Using \eqref{eq:bulkSchemeForV} and \eqref{eq:boundaryRightV} gives
    \begin{equation*}
        \nonConservedMomentDiscrete_{\numberSpacePoints - 3}^{\indexTime} - \nonConservedMomentDiscrete_{\numberSpacePoints - 1}^{\indexTime} = \tfrac{\latticeVelocity}{2} (\conservedMomentDiscrete_{\numberSpacePoints-4}^{\indexTime - 1} - 3\conservedMomentDiscrete_{\numberSpacePoints-2}^{\indexTime - 1}) + \tfrac{1-\relaxationParameter}{2} (\nonConservedMomentDiscrete_{\numberSpacePoints-4}^{\indexTime - 1} - \nonConservedMomentDiscrete_{\numberSpacePoints-2}^{\indexTime - 1}) + \tfrac{\relaxationParameter}{2} (\flux(\conservedMomentDiscrete_{\numberSpacePoints-4}^{\indexTime - 1}) - \flux(\conservedMomentDiscrete_{\numberSpacePoints-2}^{\indexTime - 1})) + \latticeVelocity \boundaryDatumInflow(\timeGridPoint{\indexTime}).
    \end{equation*}
    There is still the unknown $\nonConservedMomentDiscrete$ on the right-hand side. 
    This term can be obtained using \eqref{eq:bulkSchemeForU} written for $\indexSpace = \numberSpacePoints - 3$:
    \begin{equation*}
        \tfrac{1-\relaxationParameter}{2} (\nonConservedMomentDiscrete_{\numberSpacePoints - 4}^{\indexTime - 1} - \nonConservedMomentDiscrete_{\numberSpacePoints - 2}^{\indexTime - 1})  = \latticeVelocity \conservedMomentDiscrete_{\numberSpacePoints - 3}^{\indexTime} - \tfrac{\latticeVelocity}{2} (\conservedMomentDiscrete_{\numberSpacePoints - 4}^{\indexTime - 1} + \conservedMomentDiscrete_{\numberSpacePoints - 2}^{\indexTime - 1}) - \tfrac{\relaxationParameter}{2} (\flux(\conservedMomentDiscrete_{\numberSpacePoints - 4}^{\indexTime - 1}) - \flux(\conservedMomentDiscrete_{\numberSpacePoints - 2}^{\indexTime - 1})),
    \end{equation*}
    which thus gives $\nonConservedMomentDiscrete_{\numberSpacePoints - 3}^{\indexTime} - \nonConservedMomentDiscrete_{\numberSpacePoints - 1}^{\indexTime} = \latticeVelocity \conservedMomentDiscrete_{\numberSpacePoints - 3}^{\indexTime} -2 \latticeVelocity \conservedMomentDiscrete_{\numberSpacePoints-2}^{\indexTime - 1}+ \latticeVelocity \boundaryDatumInflow(\timeGridPoint{\indexTime}) = \latticeVelocity \conservedMomentDiscrete_{\numberSpacePoints - 3}^{\indexTime} -2 \latticeVelocity \conservedMomentDiscrete_{\numberSpacePoints-2}^{\indexTime - 1}+ \latticeVelocity \conservedMomentDiscrete_{\numberSpacePoints-1}^{\indexTime}$ using \eqref{eq:boundaryRightU} in the last inequality.
    Back into \eqref{eq:bulkSchemeForU} at $\indexSpace = \numberSpacePoints - 2$, this gives \eqref{eq:bulkFDScheme} as claimed.
\end{proofWithoutProof}

\subsection{Proof of \Cref{prop:outflowExtrapolation}}\label{app:proof:outflowExtrapolation}

\begin{proofWithoutProof}
    From \eqref{eq:extrapolationBoundaryCondition}, we obtain 
    \begin{align}
        \distributionFunctionDiscrete_{0}^{+, \indexTime +1} &=  \sum_{\indexSpace = 0}^{\orderExtrapolation - 1} \coefficientExtrapolation_{\indexSpace} \distributionFunctionDiscrete_{\indexSpace}^{+, \indexTime \collided} + \sourceTermBoundaryDiscrete_0^{\indexTime + 1} = \sum_{\indexSpace = 0}^{\orderExtrapolation - 1} \coefficientExtrapolation_{\indexSpace} \bigl ( \tfrac{1}{2} \conservedMomentDiscrete_{\indexSpace}^{\indexTime} + \tfrac{1-\relaxationParameter}{2\latticeVelocity}\nonConservedMomentDiscrete_{\indexSpace}^{\indexTime} + \tfrac{\relaxationParameter}{2\latticeVelocity} \flux(\conservedMomentDiscrete_{\indexSpace}^{\indexTime})\bigr )+ \sourceTermBoundaryDiscrete_0^{\indexTime + 1}, \label{eq:fPlusBoundaryGenericExtrap}\\
        \distributionFunctionDiscrete_{0}^{-, \indexTime +1} &= \tfrac{1}{2} \conservedMomentDiscrete_{1}^{\indexTime} - \tfrac{1-\relaxationParameter}{2\latticeVelocity}\nonConservedMomentDiscrete_{1}^{\indexTime} - \tfrac{\relaxationParameter}{2\latticeVelocity} \flux(\conservedMomentDiscrete_{1}^{\indexTime}).\label{eq:fMinusBoundaryGenericExtrap}
    \end{align}
    Taking $\eqref{eq:fPlusBoundaryGenericExtrap} + \eqref{eq:fMinusBoundaryGenericExtrap}$ and $\latticeVelocity\eqref{eq:fPlusBoundaryGenericExtrap} - \latticeVelocity \eqref{eq:fMinusBoundaryGenericExtrap}$ provides:
    \begin{align}
        \conservedMomentDiscrete_{0}^{\indexTime +1} &= \tfrac{1}{2} \Bigl ( \sum_{\substack{\indexSpace = 0 \\ \indexSpace \neq 1}}^{\orderExtrapolation - 1} \coefficientExtrapolation_{\indexSpace} \conservedMomentDiscrete_{\indexSpace}^{\indexTime} + (\coefficientExtrapolation_1 + 1) \conservedMomentDiscrete_{1}^{\indexTime} \Bigr ) + \tfrac{1-\relaxationParameter}{2\latticeVelocity} \Bigl ( \sum_{\substack{\indexSpace = 0 \\ \indexSpace \neq 1}}^{\orderExtrapolation - 1} \coefficientExtrapolation_{\indexSpace} \nonConservedMomentDiscrete_{\indexSpace}^{\indexTime} + (\coefficientExtrapolation_1 - 1) \nonConservedMomentDiscrete_{1}^{\indexTime} \Bigr ) + \tfrac{\relaxationParameter}{2\latticeVelocity} \Bigl ( \sum_{\substack{\indexSpace = 0 \\ \indexSpace \neq 1}}^{\orderExtrapolation - 1} \coefficientExtrapolation_{\indexSpace} \flux(\conservedMomentDiscrete_{\indexSpace}^{\indexTime}) + (\coefficientExtrapolation_1 - 1) \flux(\conservedMomentDiscrete_{1}^{\indexTime}) \Bigr )  + \sourceTermBoundaryDiscrete_0^{\indexTime + 1}, \label{eq:uBoundaryGenericExtrapolation} \\
        \nonConservedMomentDiscrete_{0}^{\indexTime +1} &= \tfrac{\latticeVelocity}{2} \Bigl ( \sum_{\substack{\indexSpace = 0 \\ \indexSpace \neq 1}}^{\orderExtrapolation - 1} \coefficientExtrapolation_{\indexSpace} \conservedMomentDiscrete_{\indexSpace}^{\indexTime} + (\coefficientExtrapolation_1 - 1) \conservedMomentDiscrete_{1}^{\indexTime} \Bigr ) + \tfrac{1-\relaxationParameter}{2} \Bigl ( \sum_{\substack{\indexSpace = 0 \\ \indexSpace \neq 1}}^{\orderExtrapolation - 1} \coefficientExtrapolation_{\indexSpace} \nonConservedMomentDiscrete_{\indexSpace}^{\indexTime} + (\coefficientExtrapolation_1 + 1) \nonConservedMomentDiscrete_{1}^{\indexTime} \Bigr ) + \tfrac{\relaxationParameter}{2} \Bigl ( \sum_{\substack{\indexSpace = 0 \\ \indexSpace \neq 1}}^{\orderExtrapolation - 1} \coefficientExtrapolation_{\indexSpace} \flux(\conservedMomentDiscrete_{\indexSpace}^{\indexTime}) + (\coefficientExtrapolation_1 + 1) \flux(\conservedMomentDiscrete_{1}^{\indexTime}) \Bigr )  + \latticeVelocity\sourceTermBoundaryDiscrete_0^{\indexTime + 1},\label{eq:vBoundaryGenericExtrapolation}
    \end{align}
    where in the case $\orderExtrapolation = 1$, we consider $\coefficientExtrapolation_1 = 0$.
    Let us first show that the scheme for $\indexSpace = 1$ is indeed given by \eqref{eq:bulkFDScheme} from \Cref{prop:bulkAndInflow}.
    When considering \eqref{eq:bulkSchemeForU} for $\indexSpace = 1$, we have to estimate the difference $\nonConservedMomentDiscrete_0^{\indexTime} - \nonConservedMomentDiscrete_2^{\indexTime}$, thus we obtain 
    \begin{multline*}
        \nonConservedMomentDiscrete_0^{\indexTime} - \nonConservedMomentDiscrete_2^{\indexTime} = \tfrac{\latticeVelocity}{2} \Bigl ( \sum_{\substack{\indexSpace = 0 \\ \indexSpace \neq 1}}^{\sigma - 1} \coefficientExtrapolation_{\indexSpace} \conservedMomentDiscrete_{\indexSpace}^{\indexTime-1} + (\coefficientExtrapolation_1 - 2) \conservedMomentDiscrete_{1}^{\indexTime-1} + \conservedMomentDiscrete_{3}^{\indexTime-1}  \Bigr ) \\
         + \tfrac{1-\relaxationParameter}{2} \Bigl ( \sum_{\substack{\indexSpace = 0 \\ \indexSpace \neq 1}}^{\sigma - 1} \coefficientExtrapolation_{\indexSpace} \nonConservedMomentDiscrete_{\indexSpace}^{\indexTime-1} + \coefficientExtrapolation_1 \nonConservedMomentDiscrete_{1}^{\indexTime-1} - \nonConservedMomentDiscrete_3^{\indexTime - 1}\Bigr ) + \tfrac{\relaxationParameter}{2} \Bigl ( \sum_{\substack{\indexSpace = 0 \\ \indexSpace \neq 1}}^{\sigma - 1} \coefficientExtrapolation_{\indexSpace} \flux(\conservedMomentDiscrete_{\indexSpace}^{ \indexTime-1}) + \coefficientExtrapolation_1 \flux(\conservedMomentDiscrete_{1}^{ \indexTime-1})  - \flux(\conservedMomentDiscrete_3^{ \indexTime - 1}) \Bigr )  + \latticeVelocity\sourceTermBoundaryDiscrete_0^{\indexTime}.
    \end{multline*}
    In order to get rid of the terms in $\nonConservedMomentDiscrete^{\indexTime - 1}$ on the right-hand side of this expression, we consider 
    \begin{multline*}
        \conservedMomentDiscrete_0^{\indexTime} + \conservedMomentDiscrete_2^{\indexTime} = 
        \tfrac{1}{2} \Bigl ( \sum_{\substack{\indexSpace = 0 \\ \indexSpace \neq 1}}^{\orderExtrapolation - 1} \coefficientExtrapolation_{\indexSpace} \conservedMomentDiscrete_{\indexSpace}^{\indexTime - 1} + (\coefficientExtrapolation_1 + 2) \conservedMomentDiscrete_{1}^{\indexTime-1} + \conservedMomentDiscrete_{3}^{\indexTime-1} \Bigr ) \\
        + \tfrac{1-\relaxationParameter}{2\latticeVelocity} \Bigl ( \sum_{\substack{\indexSpace = 0 \\ \indexSpace \neq 1}}^{\orderExtrapolation - 1} \coefficientExtrapolation_{\indexSpace} \nonConservedMomentDiscrete_{\indexSpace}^{\indexTime-1} + \coefficientExtrapolation_1 \nonConservedMomentDiscrete_{1}^{\indexTime-1} - \nonConservedMomentDiscrete_{3}^{\indexTime-1} \Bigr ) + \tfrac{\relaxationParameter}{2\latticeVelocity} \Bigl ( \sum_{\substack{\indexSpace = 0 \\ \indexSpace \neq 1}}^{\orderExtrapolation - 1} \coefficientExtrapolation_{\indexSpace} \flux(\conservedMomentDiscrete_{\indexSpace}^{\indexTime-1}) + \coefficientExtrapolation_1 \flux(\conservedMomentDiscrete_{1}^{\indexTime-1}) - \flux(\conservedMomentDiscrete_{3}^{\indexTime-1}) \Bigr )  + \sourceTermBoundaryDiscrete_0^{\indexTime}.
    \end{multline*}
    From this, we deduce that $\nonConservedMomentDiscrete_0^{\indexTime} - \nonConservedMomentDiscrete_2^{\indexTime} = -2\latticeVelocity \conservedMomentDiscrete_1^{\indexTime - 1} + \latticeVelocity (\conservedMomentDiscrete_0^{\indexTime} + \conservedMomentDiscrete_2^{\indexTime})$ thus the claim. 
    We now focus on the first cell, indexed by $\indexSpace = 0$.
    The expression involving the non-conserved moment on the right-hand side of \eqref{eq:uBoundaryGenericExtrapolation} reads
\begin{multline}\label{eq:tmp14}
    \sum_{\substack{\indexSpace = 0 \\ \indexSpace \neq 1}}^{\orderExtrapolation - 1} \coefficientExtrapolation_{\indexSpace} \nonConservedMomentDiscrete_{\indexSpace}^{\indexTime} + (\coefficientExtrapolation_1 - 1) \nonConservedMomentDiscrete_{1}^{\indexTime} = \\
    \tfrac{\latticeVelocity}{2} \Bigl ( (\coefficientExtrapolation_0^2 + \coefficientExtrapolation_1 - 1) \conservedMomentDiscrete_{0}^{\indexTime - 1} + (\coefficientExtrapolation_0 (\coefficientExtrapolation_1-1) + \coefficientExtrapolation_2) \conservedMomentDiscrete_{1}^{\indexTime - 1} + (\coefficientExtrapolation_0 \coefficientExtrapolation_2 - \coefficientExtrapolation_1  + 1 + \coefficientExtrapolation_3) \conservedMomentDiscrete_{2}^{\indexTime - 1} 
    + \sum_{\indexSpace = 3}^{\orderExtrapolation} (\coefficientExtrapolation_0 \coefficientExtrapolation_{\indexSpace} + \coefficientExtrapolation_{\indexSpace + 1} - \coefficientExtrapolation_{\indexSpace - 1}) \conservedMomentDiscrete_{\indexSpace}^{\indexTime - 1} \Bigr ) \\
    + \tfrac{1-\relaxationParameter}{2} \Bigl ( (\coefficientExtrapolation_0^2 + \coefficientExtrapolation_1 - 1) \nonConservedMomentDiscrete_{0}^{\indexTime - 1} + (\coefficientExtrapolation_0 (\coefficientExtrapolation_1+1) + \coefficientExtrapolation_2) \nonConservedMomentDiscrete_{1}^{\indexTime - 1} + (\coefficientExtrapolation_0 \coefficientExtrapolation_2 + \coefficientExtrapolation_1  - 1 + \coefficientExtrapolation_3) \nonConservedMomentDiscrete_{2}^{\indexTime - 1} 
    + \sum_{\indexSpace = 3}^{\orderExtrapolation } (\coefficientExtrapolation_0 \coefficientExtrapolation_{\indexSpace} + \coefficientExtrapolation_{\indexSpace + 1} + \coefficientExtrapolation_{\indexSpace - 1}) \nonConservedMomentDiscrete_{\indexSpace}^{\indexTime - 1} \Bigr ) \\
    + \tfrac{\relaxationParameter}{2} \Bigl ( (\coefficientExtrapolation_0^2 + \coefficientExtrapolation_1 - 1) \flux(\conservedMomentDiscrete_{0}^{ \indexTime - 1}) + (\coefficientExtrapolation_0 (\coefficientExtrapolation_1+1) + \coefficientExtrapolation_2) \flux(\conservedMomentDiscrete_{1}^{\indexTime - 1}) + (\coefficientExtrapolation_0 \coefficientExtrapolation_2 + \coefficientExtrapolation_1  - 1 + \coefficientExtrapolation_3) \flux(\conservedMomentDiscrete_{2}^{ \indexTime - 1}) 
    + \sum_{\indexSpace = 3}^{\orderExtrapolation} (\coefficientExtrapolation_0 \coefficientExtrapolation_{\indexSpace} + \coefficientExtrapolation_{\indexSpace + 1} + \coefficientExtrapolation_{\indexSpace - 1}) \flux(\conservedMomentDiscrete_{\indexSpace}^{\indexTime - 1})  \Bigr ) + \latticeVelocity \coefficientExtrapolation_0 \sourceTermBoundaryDiscrete_0^{\indexTime},
\end{multline}
where we used \eqref{eq:vBoundaryGenericExtrapolation} and \eqref{eq:bulkSchemeForV} at the previous time step.
In \eqref{eq:tmp14}, one has to set $\coefficientExtrapolation_{\indexSpace} = 0$ for $\indexSpace \geq \orderExtrapolation$.
With this in mind, we try to get rid of the expression in $\nonConservedMomentDiscrete^{\indexTime - 1}$ on the right-hand side of \eqref{eq:tmp14} using the equations on $\conservedMomentDiscrete^{\indexTime}$.
We introduce $\orderExtrapolationBiased \definitionEquality \max(1, \orderExtrapolation - 1) + 1$, the coefficients $\coefficientCombinationElimination_0, \dots \coefficientCombinationElimination_{\orderExtrapolationBiased - 1} \in \reals$, and write
\begin{multline*}
    \latticeVelocity \sum_{\indexSpace = 0}^{\orderExtrapolationBiased - 1}\coefficientCombinationElimination_{\indexSpace}\conservedMomentDiscrete_{\indexSpace}^{\indexTime} = 
    \tfrac{\latticeVelocity}{2} \Bigl ( (\coefficientExtrapolation_0 \coefficientCombinationElimination_0 + \coefficientCombinationElimination_1 )\conservedMomentDiscrete_0^{\indexTime - 1} + ((\coefficientExtrapolation_1 + 1)\coefficientCombinationElimination_0 + \coefficientCombinationElimination_2 ) \conservedMomentDiscrete_1^{\indexTime - 1} + (\coefficientExtrapolation_2 \coefficientCombinationElimination_0 + \coefficientCombinationElimination_3 + \coefficientCombinationElimination_1) \conservedMomentDiscrete_2^{\indexTime - 1} 
    + \sum_{\indexSpace = 3}^{\orderExtrapolationBiased} (\coefficientExtrapolation_{\indexSpace}\coefficientCombinationElimination_0 + \coefficientCombinationElimination_{\indexSpace + 1} + \coefficientCombinationElimination_{\indexSpace - 1}) \conservedMomentDiscrete_{\indexSpace}^{\indexTime - 1} \Bigr ) \\
    +\tfrac{1-\relaxationParameter}{2} \Bigl ( (\coefficientExtrapolation_0 \coefficientCombinationElimination_0 + \coefficientCombinationElimination_1 )\nonConservedMomentDiscrete_0^{\indexTime - 1} + ((\coefficientExtrapolation_1 - 1)\coefficientCombinationElimination_0 + \coefficientCombinationElimination_2 ) \nonConservedMomentDiscrete_1^{\indexTime - 1} + (\coefficientExtrapolation_2 \coefficientCombinationElimination_0 + \coefficientCombinationElimination_3 - \coefficientCombinationElimination_1) \nonConservedMomentDiscrete_2^{\indexTime - 1} 
    + \sum_{\indexSpace = 3}^{\orderExtrapolationBiased} (\coefficientExtrapolation_{\indexSpace}\coefficientCombinationElimination_0 + \coefficientCombinationElimination_{\indexSpace + 1} - \coefficientCombinationElimination_{\indexSpace - 1}) \nonConservedMomentDiscrete_{\indexSpace}^{\indexTime - 1} \Bigr ) \\
    +\tfrac{\relaxationParameter}{2} \Bigl ( (\coefficientExtrapolation_0 \coefficientCombinationElimination_0 + \coefficientCombinationElimination_1 )\flux(\conservedMomentDiscrete_0^{\indexTime - 1}) + ((\coefficientExtrapolation_1 - 1)\coefficientCombinationElimination_0 + \coefficientCombinationElimination_2 ) \flux(\conservedMomentDiscrete_1^{\indexTime - 1}) + (\coefficientExtrapolation_2 \coefficientCombinationElimination_0 + \coefficientCombinationElimination_3 - \coefficientCombinationElimination_1) \flux(\conservedMomentDiscrete_2^{ \indexTime - 1}) 
    + \sum_{\indexSpace = 3}^{\orderExtrapolationBiased} (\coefficientExtrapolation_{\indexSpace}\coefficientCombinationElimination_0 + \coefficientCombinationElimination_{\indexSpace + 1} - \coefficientCombinationElimination_{\indexSpace - 1}) \flux(\conservedMomentDiscrete_{\indexSpace}^{ \indexTime - 1})\Bigr ) \\
    + \latticeVelocity \coefficientCombinationElimination_0 \sourceTermBoundaryDiscrete_0^{\indexTime}.
\end{multline*}
using \eqref{eq:uBoundaryGenericExtrapolation} and \eqref{eq:bulkSchemeForU} at the previous time step.
As before, one can keep notations general by assuming that $\coefficientCombinationElimination_{\indexSpace} = 0$ for $\indexSpace \geq \orderExtrapolationBiased$.
We thus solve an over-determined linear system with $\orderExtrapolationBiased + 1$ equations on $\orderExtrapolationBiased$ unknowns $\coefficientCombinationElimination_0, \dots, \coefficientCombinationElimination_{\orderExtrapolationBiased - 1}$:
\begin{equation}\label{eq:systemNeededForExtrapolation}
    \begin{cases}
        \coefficientExtrapolation_0 \coefficientCombinationElimination_0 + \coefficientCombinationElimination_1 &= \coefficientExtrapolation_0^2 + \coefficientExtrapolation_1 - 1, \\
        (\coefficientExtrapolation_1 - 1)\coefficientCombinationElimination_0 + \coefficientCombinationElimination_2  &= \coefficientExtrapolation_0 (\coefficientExtrapolation_1+1) + \coefficientExtrapolation_2, \\
        \coefficientExtrapolation_2 \coefficientCombinationElimination_0 + \coefficientCombinationElimination_3 - \coefficientCombinationElimination_1 &= \coefficientExtrapolation_0 \coefficientExtrapolation_2 + \coefficientExtrapolation_1  - 1 + \coefficientExtrapolation_3,\\
        \coefficientExtrapolation_{\indexSpace}\coefficientCombinationElimination_0 + \coefficientCombinationElimination_{\indexSpace + 1} - \coefficientCombinationElimination_{\indexSpace - 1} &= \coefficientExtrapolation_0 \coefficientExtrapolation_{\indexSpace} + \coefficientExtrapolation_{\indexSpace + 1} + \coefficientExtrapolation_{\indexSpace - 1}, \qquad \indexSpace \in \integerInterval{3}{\orderExtrapolationBiased},
    \end{cases}
\end{equation}
where the fact that $\coefficientExtrapolation_{\indexSpace} = 0$ for $\indexSpace \geq \orderExtrapolation$ and $\coefficientCombinationElimination_{\indexSpace} = 0$ for $\indexSpace \geq \orderExtrapolationBiased$ is understood.
We now show that \eqref{eq:systemNeededForExtrapolation} admits a unique solution.
Summing all the equations in \eqref{eq:systemNeededForExtrapolation}, the left-hand side becomes zero, since $\sum_{\indexSpace = 0}^{\orderExtrapolation - 1}\coefficientExtrapolation_{\indexSpace} = 1$.
Concerning the right-hand side, it becomes
\begin{equation*}
    \coefficientExtrapolation_0^2 + \coefficientExtrapolation_1 - 1 + \coefficientExtrapolation_0 (\coefficientExtrapolation_1+1) + \coefficientExtrapolation_2 + \coefficientExtrapolation_0 \coefficientExtrapolation_2 + \coefficientExtrapolation_1  - 1 + \coefficientExtrapolation_3  + \sum_{\indexSpace = 3}^{\orderExtrapolationBiased} ( \coefficientExtrapolation_0 \coefficientExtrapolation_{\indexSpace} + \coefficientExtrapolation_{\indexSpace + 1} + \coefficientExtrapolation_{\indexSpace - 1})
    = \coefficientExtrapolation_0 \Bigl( \sum_{\indexSpace = 0}^{\orderExtrapolationBiased} \coefficientExtrapolation_{\indexSpace} +1 \Bigr) + 2\Bigl ( \sum_{\indexSpace = 1}^{\orderExtrapolationBiased} \coefficientExtrapolation_{\indexSpace}  - 1\Bigr ) = 0.
\end{equation*}
We can therefore remove the last equation and obtain a square system of the form $\matricial{A} \vectorial{\coefficientCombinationElimination} = \vectorial{b}$, where $\matricial{A} = \tilde{\matricial{A}} + \vectorial{\coefficientExtrapolation}\transpose{ \canonicalBasisVector{1}}$, with $\tilde{\matricial{A}}$ a tridiagonal Toeplitz matrix with zeros on the diagonal, $-1$ on the subdiagonal, and $1$ on the supradiagonal, and $\vectorial{\coefficientExtrapolation} = \transpose{(\coefficientExtrapolation_0, \dots, \coefficientExtrapolation_{\orderExtrapolation - 1})}$.
We show that $\matricial{A}$ has full rank to conclude.
Using the matrix-determinant lemma \cite[Chapter 0]{horn2012matrix}, we gain $\determinant (\matricial{A} ) = \determinant (\tilde{\matricial{A} }) + \transpose{ \canonicalBasisVector{1}} \adjugate (\tilde{\matricial{A}})\vectorial{\coefficientExtrapolation}$.
Straightforward computations deliver $\determinant (\tilde{\matricial{A}}) = \tfrac{1}{2} (1+(-1)^{\orderExtrapolationBiased})$ and $\transpose{ \canonicalBasisVector{1}} \adjugate (\tilde{\matricial{A}}) = \tfrac{1}{2}(1-(-1)^{\orderExtrapolationBiased}, -1-(-1)^{\orderExtrapolationBiased}, \dots)$, thus 
\begin{equation*}
    \determinant (\matricial{A}) = \tfrac{1}{2} \bigl  (1+(-1)^{\orderExtrapolationBiased} + \sum_{\indexSpace = 0}^{\orderExtrapolationBiased - 1} ((-1)^{\indexSpace} - (-1)^{\orderExtrapolationBiased})\coefficientExtrapolation_{\indexSpace} \bigr ) = \tfrac{1}{2} \Bigl  (1 + \sum_{\indexSpace = 0}^{\orderExtrapolationBiased - 1} \binom{\orderExtrapolationBiased}{\indexSpace} \Bigr ) = \tfrac{1}{2}(1 + 2^{\orderExtrapolationBiased} - 1) = 2^{\orderExtrapolationBiased - 1} \geq 1 > 0.
\end{equation*}
The unique solution $\coefficientCombinationElimination_0, \dots, \coefficientCombinationElimination_{\orderExtrapolationBiased - 1}$ is used and, equations from \eqref{eq:systemNeededForExtrapolation} yield
\begin{equation*}
    \sum_{\substack{\indexSpace = 0 \\ \indexSpace \neq 1}}^{\orderExtrapolation - 1} \coefficientExtrapolation_{\indexSpace} \nonConservedMomentDiscrete_{\indexSpace}^{\indexTime} + (\coefficientExtrapolation_1 - 1) \nonConservedMomentDiscrete_{1}^{\indexTime} = \latticeVelocity \sum_{\indexSpace = 0}^{\orderExtrapolationBiased - 1}\coefficientCombinationElimination_{\indexSpace}\conservedMomentDiscrete_{\indexSpace}^{\indexTime} -
    \latticeVelocity \Bigl ( (\coefficientExtrapolation_0 + \coefficientCombinationElimination_0 ) \conservedMomentDiscrete_{1}^{\indexTime - 1} + (\coefficientExtrapolation_1 + \coefficientCombinationElimination_1 - 1) \conservedMomentDiscrete_{2}^{\indexTime - 1} 
    + \sum_{\indexSpace = 3}^{\orderExtrapolation} (\coefficientExtrapolation_{\indexSpace - 1} + \coefficientCombinationElimination_{\indexSpace - 1}) \conservedMomentDiscrete_{\indexSpace}^{\indexTime - 1} \Bigr ) 
     + \latticeVelocity (\coefficientExtrapolation_0 - \coefficientCombinationElimination_0) \sourceTermBoundaryDiscrete_0^{\indexTime}.
\end{equation*}
Back into \eqref{eq:uBoundaryGenericExtrapolation} , we obtain
\begin{multline*}
    \conservedMomentDiscrete_{0}^{\indexTime +1} = \tfrac{1}{2} \Bigl ( \sum_{\substack{\indexSpace = 0 \\ \indexSpace \neq 1}}^{\orderExtrapolation - 1} (\coefficientExtrapolation_{\indexSpace} + (1-\relaxationParameter)\coefficientCombinationElimination_{\indexSpace})\conservedMomentDiscrete_{\indexSpace}^{\indexTime} + (\coefficientExtrapolation_1 + 1 + (1-\relaxationParameter)\coefficientCombinationElimination_1) \conservedMomentDiscrete_{1}^{\indexTime} \Bigr ) \\
    + \tfrac{\relaxationParameter - 1}{2} \Bigl ((\coefficientExtrapolation_0 + \coefficientCombinationElimination_0 ) \conservedMomentDiscrete_{1}^{\indexTime - 1} + (\coefficientExtrapolation_1 + \coefficientCombinationElimination_1 - 1) \conservedMomentDiscrete_{2}^{\indexTime - 1} 
    + \sum_{\indexSpace = 3}^{\orderExtrapolation} (\coefficientExtrapolation_{\indexSpace - 1} + \coefficientCombinationElimination_{\indexSpace - 1}) \conservedMomentDiscrete_{\indexSpace}^{\indexTime - 1}\Bigr ) \\
    + \tfrac{\relaxationParameter}{2\latticeVelocity} \Bigl ( \sum_{\substack{\indexSpace = 0 \\ \indexSpace \neq 1}}^{\orderExtrapolation - 1} \coefficientExtrapolation_{\indexSpace} \flux(\conservedMomentDiscrete_{\indexSpace}^{\indexTime}) + (\coefficientExtrapolation_1 - 1) \flux(\conservedMomentDiscrete_{1}^{\indexTime}) \Bigr )  + \sourceTermBoundaryDiscrete_0^{\indexTime + 1} + \tfrac{1-\relaxationParameter}{2} (\coefficientExtrapolation_0 -\coefficientCombinationElimination_0) \sourceTermBoundaryDiscrete_0^{\indexTime}.
\end{multline*}

By formally computing the coefficients for different $\orderExtrapolation$, we can prove that the solution of \eqref{eq:systemNeededForExtrapolation} for $\orderExtrapolation \geq 2$ is explicitly given by $\coefficientCombinationElimination_{0} = \catalanTwoIndices{\orderExtrapolation-2}{1}$, $ \coefficientCombinationElimination_1 = 1 - \catalanTwoIndices{\orderExtrapolation - 3}{2}$, $\coefficientCombinationElimination_{\indexSpace} = (-1)^{\indexSpace}\catalanTwoIndices{\orderExtrapolation-\indexSpace-2}{\indexSpace + 1}$ for $\indexSpace \in \integerInterval{2}{\lfloor \tfrac{\orderExtrapolation-1}{2}\rfloor - 1}$, and $\coefficientCombinationElimination_{\orderExtrapolation - \indexSpace} = (-1)^{\orderExtrapolation + 1 - \indexSpace} \catalanTwoIndices{\orderExtrapolation - \indexSpace}{\indexSpace - 1}$ for $\indexSpace \in \integerInterval{1}{\lfloor \tfrac{\orderExtrapolation-1}{2}\rfloor + 1 + \tfrac{1 + (-1)^{\orderExtrapolation}}{2}}$,
where $\catalanTwoIndices{n}{p}$ are the coefficients of the Catalan's triangle.
We observe that $\coefficientExtrapolation_0 - \coefficientCombinationElimination_0 = 2$, $\coefficientExtrapolation_0 + \coefficientCombinationElimination_0 = 2(\orderExtrapolation - 1)$, $\coefficientExtrapolation_1 + \coefficientCombinationElimination_1 - 1 = -(\orderExtrapolation-1)(\orderExtrapolation-2)$.
\end{proofWithoutProof}

\subsection{Proof of \Cref{thm:countingThroughReflection}}\label{proof:thm:countingThroughReflection}

\begin{proofWithoutProof}
    Let $\targetEigenvalue \in\complex\smallsetminus \{ 0\}$ be an isolated point of the asymptotic spectrum of $\schemeMatrixFD$ and $\epsilon > 0$.
The number of eigenvalues of $\schemeMatrixFD$ inside $\ball{\epsilon}{\targetEigenvalue}$, counted with their multiplicity, see \cite[Theorem 10.43]{rudin1987real}, is 
\begin{equation*}
    \frac{1}{2\pi i} \oint_{\partial \ball{\epsilon}{\targetEigenvalue} } \frac{\differential_{\timeShiftOperator}\determinant(\timeShiftOperator\identityMatrix{2\numberSpacePoints} - \schemeMatrixFD)}{\determinant(\timeShiftOperator\identityMatrix{2\numberSpacePoints} - \schemeMatrixFD)}\differential \timeShiftOperator = \frac{1}{2\pi i} \oint_{\partial \ball{\epsilon}{\targetEigenvalue} } \frac{\trace(\adjugate(\timeShiftOperator\identityMatrix{2\numberSpacePoints} - \schemeMatrixFD))}{\determinant(\timeShiftOperator\identityMatrix{2\numberSpacePoints} - \schemeMatrixFD)}\differential \timeShiftOperator = \frac{1}{2\pi i} \oint_{\partial \ball{\epsilon}{\targetEigenvalue} } \trace((\timeShiftOperator\identityMatrix{2\numberSpacePoints} - \schemeMatrixFD)^{-1})\differential \timeShiftOperator.
\end{equation*}
Using the Sherman-Morrison formula \cite[Chapter 0]{horn2012matrix} 
\begin{multline}\label{eq:generalContourIntegral}
    \oint_{\partial \ball{\epsilon}{\targetEigenvalue} } \trace((\timeShiftOperator\identityMatrix{2\numberSpacePoints} - \schemeMatrixFD)^{-1})\differential \timeShiftOperator 
    = \oint_{\partial \ball{\epsilon}{\targetEigenvalue} }  \trace((\timeShiftOperator\identityMatrix{2\numberSpacePoints} - \schemeMatrixFDToeplitz - \canonicalBasisVector{\numberSpacePoints}\transpose{\parturbationBoundaryInflow})^{-1}) \differential\timeShiftOperator\\
    + \oint_{\partial \ball{\epsilon}{\targetEigenvalue} } \frac{\trace((\timeShiftOperator\identityMatrix{2\numberSpacePoints} - \schemeMatrixFDToeplitz - \canonicalBasisVector{\numberSpacePoints}\transpose{\parturbationBoundaryInflow})^{-2}\canonicalBasisVector{1}\transpose{\parturbationBoundaryOutflow})}{1-\transpose{\parturbationBoundaryOutflow}(\timeShiftOperator\identityMatrix{2\numberSpacePoints} - \schemeMatrixFDToeplitz - \canonicalBasisVector{\numberSpacePoints}\transpose{\parturbationBoundaryInflow})^{-1}\canonicalBasisVector{1}}\differential\timeShiftOperator 
    =\oint_{\partial \ball{\epsilon}{\targetEigenvalue} } \frac{\trace((\timeShiftOperator\identityMatrix{2\numberSpacePoints} - \schemeMatrixFDToeplitz - \canonicalBasisVector{\numberSpacePoints}\transpose{\parturbationBoundaryInflow})^{-2}\canonicalBasisVector{1}\transpose{\parturbationBoundaryOutflow})}{1-\transpose{\parturbationBoundaryOutflow}(\timeShiftOperator\identityMatrix{2\numberSpacePoints} - \schemeMatrixFDToeplitz - \canonicalBasisVector{\numberSpacePoints}\transpose{\parturbationBoundaryInflow})^{-1}\canonicalBasisVector{1}}\differential\timeShiftOperator ,
\end{multline}
where the last equality comes from assuming that the considered $\numberSpacePoints$'s are large enough for the chosen $\epsilon$ so that no eigenvalue linked with $\schemeMatrixFDToeplitz$ and the inflow is in $\ball{\epsilon}{\targetEigenvalue}$.
In \eqref{eq:generalContourIntegral}, the only source of poles that would increment the count is in the denominator.
By the Sherman-Morrison formula:
\begin{multline}\label{eq:denominatorRemaining}
    1-\transpose{\parturbationBoundaryOutflow}(\timeShiftOperator\identityMatrix{2\numberSpacePoints} - \schemeMatrixFDToeplitz - \canonicalBasisVector{\numberSpacePoints}\transpose{\parturbationBoundaryInflow})^{-1}\canonicalBasisVector{1} \\
    = 1-\transpose{\parturbationBoundaryOutflow}(\timeShiftOperator\identityMatrix{2\numberSpacePoints} - \schemeMatrixFDToeplitz)^{-1}\canonicalBasisVector{1} \underbrace{-\frac{\transpose{\parturbationBoundaryOutflow}(\timeShiftOperator\identityMatrix{2\numberSpacePoints} - \schemeMatrixFDToeplitz)^{-1}\canonicalBasisVector{\numberSpacePoints}\transpose{\parturbationBoundaryInflow}(\timeShiftOperator\identityMatrix{2\numberSpacePoints} - \schemeMatrixFDToeplitz)^{-1}\canonicalBasisVector{1}}{1-\transpose{\parturbationBoundaryInflow}(\timeShiftOperator\identityMatrix{2\numberSpacePoints} - \schemeMatrixFDToeplitz)^{-1}\canonicalBasisVector{\numberSpacePoints}}}_{=:\text{coup}_{\text{in}}^{\text{out}}(\timeShiftOperator)}.
\end{multline}
We anticipate that---in \eqref{eq:denominatorRemaining}---the term $\text{coup}_{\text{in}}^{\text{out}}(\timeShiftOperator)$ is a coupling term between eigenvalues relative to opposite boundaries, and thus eventually vanishes as $\numberSpacePoints\to+\infty$.
Recalling that $\schemeMatrixFDToeplitzBlockMinusOne = (\relaxationParameter - 1)\identityMatrix{\numberSpacePoints}$ and using the formula for the inverse of a block matrix made up of four blocks, we gain\renewcommand{\arraystretch}{1.1}
\begin{equation}\label{eq:blockInverse}
    (\timeShiftOperator\identityMatrix{2\numberSpacePoints} - \schemeMatrixFDToeplitz)^{-1} = 
    \left [
        \begin{array}{c|c}
            (\modifiedTimeShiftOperator(\timeShiftOperator)\identityMatrix{\numberSpacePoints} - \schemeMatrixFDToeplitzBlockZero)^{-1} & (\modifiedTimeShiftOperator(\timeShiftOperator)\identityMatrix{\numberSpacePoints} - \schemeMatrixFDToeplitzBlockZero)^{-1}(\relaxationParameter-1)\timeShiftOperator^{-1} \\
            \hline
            \timeShiftOperator^{-1}(\modifiedTimeShiftOperator(\timeShiftOperator)\identityMatrix{\numberSpacePoints} - \schemeMatrixFDToeplitzBlockZero)^{-1} & \timeShiftOperator^{-1}\identityMatrix{\numberSpacePoints} + \timeShiftOperator^{-1} (\modifiedTimeShiftOperator(\timeShiftOperator)\identityMatrix{\numberSpacePoints} - \schemeMatrixFDToeplitzBlockZero)^{-1} (\relaxationParameter-1)\timeShiftOperator^{-1}
        \end{array} 
    \right ].
\end{equation}
The advantage of \eqref{eq:blockInverse} is that $\modifiedTimeShiftOperator(\timeShiftOperator)\identityMatrix{\numberSpacePoints} - \schemeMatrixFDToeplitzBlockZero$ is a tridiagonal Toeplitz matrix, for which several results are available.
Let $\indexColumn \in \integerInterval{1}{\numberSpacePoints}$. 
Equation \eqref{eq:blockInverse} entails
\begin{align*}
    \transpose{\parturbationBoundaryOutflow}(\timeShiftOperator\identityMatrix{2\numberSpacePoints} - \schemeMatrixFDToeplitz)^{-1}\canonicalBasisVector{\indexColumn} = &\sum_{\indexRow = 1}^{\numberCoefficientsOutFlowFDZero}\coefficientOutflowFDZero_{\indexRow - 1}\transpose{\canonicalBasisVector{\indexRow}} (\modifiedTimeShiftOperator(\timeShiftOperator)\identityMatrix{\numberSpacePoints} - \schemeMatrixFDToeplitzBlockZero)^{-1} \canonicalBasisVector{\indexColumn} - \schemeMatrixFDBlockZeroEntry_1 \transpose{\canonicalBasisVector{2}} (\modifiedTimeShiftOperator(\timeShiftOperator)\identityMatrix{\numberSpacePoints} - \schemeMatrixFDToeplitzBlockZero)^{-1} \canonicalBasisVector{\indexColumn}\\
    +\timeShiftOperator^{-1} \Bigl ( &\sum_{\indexRow = 1}^{\numberCoefficientsOutFlowFDMinusOne}\coefficientOutflowFDMinusOne_{\indexRow - 1}\transpose{\canonicalBasisVector{\indexRow}} (\modifiedTimeShiftOperator(\timeShiftOperator)\identityMatrix{\numberSpacePoints} - \schemeMatrixFDToeplitzBlockZero)^{-1} \canonicalBasisVector{\indexColumn} - \schemeMatrixFDBlockMinusOneEntry_0 \transpose{\canonicalBasisVector{1}} (\modifiedTimeShiftOperator(\timeShiftOperator)\identityMatrix{\numberSpacePoints} - \schemeMatrixFDToeplitzBlockZero)^{-1} \canonicalBasisVector{\indexColumn} \Bigr ), 
\end{align*}
and $\transpose{\parturbationBoundaryInflow}(\timeShiftOperator\identityMatrix{2\numberSpacePoints} - \schemeMatrixFDToeplitz)^{-1}\canonicalBasisVector{\indexColumn} = - \schemeMatrixFDBlockZeroEntry_{-1} \transpose{\canonicalBasisVector{\numberSpacePoints-1}} (\modifiedTimeShiftOperator(\timeShiftOperator)\identityMatrix{\numberSpacePoints} - \schemeMatrixFDToeplitzBlockZero)^{-1} \canonicalBasisVector{\indexColumn} - \timeShiftOperator^{-1} \schemeMatrixFDBlockMinusOneEntry_{0} \transpose{\canonicalBasisVector{\numberSpacePoints}} (\modifiedTimeShiftOperator(\timeShiftOperator)\identityMatrix{\numberSpacePoints} - \schemeMatrixFDToeplitzBlockZero)^{-1} \canonicalBasisVector{\indexColumn}$.

When either $\schemeMatrixFDBlockZeroEntry_{-1}$ or $\schemeMatrixFDBlockZeroEntry_{1}$ equals zero, $\modifiedTimeShiftOperator(\timeShiftOperator)\identityMatrix{\numberSpacePoints} - \schemeMatrixFDToeplitzBlockZero$ is lower or upper triangular, analogously to its inverse.
In these circumstances, explicit computations provide the following results.
\begin{itemize}
    \item If $\schemeMatrixFDBlockZeroEntry_{-1} = 0$, whence $\schemeMatrixFDBlockZeroEntry_1 = 2-\relaxationParameter \in [0, 1]$, we gain $\transpose{\canonicalBasisVector{\indexRow}} (\modifiedTimeShiftOperator(\timeShiftOperator)\identityMatrix{\numberSpacePoints} - \schemeMatrixFDToeplitzBlockZero)^{-1}\canonicalBasisVector{\indexColumn} = \schemeMatrixFDBlockZeroEntry_1^{\indexColumn - \indexRow} \modifiedTimeShiftOperator(\timeShiftOperator)^{\indexRow - \indexColumn - 1} \mathds{1}_{\indexColumn \geq \indexRow}$, and thus deduce 
    \begin{align*}
        \transpose{\parturbationBoundaryInflow}(\timeShiftOperator\identityMatrix{2\numberSpacePoints} - \schemeMatrixFDToeplitz)^{-1}\canonicalBasisVector{1} = 0 \qquad \text{and} \quad
        1 - \transpose{\parturbationBoundaryInflow}(\timeShiftOperator\identityMatrix{2\numberSpacePoints} - \schemeMatrixFDToeplitz)^{-1}\canonicalBasisVector{\numberSpacePoints} = 1 + \timeShiftOperator^{-1} \schemeMatrixFDBlockMinusOneEntry_{0} \modifiedTimeShiftOperator(\timeShiftOperator)^{-1} = \timeShiftOperator \modifiedTimeShiftOperator(\timeShiftOperator)^{-1}.
    \end{align*}
    The first equality above means that---even at finite $\numberSpacePoints$---the inflow does not impact the number of isolated eigenvalues linked with the outflow.
    Since we are integrating away from the origin, \eqref{eq:denominatorRemaining} becomes
    \begin{align*}
        1-\transpose{\parturbationBoundaryOutflow}(\timeShiftOperator\identityMatrix{2\numberSpacePoints} - \schemeMatrixFDToeplitz - \canonicalBasisVector{\numberSpacePoints}\transpose{\parturbationBoundaryInflow})^{-1}\canonicalBasisVector{1} = 1-\transpose{\parturbationBoundaryOutflow}(\timeShiftOperator\identityMatrix{2\numberSpacePoints} - \schemeMatrixFDToeplitz)^{-1}\canonicalBasisVector{1} 
        &= \timeShiftOperator^{-1}\modifiedTimeShiftOperator(\timeShiftOperator)^{-1}(\timeShiftOperator^2 -  \coefficientOutflowFDZero_{0}\timeShiftOperator - \coefficientOutflowFDMinusOne_{0}) \\
        &=\frac{\timeShiftOperator^2 -  \coefficientOutflowFDZero_{0}\timeShiftOperator - \coefficientOutflowFDMinusOne_{0}}{(\timeShiftOperator - i\sqrt{\relaxationParameter-1})(\timeShiftOperator + i\sqrt{\relaxationParameter-1})}.
    \end{align*}
    Back into \eqref{eq:generalContourIntegral}, we finally obtain:
    \begin{equation*}
        \oint_{\partial \ball{\epsilon}{\targetEigenvalue} } \trace((\timeShiftOperator\identityMatrix{2\numberSpacePoints} - \schemeMatrixFD)^{-1})\differential \timeShiftOperator 
        =\oint_{\partial \ball{\epsilon}{\targetEigenvalue} } \frac{(\timeShiftOperator - i\sqrt{\relaxationParameter-1})(\timeShiftOperator + i\sqrt{\relaxationParameter-1})\trace((\timeShiftOperator\identityMatrix{2\numberSpacePoints} - \schemeMatrixFDToeplitz - \canonicalBasisVector{\numberSpacePoints}\transpose{\parturbationBoundaryInflow})^{-2}\canonicalBasisVector{1}\transpose{\parturbationBoundaryOutflow})}{\timeShiftOperator^2 -  \coefficientOutflowFDZero_{0}\timeShiftOperator - \coefficientOutflowFDMinusOne_{0}}\differential\timeShiftOperator.
    \end{equation*}

    \item If $\schemeMatrixFDBlockZeroEntry_{1} = 0$, we gain $\transpose{\canonicalBasisVector{\indexRow}} (\modifiedTimeShiftOperator(\timeShiftOperator)\identityMatrix{\numberSpacePoints} - \schemeMatrixFDToeplitzBlockZero)^{-1}\canonicalBasisVector{\indexColumn} = \schemeMatrixFDBlockZeroEntry_{-1}^{\indexRow-\indexColumn}\modifiedTimeShiftOperator(\timeShiftOperator)^{ \indexColumn - \indexRow - 1} \mathds{1}_{\indexRow \geq \indexColumn}$, thus $\transpose{\parturbationBoundaryOutflow}(\timeShiftOperator\identityMatrix{2\numberSpacePoints} - \schemeMatrixFDToeplitz)^{-1}\canonicalBasisVector{\numberSpacePoints} = 0$ and $1 - \transpose{\parturbationBoundaryInflow}(\timeShiftOperator\identityMatrix{2\numberSpacePoints} - \schemeMatrixFDToeplitz)^{-1}\canonicalBasisVector{\numberSpacePoints} = 1 + \timeShiftOperator^{-1} \schemeMatrixFDBlockMinusOneEntry_{0} \modifiedTimeShiftOperator(\timeShiftOperator)^{-1} = \timeShiftOperator\modifiedTimeShiftOperator(\timeShiftOperator)^{-1}$.
    Eventually, \eqref{eq:denominatorRemaining} becomes
    \begin{align*}
        1-\transpose{\parturbationBoundaryOutflow}(\timeShiftOperator\identityMatrix{2\numberSpacePoints} - \schemeMatrixFDToeplitz - \canonicalBasisVector{\numberSpacePoints}\transpose{\parturbationBoundaryInflow})^{-1}\canonicalBasisVector{1} &= 1-\modifiedTimeShiftOperator(\timeShiftOperator)^{-1} \Bigl ( \sum_{\indexRow = 0}^{\numberCoefficientsOutFlowFDZero - 1} \coefficientOutflowFDZero_{\indexRow} (\schemeMatrixFDBlockZeroEntry_{-1} \modifiedTimeShiftOperator(\timeShiftOperator)^{-1} )^{\indexRow} + \timeShiftOperator^{-1} \sum_{\indexRow = 0}^{\numberCoefficientsOutFlowFDMinusOne- 1} \coefficientOutflowFDMinusOne_{\indexRow} (\schemeMatrixFDBlockZeroEntry_{-1} \modifiedTimeShiftOperator(\timeShiftOperator)^{-1} )^{\indexRow} - \timeShiftOperator^{-1} \schemeMatrixFDBlockMinusOneEntry_{0}\Bigr )\\
        &= \frac{\fourierShift(\timeShiftOperator)}{\schemeMatrixFDBlockZeroEntry_{-1}} \Bigl (\timeShiftOperator -\sum_{\indexRow = 0}^{\numberCoefficientsOutFlowFDZero - 1} \coefficientOutflowFDZero_{\indexRow}\fourierShift(\timeShiftOperator)^{\indexRow} - \timeShiftOperator^{-1} \sum_{\indexRow = 0}^{\numberCoefficientsOutFlowFDMinusOne- 1} \coefficientOutflowFDMinusOne_{\indexRow} \fourierShift(\timeShiftOperator)^{\indexRow} \Bigr ) + \underbrace{1 + \frac{\fourierShift(\timeShiftOperator)}{\schemeMatrixFDBlockZeroEntry_{-1}}  (\timeShiftOperator^{-1} \schemeMatrixFDBlockMinusOneEntry_{0} - \timeShiftOperator  )}_{=0},
    \end{align*}
    using that the bulk characteristic equation in $\fourierShift(\timeShiftOperator)$ reads $\modifiedTimeShiftOperator(\timeShiftOperator) = \schemeMatrixFDBlockZeroEntry_{-1}\fourierShift(\timeShiftOperator)^{-1}$.
    We finally obtain
    \begin{equation*}
        \oint_{\partial \ball{\epsilon}{\targetEigenvalue} } \trace((\timeShiftOperator\identityMatrix{2\numberSpacePoints} - \schemeMatrixFD)^{-1})\differential \timeShiftOperator = \oint_{\partial \ball{\epsilon}{\targetEigenvalue} } \frac{\schemeMatrixFDBlockZeroEntry_{-1} \trace((\timeShiftOperator\identityMatrix{2\numberSpacePoints} - \schemeMatrixFDToeplitz - \canonicalBasisVector{\numberSpacePoints}\transpose{\parturbationBoundaryInflow})^{-2}\canonicalBasisVector{1}\transpose{\parturbationBoundaryOutflow}) }{\fourierShift(\timeShiftOperator)\Bigl ( \timeShiftOperator -\sum_{\indexRow = 0}^{\numberCoefficientsOutFlowFDZero - 1} \coefficientOutflowFDZero_{\indexRow}\fourierShift(\timeShiftOperator)^{\indexRow} - \timeShiftOperator^{-1} \sum_{\indexRow = 0}^{\numberCoefficientsOutFlowFDMinusOne- 1} \coefficientOutflowFDMinusOne_{\indexRow} \fourierShift(\timeShiftOperator)^{\indexRow} \Bigr ) }\differential\timeShiftOperator.
    \end{equation*}
    We can forget about $\fourierShift(\timeShiftOperator)$ since it never vanishes close to $\targetEigenvalue$.
\end{itemize}

Let us now consider the case where both $\schemeMatrixFDBlockZeroEntry_{-1}, \schemeMatrixFDBlockZeroEntry_{1} \neq 0$.
In this case, generalizing \cite[Corollary 4.1]{DAFONSECA20017} to the case where $\schemeMatrixFDBlockZeroEntry_{-1}\schemeMatrixFDBlockZeroEntry_1$ is negative, we obtain, for $\indexRow \leq \indexColumn$
\begin{equation*}
    \transpose{\canonicalBasisVector{\indexRow}} (\modifiedTimeShiftOperator(\timeShiftOperator)\identityMatrix{\numberSpacePoints} - \schemeMatrixFDToeplitzBlockZero)^{-1} \canonicalBasisVector{\indexColumn} = \frac{\schemeMatrixFDBlockZeroEntry_1^{\indexColumn-\indexRow}}{(\sqrt{\schemeMatrixFDBlockZeroEntry_{-1}\schemeMatrixFDBlockZeroEntry_1})^{\indexColumn-\indexRow + 1}} \frac{\chebyshevPolynomialSecondKind{\indexRow-1}(\frac{\modifiedTimeShiftOperator(\timeShiftOperator) }{2\sqrt{\schemeMatrixFDBlockZeroEntry_{-1}\schemeMatrixFDBlockZeroEntry_1}}) \chebyshevPolynomialSecondKind{\numberSpacePoints-\indexColumn}(\frac{\modifiedTimeShiftOperator(\timeShiftOperator) }{2\sqrt{\schemeMatrixFDBlockZeroEntry_{-1}\schemeMatrixFDBlockZeroEntry_1}})}{\chebyshevPolynomialSecondKind{\numberSpacePoints}(\frac{\modifiedTimeShiftOperator(\timeShiftOperator) }{2\sqrt{\schemeMatrixFDBlockZeroEntry_{-1}\schemeMatrixFDBlockZeroEntry_1}})},
\end{equation*}
for $\modifiedTimeShiftOperator(\timeShiftOperator) \neq 0$.
Here $\chebyshevPolynomialSecondKind{\numberSpacePoints}$ is the Chebyshev polynomial of second kind of degree $\numberSpacePoints$.
Since $\chebyshevPolynomialSecondKind{\numberSpacePoints}(\timeShiftOperator) = ((\timeShiftOperator + \sqrt{\timeShiftOperator^2 - 1})^{\numberSpacePoints + 1} - (\timeShiftOperator - \sqrt{\timeShiftOperator^2 - 1})^{\numberSpacePoints + 1})/(2\sqrt{\timeShiftOperator^2 - 1})$, and we can easily see that 
\begin{align*}
    \{ \timeShiftOperator \in \complex ~:~ |\timeShiftOperator \pm \sqrt{\timeShiftOperator^2 - 1}|= 1 \} &= [-1, 1], \\
    \{ \timeShiftOperator \in \complex ~:~ |\timeShiftOperator \pm \sqrt{\timeShiftOperator^2 - 1}|>1 \} &= \{ \timeShiftOperator \in \complex ~:~ \mp\realPart{\timeShiftOperator}<0\} \cup \{ \timeShiftOperator \in i\reals ~:~ \mp\imagPart{\timeShiftOperator}<0\} \smallsetminus[-1, 1], \\
    \{ \timeShiftOperator \in \complex ~:~ |\timeShiftOperator \pm \sqrt{\timeShiftOperator^2 - 1}|<1 \} &= \{ \timeShiftOperator \in \complex ~:~ \mp\realPart{\timeShiftOperator}>0\} \cup \{ \timeShiftOperator \in i\reals ~:~ \mp\imagPart{\timeShiftOperator}>0\} \smallsetminus[-1, 1],
\end{align*}
we deduce that 
\begin{equation}\label{eq:limitCheby}
    \lim_{\numberSpacePoints\to+\infty} |\chebyshevPolynomialSecondKind{\numberSpacePoints}(\timeShiftOperator)|  = 
    \begin{cases}
        +\infty, \qquad &\timeShiftOperator \in \complex\smallsetminus (-1, 1), \\
        \text{does not exist}, \qquad &\timeShiftOperator \in (-1, 1). \\
    \end{cases}
\end{equation}
We would like to investigate the fact that we are allowed to consider the previous limit in the setting we are currently in---namely that of a contour integral on $\partial\ball{\epsilon}{\targetEigenvalue}$.
In particular, we would like to know if we are in the latter situation in \eqref{eq:limitCheby} only on a non-negligible set.
Let $\frequency \in \reals$, then
\begin{equation*}
    \frac{\modifiedTimeShiftOperator(\epsilon e^{i\frequency}) }{2\sqrt{\schemeMatrixFDBlockZeroEntry_{-1}\schemeMatrixFDBlockZeroEntry_1}} = \frac{\epsilon + \frac{1-\relaxationParameter}{\epsilon}}{2\sqrt{\schemeMatrixFDBlockZeroEntry_{-1}\schemeMatrixFDBlockZeroEntry_1}}\cos(\frequency) + i\frac{\epsilon + \frac{\relaxationParameter - 1}{\epsilon}}{2\sqrt{\schemeMatrixFDBlockZeroEntry_{-1}\schemeMatrixFDBlockZeroEntry_1}}\sin(\frequency).
\end{equation*}
If $\schemeMatrixFDBlockZeroEntry_{-1}\schemeMatrixFDBlockZeroEntry_1 > 0$, hence we would like to avoid $0 < \epsilon^2 = 1-\relaxationParameter$, which never happens if $\relaxationParameter \geq 1$. When $\relaxationParameter < 1$, we can always avoid this situation by taking $\epsilon$ small enough.
Otherwise, when $\schemeMatrixFDBlockZeroEntry_{-1}\schemeMatrixFDBlockZeroEntry_1 < 0$, the same way of reasoning yields the same conclusions.
We can therefore discard the case where we integrate along $(-1, 1)$ (where the zeros of any Chebyshev polynomial dwell) on a non-negligible set and thus safely consider limits of Chebyshev polynomials of second kind when their degree tends to infinity.
Let us consider the terms in \eqref{eq:denominatorRemaining}.
\begin{multline*}
    \transpose{\parturbationBoundaryOutflow}(\timeShiftOperator\identityMatrix{2\numberSpacePoints} - \schemeMatrixFDToeplitz)^{-1}\canonicalBasisVector{\numberSpacePoints}
    =\Bigl ( \frac{\schemeMatrixFDBlockZeroEntry_{1}}{\sqrt{\schemeMatrixFDBlockZeroEntry_{-1}\schemeMatrixFDBlockZeroEntry_1}} \Bigr )^{\numberSpacePoints}  \frac{1}{\chebyshevPolynomialSecondKind{\numberSpacePoints}(\frac{\modifiedTimeShiftOperator(\timeShiftOperator) }{2\sqrt{\schemeMatrixFDBlockZeroEntry_{-1}\schemeMatrixFDBlockZeroEntry_1}})} \Biggl [ \sum_{\indexRow = 1}^{\numberCoefficientsOutFlowFDZero}\coefficientOutflowFDZero_{\indexRow - 1} \frac{(\sqrt{\schemeMatrixFDBlockZeroEntry_{-1}\schemeMatrixFDBlockZeroEntry_1})^{\indexRow-1}}{\schemeMatrixFDBlockZeroEntry_1^{\indexRow}} \chebyshevPolynomialSecondKind{\indexRow-1}(\frac{\modifiedTimeShiftOperator(\timeShiftOperator) }{2\sqrt{\schemeMatrixFDBlockZeroEntry_{-1}\schemeMatrixFDBlockZeroEntry_1}}) \\
    - \frac{\sqrt{\schemeMatrixFDBlockZeroEntry_{-1}\schemeMatrixFDBlockZeroEntry_1}}{\schemeMatrixFDBlockZeroEntry_1} \chebyshevPolynomialSecondKind{1}(\frac{\modifiedTimeShiftOperator(\timeShiftOperator) }{2\sqrt{\schemeMatrixFDBlockZeroEntry_{-1}\schemeMatrixFDBlockZeroEntry_1}}) 
    + \timeShiftOperator^{-1}\Bigl ( \sum_{\indexRow = 1}^{\numberCoefficientsOutFlowFDMinusOne}\coefficientOutflowFDMinusOne_{\indexRow - 1} \frac{(\sqrt{\schemeMatrixFDBlockZeroEntry_{-1}\schemeMatrixFDBlockZeroEntry_1})^{\indexRow-1}}{\schemeMatrixFDBlockZeroEntry_1^{\indexRow}} \chebyshevPolynomialSecondKind{\indexRow-1}(\frac{\modifiedTimeShiftOperator(\timeShiftOperator) }{2\sqrt{\schemeMatrixFDBlockZeroEntry_{-1}\schemeMatrixFDBlockZeroEntry_1}}) - \frac{\schemeMatrixFDBlockMinusOneEntry_0}{\schemeMatrixFDBlockZeroEntry_1}\Bigr ) \Biggr ],
\end{multline*}
\begin{align*}
    \transpose{\parturbationBoundaryInflow}(\timeShiftOperator\identityMatrix{2\numberSpacePoints} - \schemeMatrixFDToeplitz)^{-1}\canonicalBasisVector{1} 
    &=-\Bigl ( \frac{\schemeMatrixFDBlockZeroEntry_{-1}}{\sqrt{\schemeMatrixFDBlockZeroEntry_{-1}\schemeMatrixFDBlockZeroEntry_1}} \Bigr )^{\numberSpacePoints} \frac{1}{\chebyshevPolynomialSecondKind{\numberSpacePoints}(\frac{\modifiedTimeShiftOperator(\timeShiftOperator) }{2\sqrt{\schemeMatrixFDBlockZeroEntry_{-1}\schemeMatrixFDBlockZeroEntry_1}})} \Bigl ( \frac{\sqrt{\schemeMatrixFDBlockZeroEntry_{-1}\schemeMatrixFDBlockZeroEntry_1}}{\schemeMatrixFDBlockZeroEntry_{-1}} + \timeShiftOperator^{-1}\frac{\schemeMatrixFDBlockMinusOneEntry_{0}}{\schemeMatrixFDBlockZeroEntry_{-1}}  \Bigr ),
\end{align*}
and finally
\begin{equation*}
    \transpose{\parturbationBoundaryInflow}(\timeShiftOperator\identityMatrix{2\numberSpacePoints} - \schemeMatrixFDToeplitz)^{-1}\canonicalBasisVector{\numberSpacePoints} =-\frac{\schemeMatrixFDBlockZeroEntry_{-1}\schemeMatrixFDBlockZeroEntry_{1}}{(\sqrt{\schemeMatrixFDBlockZeroEntry_{-1}\schemeMatrixFDBlockZeroEntry_1})^2} \frac{\chebyshevPolynomialSecondKind{\numberSpacePoints-2}(\frac{\modifiedTimeShiftOperator(\timeShiftOperator) }{2\sqrt{\schemeMatrixFDBlockZeroEntry_{-1}\schemeMatrixFDBlockZeroEntry_1}})}{\chebyshevPolynomialSecondKind{\numberSpacePoints}(\frac{\modifiedTimeShiftOperator(\timeShiftOperator) }{2\sqrt{\schemeMatrixFDBlockZeroEntry_{-1}\schemeMatrixFDBlockZeroEntry_1}})} - \timeShiftOperator^{-1}\frac{\schemeMatrixFDBlockMinusOneEntry_0}{\sqrt{\schemeMatrixFDBlockZeroEntry_{-1}\schemeMatrixFDBlockZeroEntry_1}} \frac{\chebyshevPolynomialSecondKind{\numberSpacePoints-1}(\frac{\modifiedTimeShiftOperator(\timeShiftOperator) }{2\sqrt{\schemeMatrixFDBlockZeroEntry_{-1}\schemeMatrixFDBlockZeroEntry_1}})}{\chebyshevPolynomialSecondKind{\numberSpacePoints}(\frac{\modifiedTimeShiftOperator(\timeShiftOperator) }{2\sqrt{\schemeMatrixFDBlockZeroEntry_{-1}\schemeMatrixFDBlockZeroEntry_1}})}.
\end{equation*}

Putting all together, we obtain
\begin{equation}\label{eq:generalCaseDen}
    \lim_{\numberSpacePoints\to+\infty} \text{coup}_{\text{in}}^{\text{out}}(\timeShiftOperator) 
    = \lim_{\numberSpacePoints\to+\infty} \frac{\Bigl ( \frac{\schemeMatrixFDBlockZeroEntry_{-1} \schemeMatrixFDBlockZeroEntry_{1}}{\sqrt{\schemeMatrixFDBlockZeroEntry_{-1}\schemeMatrixFDBlockZeroEntry_1}} \Bigr )^{\numberSpacePoints}  \frac{1}{\chebyshevPolynomialSecondKind{\numberSpacePoints}(\frac{\modifiedTimeShiftOperator(\timeShiftOperator) }{2\sqrt{\schemeMatrixFDBlockZeroEntry_{-1}\schemeMatrixFDBlockZeroEntry_1}})^2} [\cdots]}{1+\frac{\schemeMatrixFDBlockZeroEntry_{-1}\schemeMatrixFDBlockZeroEntry_{1}}{(\sqrt{\schemeMatrixFDBlockZeroEntry_{-1}\schemeMatrixFDBlockZeroEntry_1})^2} \frac{\chebyshevPolynomialSecondKind{\numberSpacePoints-2}(\frac{\modifiedTimeShiftOperator(\timeShiftOperator) }{2\sqrt{\schemeMatrixFDBlockZeroEntry_{-1}\schemeMatrixFDBlockZeroEntry_1}})}{\chebyshevPolynomialSecondKind{\numberSpacePoints}(\frac{\modifiedTimeShiftOperator(\timeShiftOperator) }{2\sqrt{\schemeMatrixFDBlockZeroEntry_{-1}\schemeMatrixFDBlockZeroEntry_1}})} + \timeShiftOperator^{-1}\frac{\schemeMatrixFDBlockMinusOneEntry_0}{\sqrt{\schemeMatrixFDBlockZeroEntry_{-1}\schemeMatrixFDBlockZeroEntry_1}} \frac{\chebyshevPolynomialSecondKind{\numberSpacePoints-1}(\frac{\modifiedTimeShiftOperator(\timeShiftOperator) }{2\sqrt{\schemeMatrixFDBlockZeroEntry_{-1}\schemeMatrixFDBlockZeroEntry_1}})}{\chebyshevPolynomialSecondKind{\numberSpacePoints}(\frac{\modifiedTimeShiftOperator(\timeShiftOperator) }{2\sqrt{\schemeMatrixFDBlockZeroEntry_{-1}\schemeMatrixFDBlockZeroEntry_1}})}},
\end{equation}
where the quantity $[\cdots]$ does not depend on $\numberSpacePoints$.
The numerator in the overall fraction tends to zero, for the Chebyshev polynomial tends to infinity with its degree $\numberSpacePoints\to+\infty$ and we have $ |{\schemeMatrixFDBlockZeroEntry_{-1} \schemeMatrixFDBlockZeroEntry_{1}}/{\sqrt{\schemeMatrixFDBlockZeroEntry_{-1}\schemeMatrixFDBlockZeroEntry_1}}  | = \sqrt{|\schemeMatrixFDBlockZeroEntry_{-1} | |\schemeMatrixFDBlockZeroEntry_{1}|}\leq 1$ under the stability conditions by \Cref{prop:stabilityConditionsPeriodic}.
We are left to check the denominator in \eqref{eq:generalCaseDen}.
One can show that, for $\indexRow \geq 1$ finite, we have 
\begin{equation*}
    \lim_{\numberSpacePoints\to+\infty}\frac{\chebyshevPolynomialSecondKind{\numberSpacePoints - \indexRow}(\timeShiftOperator)}{\chebyshevPolynomialSecondKind{\numberSpacePoints}(\timeShiftOperator)} = 
    \begin{cases}
        (\timeShiftOperator + \sqrt{\timeShiftOperator^2 - 1})^{\indexRow}, \qquad &\text{when}\quad \timeShiftOperator \in \{ \timeShiftOperator \in \complex ~:~ \realPart{\timeShiftOperator}<0\} \cup \{ \timeShiftOperator \in i\reals ~:~ \imagPart{\timeShiftOperator}<0\} \smallsetminus(-1, 1),  \\
        (\timeShiftOperator - \sqrt{\timeShiftOperator^2 - 1})^{\indexRow}, \qquad &\text{when}\quad \timeShiftOperator \in \{ \timeShiftOperator \in \complex ~:~ \realPart{\timeShiftOperator}>0\} \cup \{ \timeShiftOperator \in i\reals ~:~ \imagPart{\timeShiftOperator}>0\} \smallsetminus(-1, 1), \\
        \text{does not exist}, \qquad &\text{when}\quad \timeShiftOperator \in (-1, 1).
    \end{cases}
\end{equation*}
This expression leads us thinking that the limit function is not continuous at $\realPart{\timeShiftOperator} = 0$.
This is not what happens and the way of separating different formul\ae{} is just a matter of notational convenience, for we indeed have $\lim_{\numberSpacePoints\to+\infty}\frac{\chebyshevPolynomialSecondKind{\numberSpacePoints - \indexRow}(\timeShiftOperator)}{\chebyshevPolynomialSecondKind{\numberSpacePoints}(\timeShiftOperator)} = \overline{\fourierShift}(\timeShiftOperator)^{\indexRow}$ for $\timeShiftOperator \in \partial\ball{\epsilon}{\targetEigenvalue} \smallsetminus (-1, 1)$ where $\overline{\fourierShift}(\timeShiftOperator)$ is the root of $\overline{\fourierShift}(\timeShiftOperator)^2 - 2\timeShiftOperator\overline{\fourierShift}(\timeShiftOperator) + 1 = 0$ such that it exists $\tilde{\timeShiftOperator} \in \partial\ball{\epsilon}{\targetEigenvalue} \smallsetminus (-1, 1)$ characterized by the fact that $\lim_{\numberSpacePoints\to+\infty}\frac{\chebyshevPolynomialSecondKind{\numberSpacePoints - \indexRow}(\tilde{\timeShiftOperator})}{\chebyshevPolynomialSecondKind{\numberSpacePoints}(\tilde{\timeShiftOperator})} = \overline{\fourierShift}(\tilde{\timeShiftOperator})^{\indexRow}$.
This procedure is made possible since the roots of the polynomial equation depend continuously on $\timeShiftOperator$.
Therefore, the denominator in \eqref{eq:generalCaseDen} reads
\begin{multline*}
    \lim_{\numberSpacePoints\to+\infty} 1+\frac{\schemeMatrixFDBlockZeroEntry_{-1}\schemeMatrixFDBlockZeroEntry_{1}}{(\sqrt{\schemeMatrixFDBlockZeroEntry_{-1}\schemeMatrixFDBlockZeroEntry_1})^2} \frac{\chebyshevPolynomialSecondKind{\numberSpacePoints-2}(\frac{\modifiedTimeShiftOperator(\timeShiftOperator) }{2\sqrt{\schemeMatrixFDBlockZeroEntry_{-1}\schemeMatrixFDBlockZeroEntry_1}})}{\chebyshevPolynomialSecondKind{\numberSpacePoints}(\frac{\modifiedTimeShiftOperator(\timeShiftOperator) }{2\sqrt{\schemeMatrixFDBlockZeroEntry_{-1}\schemeMatrixFDBlockZeroEntry_1}})} + \timeShiftOperator^{-1}\frac{\schemeMatrixFDBlockMinusOneEntry_0}{\sqrt{\schemeMatrixFDBlockZeroEntry_{-1}\schemeMatrixFDBlockZeroEntry_1}} \frac{\chebyshevPolynomialSecondKind{\numberSpacePoints-1}(\frac{\modifiedTimeShiftOperator(\timeShiftOperator) }{2\sqrt{\schemeMatrixFDBlockZeroEntry_{-1}\schemeMatrixFDBlockZeroEntry_1}})}{\chebyshevPolynomialSecondKind{\numberSpacePoints}(\frac{\modifiedTimeShiftOperator(\timeShiftOperator) }{2\sqrt{\schemeMatrixFDBlockZeroEntry_{-1}\schemeMatrixFDBlockZeroEntry_1}})} \\
    =1+\frac{\schemeMatrixFDBlockZeroEntry_{1}}{\schemeMatrixFDBlockZeroEntry_{-1}} \Biggl (  {\frac{\modifiedTimeShiftOperator(\timeShiftOperator) }{2\schemeMatrixFDBlockZeroEntry_1} \pm \frac{1}{2}\sqrt{\Bigl (\frac{\modifiedTimeShiftOperator(\timeShiftOperator) }{\schemeMatrixFDBlockZeroEntry_1}  \Bigr )^2 - 4\frac{\schemeMatrixFDBlockZeroEntry_{-1}}{\schemeMatrixFDBlockZeroEntry_{1}}}} \Biggr )^2 + \timeShiftOperator^{-1}\frac{\schemeMatrixFDBlockMinusOneEntry_0}{\schemeMatrixFDBlockZeroEntry_{-1}} \Biggl (  {\frac{\modifiedTimeShiftOperator(\timeShiftOperator) }{2\schemeMatrixFDBlockZeroEntry_1} \pm \frac{1}{2}\sqrt{\Bigl (\frac{\modifiedTimeShiftOperator(\timeShiftOperator) }{\schemeMatrixFDBlockZeroEntry_1}  \Bigr )^2 - 4\frac{\schemeMatrixFDBlockZeroEntry_{-1}}{\schemeMatrixFDBlockZeroEntry_{1}}}} \Biggr )\\
    =1+ \frac{\fourierShift(\timeShiftOperator)}{\schemeMatrixFDBlockZeroEntry_{-1} } (\timeShiftOperator^{-1}\schemeMatrixFDBlockMinusOneEntry_0 + \schemeMatrixFDBlockZeroEntry_{1} \fourierShift(\timeShiftOperator)) 
    = \frac{\schemeMatrixFDBlockZeroEntry_{-1}}{\timeShiftOperator\fourierShift(\timeShiftOperator)},
\end{multline*}
where $\fourierShift(\timeShiftOperator)$ is the solution of \eqref{eq:bulkCharEquation} such that it exist (by all the previous arguments) $\tilde{\timeShiftOperator} \in \partial\ball{\epsilon}{\targetEigenvalue}$ so that 
            \begin{equation*}
                \lim_{\numberSpacePoints\to+\infty}\frac{1}{\sqrt{\schemeMatrixFDBlockZeroEntry_{-1}\schemeMatrixFDBlockZeroEntry_1}}\frac{\chebyshevPolynomialSecondKind{\numberSpacePoints-1}(\frac{\modifiedTimeShiftOperator(\tilde{\timeShiftOperator})}{2\sqrt{\schemeMatrixFDBlockZeroEntry_{-1}\schemeMatrixFDBlockZeroEntry_1}})}{\chebyshevPolynomialSecondKind{\numberSpacePoints}(\frac{\modifiedTimeShiftOperator(\tilde{\timeShiftOperator})}{2\sqrt{\schemeMatrixFDBlockZeroEntry_{-1}\schemeMatrixFDBlockZeroEntry_1}})} = \frac{{\fourierShift}(\tilde{\timeShiftOperator})}{\schemeMatrixFDBlockZeroEntry_{-1}}.
            \end{equation*}
This gives $\lim_{\numberSpacePoints\to+\infty} \text{coup}_{\text{in}}^{\text{out}}(\timeShiftOperator) = 0$.
Eventually, we are left with 
\begin{align*}
    1-\transpose{\parturbationBoundaryOutflow}(\timeShiftOperator\identityMatrix{2\numberSpacePoints} - \schemeMatrixFDToeplitz)^{-1}\canonicalBasisVector{1} =
    1 - &\sum_{\indexRow = 1}^{\numberCoefficientsOutFlowFDZero}\coefficientOutflowFDZero_{\indexRow - 1}\frac{\schemeMatrixFDBlockZeroEntry_{-1}^{\indexRow - 1}}{(\sqrt{\schemeMatrixFDBlockZeroEntry_{-1}\schemeMatrixFDBlockZeroEntry_1})^{\indexRow}} \frac{\chebyshevPolynomialSecondKind{\numberSpacePoints-\indexRow}(\frac{\modifiedTimeShiftOperator(\timeShiftOperator) }{2\sqrt{\schemeMatrixFDBlockZeroEntry_{-1}\schemeMatrixFDBlockZeroEntry_1}})}{\chebyshevPolynomialSecondKind{\numberSpacePoints}(\frac{\modifiedTimeShiftOperator(\timeShiftOperator) }{2\sqrt{\schemeMatrixFDBlockZeroEntry_{-1}\schemeMatrixFDBlockZeroEntry_1}})} +  \frac{\schemeMatrixFDBlockZeroEntry_{-1}\schemeMatrixFDBlockZeroEntry_1}{(\sqrt{\schemeMatrixFDBlockZeroEntry_{-1}\schemeMatrixFDBlockZeroEntry_1})^{2}} \frac{\chebyshevPolynomialSecondKind{\numberSpacePoints-2}(\frac{\modifiedTimeShiftOperator(\timeShiftOperator) }{2\sqrt{\schemeMatrixFDBlockZeroEntry_{-1}\schemeMatrixFDBlockZeroEntry_1}})}{\chebyshevPolynomialSecondKind{\numberSpacePoints}(\frac{\modifiedTimeShiftOperator(\timeShiftOperator) }{2\sqrt{\schemeMatrixFDBlockZeroEntry_{-1}\schemeMatrixFDBlockZeroEntry_1}})}\\
    -\timeShiftOperator^{-1} \Bigl ( &\sum_{\indexRow = 1}^{\numberCoefficientsOutFlowFDMinusOne}\coefficientOutflowFDMinusOne_{\indexRow - 1}\frac{\schemeMatrixFDBlockZeroEntry_{-1}^{\indexRow - 1}}{(\sqrt{\schemeMatrixFDBlockZeroEntry_{-1}\schemeMatrixFDBlockZeroEntry_1})^{\indexRow}} \frac{\chebyshevPolynomialSecondKind{\numberSpacePoints-\indexRow}(\frac{\modifiedTimeShiftOperator(\timeShiftOperator) }{2\sqrt{\schemeMatrixFDBlockZeroEntry_{-1}\schemeMatrixFDBlockZeroEntry_1}})}{\chebyshevPolynomialSecondKind{\numberSpacePoints}(\frac{\modifiedTimeShiftOperator(\timeShiftOperator) }{2\sqrt{\schemeMatrixFDBlockZeroEntry_{-1}\schemeMatrixFDBlockZeroEntry_1}})} - \schemeMatrixFDBlockMinusOneEntry_0 \frac{1}{\sqrt{\schemeMatrixFDBlockZeroEntry_{-1}\schemeMatrixFDBlockZeroEntry_1}} \frac{\chebyshevPolynomialSecondKind{\numberSpacePoints-1}(\frac{\modifiedTimeShiftOperator(\timeShiftOperator) }{2\sqrt{\schemeMatrixFDBlockZeroEntry_{-1}\schemeMatrixFDBlockZeroEntry_1}})}{\chebyshevPolynomialSecondKind{\numberSpacePoints}(\frac{\modifiedTimeShiftOperator(\timeShiftOperator) }{2\sqrt{\schemeMatrixFDBlockZeroEntry_{-1}\schemeMatrixFDBlockZeroEntry_1}})} \Bigr ),
\end{align*}
thus 
\begin{multline*}
    \lim_{\numberSpacePoints\to+\infty}1-\transpose{\parturbationBoundaryOutflow}(\timeShiftOperator\identityMatrix{2\numberSpacePoints} - \schemeMatrixFDToeplitz)^{-1}\canonicalBasisVector{1} =
    1 - \sum_{\indexRow = 1}^{\numberCoefficientsOutFlowFDZero}\coefficientOutflowFDZero_{\indexRow - 1}\frac{\fourierShift(\timeShiftOperator)^{\indexRow}}{\schemeMatrixFDBlockZeroEntry_{-1}} +  \frac{\schemeMatrixFDBlockZeroEntry_{1} \fourierShift(\timeShiftOperator)^{2}}{\schemeMatrixFDBlockZeroEntry_{-1}} 
    -\timeShiftOperator^{-1} \Bigl ( \sum_{\indexRow = 1}^{\numberCoefficientsOutFlowFDMinusOne}\coefficientOutflowFDMinusOne_{\indexRow - 1} \frac{\fourierShift(\timeShiftOperator)^{\indexRow}}{\schemeMatrixFDBlockZeroEntry_{-1}} - \schemeMatrixFDBlockMinusOneEntry_0 \frac{\fourierShift(\timeShiftOperator)}{\schemeMatrixFDBlockZeroEntry_{-1}} \Bigr ) \\
    = \frac{\fourierShift(\timeShiftOperator)}{\schemeMatrixFDBlockZeroEntry_{-1}} \Bigl  ( \timeShiftOperator - \sum_{\indexRow = 0}^{\numberCoefficientsOutFlowFDZero - 1}\coefficientOutflowFDZero_{\indexRow } \fourierShift(\timeShiftOperator)^{\indexRow} - \timeShiftOperator^{-1} \sum_{\indexRow = 0}^{\numberCoefficientsOutFlowFDMinusOne - 1}\coefficientOutflowFDMinusOne_{\indexRow} \fourierShift(\timeShiftOperator)^{\indexRow} \Bigr ) + \underbrace{1 + \frac{\fourierShift(\timeShiftOperator)}{\schemeMatrixFDBlockZeroEntry_{-1}} (\schemeMatrixFDBlockZeroEntry_{1} \fourierShift(\timeShiftOperator) + \timeShiftOperator^{-1}\schemeMatrixFDBlockMinusOneEntry_0 - \timeShiftOperator)}_{=0} \\
    = \frac{\fourierShift(\timeShiftOperator)}{\schemeMatrixFDBlockZeroEntry_{-1}} \Bigl  ( \timeShiftOperator - \sum_{\indexRow = 0}^{\numberCoefficientsOutFlowFDZero - 1}\coefficientOutflowFDZero_{\indexRow } \fourierShift(\timeShiftOperator)^{\indexRow} - \timeShiftOperator^{-1} \sum_{\indexRow = 0}^{\numberCoefficientsOutFlowFDMinusOne - 1}\coefficientOutflowFDMinusOne_{\indexRow} \fourierShift(\timeShiftOperator)^{\indexRow} \Bigr ).
\end{multline*}
This finally yields the desired result.
\end{proofWithoutProof}

\section{Expression of scheme matrix for the original \lbm{} scheme}\label{app:exprMatrixLBM}

The blocks in $\schemeMatrixLBM$ are made up by 
    \begin{align*}
        \schemeMatrixLBMBlock{++} &= \tfrac{1}{2}(2-\relaxationParameter+\relaxationParameter\courantNumber)
        \begin{pmatrix}
            \coefficientExtrapolation_0  & \coefficientExtrapolation_1 & \cdots & \cdots \\
            1 & 0 & &\\
            & \ddots & \ddots & \\
            & & 1 & 0
        \end{pmatrix}, \qquad 
        &&\schemeMatrixLBMBlock{+-}  = \tfrac{\relaxationParameter}{2}(1-\courantNumber)
        \begin{pmatrix}
            \coefficientExtrapolation_0  & \coefficientExtrapolation_1 & \cdots & \cdots \\
            1 & 0 & &\\
            & \ddots & \ddots & \\
            & & 1 & 0
        \end{pmatrix}, \\
        \schemeMatrixLBMBlock{-+} &= \tfrac{\relaxationParameter}{2}(1+\courantNumber)
        \begin{pmatrix}
            0 & 1 & & \\
            & \ddots & \ddots & \\
            & & 0 & 1\\
            & & -1& 0
        \end{pmatrix}, \qquad
        &&\schemeMatrixLBMBlock{--} = \tfrac{1}{2}(2-\relaxationParameter-\relaxationParameter\courantNumber)
        \begin{pmatrix}
            0 & 1 & & \\
            & \ddots & \ddots & \\
            & & 0 & 1\\
            & & -1& 0
        \end{pmatrix},
    \end{align*}
    in the case where we use \eqref{eq:extrapolationBoundaryCondition}.

\end{document}